\documentclass[11pt, reqno]{amsart}

\usepackage[utf8]{inputenc}

\usepackage[letterpaper,hmargin=1.0in,vmargin=1.0in]{geometry} 
\usepackage{amsmath,amsthm,amsfonts,amssymb,mathrsfs,stmaryrd,bigints}
\usepackage{dsfont}
\usepackage{mathtools}
\usepackage[usenames]{color}
\usepackage{float}
\usepackage[colorlinks=true,linkcolor=blue]{hyperref}
\usepackage{tikz}
\usepackage{subcaption}
\usepackage{epstopdf} 
\usepackage{amssymb}
\usepackage{setspace}
\usepackage{enumerate}
\usepackage{bigstrut}
\usepackage{multirow}
\usepackage{mathtools,xparse}
\usepackage{mathrsfs}
\usepackage{hyperref}
\usepackage{xcolor}
\usepackage [autostyle, english = american]{csquotes}
\usepackage{pgfplots}
\newtheorem*{lemma*}{Lemma}
\newtheorem{theorem}{Theorem}[section]
\newtheorem{corollary}[theorem]{Corollary}
\newtheorem{lemma}[theorem]{Lemma}
\newtheorem{proposition}[theorem]{Proposition}
\theoremstyle{definition}

\newtheorem{remark}[theorem]{Remark}
\newtheorem{assumption}{Assumption}

\DeclareMathOperator*{\argmax}{arg\,max}
\allowdisplaybreaks
\usepackage{chngcntr}

\def\it{\textit}

\newcommand{\cov}{\text{Cov}}
\newcommand{\var}{\text{Var}}
\newcommand{\pu}{\phi_{N,u}}
\newcommand{\pv}{\phi_{N,v}}

\newcommand{\ri}{\rightarrow \infty}

\newcommand{\E}{{\mathbb{E}}}
\newcommand{\1}{\mathds{1}}

\newcommand{\Z}{\mathbb{Z}}

\newcommand{\R}{\mathbb{R}}
\newcommand{\norm}[1]{\left\lVert#1\right\rVert}

\newcommand{\e}{\varepsilon}

\newcommand{\ka}{\kappa}

	\renewcommand{\P}{\mathbb{P}}

\newcommand{\cB}{\mathcal{B}}

\newtheorem{maintheorem}{Theorem}

\renewcommand{\emptyset}{\varnothing}

\renewcommand{\setminus}{\backslash}

\def\ba{\begin{align}}
\def\ea{\end{align}}
\def\bs{\begin{split}}
\def\es{\end{split}}

\begin{document}

\title[Level sets and Poisson-Dirichlet law for log-correlated Gaussian fields]{Universality of Poisson-Dirichlet law for log-correlated Gaussian fields via level set statistics } 
\author{Shirshendu Ganguly and Kyeongsik Nam}
 
\begin{abstract}  
Many low temperature disordered systems are expected to exhibit Poisson-Dirichlet (PD) statistics. In this paper, we focus on the case when the underlying disorder is a logarithmically correlated Gaussian process $\phi_N$ on the box $[-N,N]^d\subset \Z^d$. Canonical examples include branching random walk, $*$-scale invariant fields, with the central example being the two dimensional Gaussian free field (GFF), a universal scaling limit of a wide range of statistical mechanics models.

The corresponding Gibbs measure obtained by exponentiating $\beta$ (inverse temperature) times $\phi_N$  is a discrete version of the Gaussian multiplicative chaos (GMC) famously constructed by Kahane \cite{k}. In the low temperature or supercritical regime, i.e., $\beta$ larger than a critical $\beta_c,$ the GMC  is expected to exhibit atomic behavior on suitable renormalization, dictated by the extremal statistics or near maximum values of $\phi_N$.  Moreover, it is predicted  going back to a conjecture made in 2001 in \cite{physics}, that the weights  of this atomic GMC has a PD distribution. 

In a series of works culminating in \cite{bl2}, Biskup and Louidor carried out a comprehensive study of the near maxima of the 2D GFF, and established  the conjectured PD behavior throughout the super-critical regime ($\beta> 2$). 
In another direction, in \cite{drz}, Ding, Roy and Zeitouni  established universal behavior of the maximum for  a general class of log-correlated Gaussian fields.

In this paper we continue this program simply under the assumption of log-correlation and nothing further. We prove that the GMC concentrates on an $O(1)$ neighborhood of the local extrema and the PD  prediction made in \cite{physics} holds, in any dimension $d$, throughout the supercritical regime $\beta > \sqrt{2d}$, significantly generalizing past results. 
While many of the arguments for the GFF make use of the powerful Gibbs-Markov property, in absence of any Markovian structure for general Gaussian fields, we develop and use as our key input a sharp estimate of the size of level sets, a result we believe could have other applications.

\end{abstract}

\address{Department of Statistics, Evans Hall, University of California, Berkeley, CA
94720, USA} 

\email{sganguly@berkeley.edu}

\address{Department of Mathematical Sciences, KAIST, South Korea} 

\email{ksnam@kaist.ac.kr}

\maketitle

\tableofcontents

\section{Introduction}
 The
\it{Gaussian multiplicative chaos} (GMC), introduced in the seminal work of Kahane \cite{k}, is a canonical multi-fractal  random measure arising in various physical situations   such as {two-dimensional Liouville quantum gravity and three-dimensional turbulence theory}. Formally speaking, given a Gaussian field $X$, the Gaussian multiplicative chaos   is a random measure given by
\begin{align} \label{gmc}
M^\beta(A) = \int_A e^{\beta X(x) - \frac{\beta^2}{2}\mathbb{E}[X^2(x)]}dx, \quad A\subseteq \R^d \ \text{Borel},
\end{align}
where   $\beta>0$ denotes the inverse temperature. Often, in applications, the process 
$X$ is highly irregular and cannot be made sense of pointwise as a function and lives in the space of  {distributions} which makes the above expression ill-defined. However, there are by now, albeit highly non-trivial, procedures where one first mollifies $X$ by a suitable smoothing to rigorously define a smoother version of \eqref{gmc} and then passes to the limit (by introducing appropriately chosen renormalization terms) as the mollification becomes progressively singular. Of particular interest are Gaussian processes arising as fluctuation theories in various statistical physics models.  It turns out that in many such cases the covariance structure exhibits a logarithmic dependence property. To make things more concrete, let us consider discrete Gaussian processes $\{X_N\}_{N\geq 1}$ where $X_N$ is defined on the vertex set $V_N,$ formed by the box of length $N$ in $\Z^d$, or interpolated in a natural way to be defined on the counterpart box in $\R^d.$ Then $\{X_N\}_{N\geq 1}$ is said to exhibit the log-correlation property if for $x,y \in V_N$
\begin{align} \label{log correlation}
\cov(X_N(x),X_N(y))\approx \log \frac{N}{ |x-y|}.
\end{align}
While perhaps the canonical log-correlated Gaussian process is the two dimensional Gaussian free field (2D GFF) conjectured to be the universal scaling limit of many critical statistical physics models, other simpler ones include  hierarchical versions known as branching Brownian motion (BBM) or its discrete counterpart (BRW); both very well studied interacting particle systems. More details about the above and other log-correlated fields including $*$-scale invariant fields are included later.

For a large class of Gaussian processes  $\{X_N\}_{N\geq 1}$  satisfying \eqref{log correlation}, the limiting behavior of  the random measures
\begin{align}  \label{gmc2}
M_N^\beta (A) = \int_A e^{\beta X_N(x) - \frac{\beta^2}{2}\mathbb{E}[X_N^2(x)]} dx
\end{align}
(where the integral is interpreted as a sum over the vertices in the discrete case)
is predicted to exhibit a phase transition at $\beta_c=\sqrt{2d}$.
In the high temperature regime (or sub-critical)  $\beta<\beta_c$,  the random measures ought to
converge almost surely  to a non-trivial random measure $M^\beta$  as $N\rightarrow \infty$. Whereas in the super-critical case, when $\beta> \beta_c$, the limiting measure should degenerate.

Nonetheless, in the latter case, an alternate renormalization is still expected to lead to a \it{purely atomic} random measure as a limit exhibiting freezing. In fact,  such ``low temperature glassy" behavior was indeed established in \cite{mrv} for a certain class of exactly solvable Gaussian fields admitting an explicit expression for the covariance (e.g., $*$-scale invariant fields and the two-dimensional continuous Gaussian free field), where it was shown that  there exists a normalizing constant $c_N$ and a purely atomic random measure $K_{\sqrt{2d}/\beta}(dx)$  such that
\begin{align} \label{low}
c_N^{\frac{3\beta}{2\sqrt{2d}}}e ^{c_N (\frac{\beta}{\sqrt{2}}-\sqrt{d})^2} M_N^\beta(dx)  \xrightarrow[N\ri  ]{\textup{law}}C(\beta) K_{\sqrt{2d}/\beta}(dx).
\end{align}

Further, it is also expected that in the supercritical regime, the weight distribution of the limiting atomic measure exhibits Poisson-Dirichlet (PD) statistics (the precise definition appears later \eqref{pd def}). This has appeared as conjectures in the physics literature going back a couple of decades (see \cite{physics} and the comprehensive survey by Rhodes-Vargas \cite{rv}).  Subsequently, versions of the conjecture have been established in a handful of special cases such as log-normal Mandelbrot's cascades \cite{bknsw,w}, branching random walks \cite{brv}, and a specific one-dimensional model similar to the logarithmic Random Energy Model \cite{az}.  {The arguments of the last article employ techniques from spin-glass theory and were extended to the 2D GFF in the follow up article \cite{az2}.}  More recently, a very precise picture about the freezing and clustering phenomenon exhibited by the atomic GMC for the 2D GFF was obtained by \cite{bl2} (we elaborate more on this later in Section \ref{section 1.3}).

The above significant advances provides the necessary footing to be able to undertake an inquiry into the validity of the above results or more broadly the universality features of low temperature glasses, for general log-correlated Gaussian processes. The obvious barrier  such a venture faces is the absence of any exact solvability or probabilistic features such as the underlying Markovian structure exhibited by the GFF or BRW. Nonetheless, in the breakthrough work \cite{drz}, a systematic study of the maxima of general log-correlated fields was carried out establishing a rather sharp understanding under a general axiomatic set up only involving the covariance structure. More recently, the arguments were extended to Gaussian fields with logarithmic correlation up to certain local defects, such as the GFF on a supercritical percolation cluster in $\Z^2$, in \cite{sz1}.

Drawing inspiration from the above set of works, in this paper, we continue and extend this program only under the assumption of log-correlation without any further requirement on the covariance. We prove that the GMC concentrates on an $O(1)$ neighborhood of the local extrema and the PD  prediction holds, in any dimension $d$, throughout the supercritical regime $\beta > \sqrt{2d}$. Thus, this confirms the conjecture going back to \cite{physics} (also reiterated later in \cite{conjecture, rv}) simply under an assumption of logarithmic correlation, significantly generalizing past works.

We now move on to the precise statement of our results.

\subsection{Main result}

Let $V_N = \{0,1,\cdots,N-1\}^d$ be the discrete box of size $N$ in $\mathbb{Z}^d$. We
consider   a centered Gaussian field $\phi_N = (\pv)_{v\in V_N}$. We now state the only two assumptions on $\phi_N$ taken from \cite{drz} that we will be working under. 
The first assumption states that $\phi_N$ is globally  logarithmically bounded and provides a lower bound on the covariance.  The second crucial assumption states the log-correlation. Throughout the paper,  we denote by $|\cdot |$ the {$\ell^2$-Euclidean distance}. 

\begin{assumption} \label{a1}
There exists a constant $\alpha_0>0$ such that for all $u,v\in V_N$,
\begin{align}  \label{a0}
\var \pv \leq \log N + \alpha_0
\end{align}
and
\begin{align} \label{a}
\mathbb{E}(\pu-\pv)^2 \leq 2\log_+|u-v| - |\var \pu - \var \pv| + 4\alpha_0,
\end{align}
where $\log_+x = \log x \vee 0$.
\end{assumption}
 
Another equivalent formulation of the last display is 
\begin{align} \label{aa}
\cov(\pu,\pv) \geq \max\{\var \pu,\var \pv\} - \log_+|u-v| - 2\alpha_0.
\end{align}

The next crucial assumption states the log-correlation of the field. For $\delta>0$, define $V_N^\delta:= \{v\in V_N: d_\infty(v,\partial V_N) \geq \delta N\}$, where  $d_\infty(v,\partial V_N) := \min_{u\notin V_N} |v-u|_\infty$.
\begin{assumption} \label{a2}
For any $\delta>0$, there exists a constant $\alpha(\delta)>0$ such that for any $u,v\in V_N^\delta$,
\begin{align*}
|\cov(\phi_{N,u},\phi_{N,v}) - (\log N-\log_+ |u-v|)| < \alpha(\delta).
\end{align*}
\end{assumption}
 
It is well known that the Assumptions \ref{a1} and \ref{a2}  hold for the 2D GFF and $*$-scale invariant  fields.  However, the expert reader might observe that the  assumptions simply ensure a logarithmic correlation structure and nothing beyond that. As a consequence, by themselves, they are not enough to guarantee convergence of, say, the law of the maximum on appropriate centering. We will expand on this point further after the statement of our main results.

As already alluded to, the key observable we study is the atomic GMC. Towards this we define the pre-limiting Gibbs measure $\mu^\beta_N$ associated with the log-correlated Gaussian field $\phi_N$: For the inverse temperature $\beta>0$,
\begin{align} \label{lqg}
\mu_N^\beta (\{v\})  = \frac{1}{Z} e^{\beta \pv},\quad v\in V_N,
\end{align}
where  $Z$ is the partition function making $\mu_N^{\beta}$ a probability measure.

Recall that the regime of interest for us is the super-critical one, i.e., $\beta>\beta_c  =\sqrt{2d}$ (why this is indeed the location of criticality will be explained in Section \ref{section 1.3}) and that our main result pertains to establishing the   
 PD behavior of $\mu_N^\beta (\{v\})$  throughout the whole low temperature regime.
 We start by defining the  PD law precisely next, denoted by $\textup{PD}(s)$ for $s\in (0,1).$ This is a probability distribution on ordered partitions of unity, i.e. ordered sequences of non-negative numbers which sum to one,
 \begin{align} \label{pd def}
 \{ \{p_i\}_{i\geq 1}:p_1\geq p_2\geq \cdots \geq 0, \sum_{i\in \mathbb{N}} p_i =  1\}.
 \end{align}
The law is described by the following (one of the many known) sampling procedure. Sample points from the Poisson process on $[0,\infty)$ with intensity measure $x^{-1-s}dx$, and then normalize them by their sum and order the values in the decreasing order.
  
Further, for any sequence $a=(a_i)_{i\in \mathbb{N}}$ with $a_i>0$ for all $i$, denote by $\mathsf{Ord}(a),$ the reordering of $a$ in a non-increasing order, when it can  be defined which is what we will implicitly assume whenever using this notation. Finally for a finite vector $a,$ we define $\mathsf{Ord}(a)$ by ordering $a$ and then appending with zeros to pass to an infinite vector.
  
With the above preparation we can now state our main results.
 \begin{maintheorem}\label{theorem 1.1}
 Fix $\beta>\beta_c = \sqrt{2d}$. Let $\{\phi_N\}_{N\ge 1} $ be a sequence of $d$-dimensional centered Gaussian fields satisfying Assumptions \ref{a1} and \ref{a2}. Let $\{r_N\}_{N\geq 1}$ be any sequence such that $r_N\ri$ and $r_N/N\rightarrow 0$ as $N\rightarrow \infty$. Then,  there exists a collection of disjoint $r_N$-balls, say $\{B_1, B_2,\ldots, B_{m}\}$ (where $m$ is random) in $V_N$ such that 
 \begin{align} \label{statement}
\mathsf{Ord} (\{\mu_N^\beta(B_i) \}_{1\le i\le m} ) \xrightarrow[N\ri  ]{\textup{law}}  \textup{PD}(\beta_c/\beta)
 \end{align}
with respect to the $\ell^1$-distance between sequences.
 {This in particular implies that  for any $\e>0$, there exists $k \in \mathbb{N}$  (depending only on  $\e$) such that with probability at least $1-\e,$ there exists disjoint balls $\{B_i\}_{1\le i\le k}$ of radius $r_N,$ such that $\mu_N^\beta ( \cup_{i=1}^kB_i) > 1-\e$ for sufficiently large $N$.}
 \end{maintheorem}

Theorem \ref{theorem 1.1} thus exhibits the universality of the PD law for the atomic behavior or the glassy phase property of the low temperature GMC for logarithmically correlated Gaussian processes,  throughout the whole super-critical regime, establishing the conjecture from \cite{physics, conjecture, rv}. 
This generalizes considerably the results of \cite{bl2} and \cite{mrv} who established the same phenomenon for 2D GFF and $*$- scale invariant fields respectively.  A point worth reiterating is that the above theorem guarantees PD behavior simply under Assumptions \ref{a1}, \ref{a2} which by themselves do not ensure any fine convergence of the covariance kernel after appropriate centering and scaling.

Further, as already mentioned, techniques from spin-glasses were first applied in \cite{az}  in the context of a specific one-dimensional log-correlated Gaussian field and then eventually extended to the 2D GFF in \cite{az2}. In these articles, the ``overlap'' distribution (using spin-glass terminology) is shown to have a PD limit. The methods seem robust and potentially could be adapted to study more general log-correlated fields in the spirit of this paper.
However, the arguments being somewhat analytical in nature do not provide any further information about the associated GMC. 

In contrast, as elaborated below, our proof in fact allows us to extract quite a sharp understanding of the Gibbs measure itself as well as the geometry of the underlying extremal process, resembling the results of \cite{bl1} for the 2D GFF. 
To accomplish this, a key technical ingredient of the paper is to develop a sharp understanding of the size of level sets. This leads to the following estimate, which we expect to be more widely applicable and state as our second main result. 
 \begin{maintheorem} \label{theorem level} Let Assumptions \ref{a1} and \ref{a2} be in force,
{and $\e,\iota>0$ be any constants. Then, for  any large enough (fixed) $t$, for sufficiently large $N$,}
\begin{align*}
  \mathbb{P}(e^{(\sqrt{2d} -\e) t} \leq  |\Gamma_N(t) | \leq  e^{(\sqrt{2d}+\e) t})  \geq 1-\iota.
\end{align*}
where $\Gamma_{N}(t) := \{v\in V_N: \pv \geq m_N-t\}.$
\end{maintheorem}
In fact for the proof of Theorem \ref{theorem 1.1}, we need and establish an exponential bound for the upper tail (see Theorem \ref{theorem 3.0} in Section \ref{section 3}). Note however that our lower tail bound is  considerably weaker. 
While optimal bounds for the sizes of level sets in BBM were established in \cite{oren} who even pinned down the polynomial pre-factor, perhaps somewhat surprisingly, an explicit statement for the particular case of the 2D GFF about the lower bound seems to be missing from the literature, though the upper bound statement was established in \cite{bl1,bl2}.
It is quite plausible that a form of a truncated second moment argument could be employed to get a much sharper version of the theorem in the GFF case and thereby extend the results of \cite{oren}. In any case, achieving sharper extensions of the above theorem in various settings of interest remains an important open problem.


A few further remarks are in order.
\begin{itemize} 
\item [1.] As already indicated, it might, at first glance, seem surprising that the above theorem holds simply under Assumptions \ref{a1} and \ref{a2}, since it can be shown that this is \emph{not enough} for the law of the maximum to converge, a point we deem worth emphasizing. In fact the convergence of the law of the maximum was established in \cite{drz} only under additional finer assumptions ensuring the covariance structure converges both in a local sense around the diagonal as well as in a macroscopic sense off-diagonally. In our setting, the maximum will have sub-sequential limits owing to tightness. However it turns out that the effect of the subsequence affects the entire extremal process of the near maxima in a common way which disappears when one considers the Gibbs measure $\mu^{\beta}_N$, owing to the normalization by the partition function $Z$ leading to the same PD limit regardless of the subsequence.

\item [2.]
The proof of Theorem \ref{theorem 1.1} (see Section \ref{section 1.3} for further elaboration) will further show that the low-temperature Gibbs measure is essentially concentrated on the local extreme points as was previously established for 2D GFF in \cite{bl2}. However in our setting, like the maximum, the extremal process is not guaranteed to converge, but will only do so under further assumptions on the covariance as described in the previous remark wherein the arguments of this paper can be used to establish a Cox process limit of the local extreme points and their values. Remark \ref{spatial} expands on this.
\item [3.] It is also worth remarking that the set up of the paper as well as Assumptions \ref{a1}, \ref{a2} do not cover continuous space processes such as mollifications of the continuous 2D GFF (analyzed in \cite{acosta}), or fields with local defects (treated in \cite{sz1}). While we expect our arguments to be extendable to such cases as well, we don't pursue this in this paper.
\end{itemize}

\subsection{Previous works} \label{section 1.2}
Our aim to establish universality features of low temperature GMC for a large class of log-correlated Gaussian fields is strongly inspired by and often relies crucially on several past works investigating extremal statistics as well as the low temperature glassy behavior for fields possessing certain special structures. In this section we provide a brief review of such models and some of the important advances made in this regard.  
The interested reader is also encouraged to refer to \cite{rv,biskupsurvey,zeitounisurvey} for a comprehensive introduction to the topic. 
 
 \subsubsection{Branching Brownian motion and branching random walk}
 
Branching Brownian motion (BBM) and  branching random walk (BRW) are perhaps the most naturally arising (physically speaking)  log-correlated fields with a clearly embedded tree structure leading to logarithmic behavior. {BBM is an interacting particle system where the particles perform independent Brownian motions and split into two particles independently at rate one (BRW is the discrete time analog).
} 
 
The   distribution function of the right most (maximum) particle of BBM can be described in terms of  the Kolmogorov-Petrovsky-Piscounov or
 Fisher (F-KKP) equation \cite{kol}:
 \begin{align} \label{kpp}
 u_t = \frac{1}{2}u_{xx} + u^2-u.
 \end{align}
In fact, the solution $u(t,x)$ to \eqref{kpp} with  an initial data $u(0,x) = \1_{x\geq 0}$ satisfies that \cite{mc} 
\begin{align*}
u(t,x) = \mathbb{P} (\max_{i\in I(t)} X_i(t)\leq x).
\end{align*}
where $I(t)$ is the set of particles at time $t$ and $X_i(t)$ the location of the $i^{th}$ particle.
Also, denoting 
 \begin{align*}
 m(t):=\sqrt{2}t-\frac{3}{2\sqrt{2}}\log t,
 \end{align*}
the asymptotic distribution of  $\max_{i\in I(t)} X_i(t)-m(t)$ as $t\ri$, can be  expressed in terms of  a random shift of the Gumbel distribution \cite{ls}.  Furthermore, a precise behavior of the extreme points was described in terms of a decorated Poisson process  \cite{abbs,abk3,abk2,abk} analog of which will feature in our analysis as well. 
 
 We refer to Section \ref{section 2.4} for a precise definition of BRW which will play an important role in several of our arguments (As the reader will notice from the definition, BRW does not quite satisfy the log-correlation property of Assumption \ref{a2}  at all points owing to integer effects). Similarly as for BBM, the behavior of the rightmost (maximum) particle for BRW has also been studied extensively (see \cite{ar,a,bk,hs,m2}).  Further, PD statistics for the overlap function was established in \cite{aukosh}, relying on ideas similar to \cite{az} borrowing techniques from spin-glass theory.

 \subsubsection{Two-dimensional discrete Gaussian free field} 
The 2D GFF on $\R^2$ (and its natural discretization on $\Z^2$) is perhaps the most canonical two dimensional Gaussian process expected to exhibit the fluctuation theory for a large class of statistical mechanics models. A characteristic feature exhibited by the  2D GFF is an invariance property under conformal maps making it the natural candidate for the scaling limits of various two dimensional critical models ({see e.g. \cite{ss} where the contour lines of the 2D GFF are studied}).
 
 The law of the discrete 2D GFF $\phi^V = (\phi^V_v)_{v\in V}$  on a domain $V \subset \mathbb{Z}^2$  is given by 
\begin{align*}
 \mathbb{P}(\phi^V \in A) =  \frac{1}{Z_V} \int_A  \exp\Big(-\frac{1}{8} \sum_{u,v\in V, u\sim v} (\phi^V_u - \phi^V_v)^2\Big) \prod_{v\in V} d \phi^V_v \prod_{v\notin V} \delta_0 (d\phi^V_v),
\end{align*} 
where $Z_V$ is a normalizing constant and $u\sim v$ denotes that  $u$ and $v$ are adjacent.

Two important properties have facilitated a refined understanding of this central object. The first is a random walk interpretation for the covariance. That is, denoting by $G_V$ the Green function of the 2D simple random walk on $V\subseteq \mathbb{Z}^2$,
\begin{align*}
\cov(\phi_u^V,\phi_v^V) = G_V(u,v).
\end{align*}
 
The second is a spatial 
 \it{Gibbs-Markov property}, which provides an elegant description of the conditional description of field on a subdomain given its complement.    
This can be viewed as the analog of the tree structure in BRW.  The last two decades has seen remarkable progress, beautiful ideas, many of which we make use of as well, leading to an extremely precise understanding of the extremal statistics, low temperature freezing phenomenon and a version of our main result about PD behavior for the 2D GFF (see \cite{bl1,bl2} and the comprehensive survey \cite{biskupsurvey}). A weaker version of such a result focusing only on the PD structure was established in the prior work \cite{az2}.

 \subsubsection{$*$-scale invariant fields} These are another class of Gaussian processes which admit an exact covariance representation.
The precise description involves a $C^1$ function $k:\R^d \rightarrow \R$  satisfying $k(0)=1$ and $k(x)=0$ for $x\notin [-1,1]^d$. For $t\geq 0$ and $x\in \R^d$, define
 \begin{align} \label{k}
 K_t(x) := \int_1^{e^t}\frac{k(xu)}{u}du,
 \end{align}
and consider a family of centered (space-time) Gaussian process $(X_t(x))_{x\in \R^d,t\geq 0}$ with covariance given by
 \begin{align*}
 \cov (X_t(x),X_s(y)) = K_{t\wedge s}(y-x).
 \end{align*}
Note that for each $t\geq 0$,  
$$
 K_t(x) = \int_1^{1/|x|_\infty} \frac{k(xu)}{u}du   \asymp \log  (1/|x|_\infty), \quad e^{-t} \leq |x|_\infty  \leq 1
$$
(for $x = (x_1,\cdots,x_d)\in \R^d$, $|x|_\infty := \max \{|x_1|,\cdots,|x_d|\}$),
leading to log-correlation. Further, $X_t$ has independent increments, allowing one to perform exact computations. This has permitted a detailed study of the maximum \cite{m} as well as the associated GMC in both the subcritical and supercritical regimes in \cite{mrv}.

 \subsubsection{Random energy model (REM)}
In this simple case, when the Gaussian process simply consists of i.i.d. Gaussians with variance $n,$ indexed by say $\{-1,1\}^n$, it is not difficult to show that for  $\beta>\sqrt{2\log 2}$, the ordered masses of the associated GMC converges  weakly to the PD process with a parameter $\sqrt{2\log 2}/\beta$, i.e., Theorem \ref{theorem 1.1} holds in this setting. A one dimensional model known as the logarithmic REM was analyzed precisely in \cite{az}, whose arguments were shown to work for the 2D GFF in \cite{az2}. Finally such arguments were extended to the case of the Riemann zeta function in \cite{ouimet}. Further, for high dimensional models arising from spin-glass theory, in \cite{sz2}, the extremal process of critical points was shown to behave similar to the REM. \\
 
\subsubsection{Non-Gaussian processes} 
More recently, impressive progress has also been witnessed for random fields, which while not Gaussian do have an embedded tree structure leading to log-correlation. We mention three important examples in this regard. The first and perhaps the most studied one is the log determinant of a random unitary matrix producing a random field on the unit circle often called in literature as the CUE field (and C$\beta$E for general ensembles). A remarkable analogy with random matrices predicts similar behavior in the case of the Riemann zeta function \cite{fk,fhk}. A third related class of models arise from the determinant of random permutation matrices. There has been a series of works investigating universality of the behavior of the maximum of the CUE field (and more generally C$\beta$E field), see e.g., \cite{pz0, cmn}, leading to the recent preprint \cite{pz} where the convergence of the maximum along with a description of the  energy landscape around the extrema  were finally established. For the counterpart progress on the Riemann zeta side, we refer the reader to \cite{arguin1, arguin2, zeta}.
For random permutation matrices, a law of large numbers result was obtained in \cite{cz}. \\
 
We finish this section with the discussion of the central ideas in the proof of Theorem \ref{theorem 1.1} and the other results we obtain along the way. However, we first set up the notational apparatus for the paper. 
 
\subsection{Notations}
 For $u,v\in \mathbb{Z}^d $, $|u-v|:=|u-v|_2$ and $|u-v|_\infty$ denote the  {$\ell^2$ and $\ell^\infty$}-distances, respectively. Note that $|u-v|_\infty \leq   |u-v| \leq \sqrt{d}|u-v|_\infty$. For $D,D'\subseteq \mathbb{Z}^d$, let $d_\infty(D,D') = \min_{u\in D,v\in D'} |u-v|_\infty$. For $v\in \mathbb{Z}^d$ and   $r>0$, we denote by $B(v,r)$ and $B_\infty(v,r)$ the  collections of $u\in \mathbb{Z}^d $ such that $|u-v|\leq r$ and $|u-v|_\infty \leq r$, respectively.    We say that $B$ is a box of size $k$ if $B = z + \{0,1,\cdots,k-1\}^d$ for some $z\in \mathbb{Z}^d$. Throughout the paper, for any $A\subseteq \mathbb{Z}^d$ and a positive integer $k$, we define $kA := \{x \in \mathbb{Z}^d: x/k \in A\}.$

Further,  $\mathcal{N}(\mu, \sigma^2)$ will denote the Gaussian distribution with mean $\mu$ and variance $\sigma^2$. Also, the letter $C$ denotes a positive constant, whose value may change from line to line in the proofs.  {The positive constants which are fixed once and throughout the paper are denoted by the lower case $c$ with a subscript, for instance $c_4$}. Finally, $f \asymp g$ will mean that $\frac{1}{C'} f \leq g \leq C'f$ for {some  universal constant $C'>0$.}

 \subsection{Idea of proof} \label{section 1.3} While many important ideas have appeared in the past works on BRW, and the 2D GFF, the papers \cite{bl2} and \cite{drz} are of particular importance for us. In \cite{bl2} a much stronger version of Theorem \ref{theorem 1.1} was proven for the 2D GFF. Whereas in \cite{drz} a universality result for the maximum was established in the setting of the present paper of general log-correlated fields.  
 
The powerful Gibbs-Markov property is often the central input in many of the arguments in \cite{bl2}, which established an extremely refined understanding of the atomic GMC in the 2D GFF case. However, the key underpinning of our work is driven by the observation that one can get by, somewhat barely, even in the absence of any such Markovian structure, provided there is a sharp control on the size of the super-level sets, i.e., set of points with values bigger than a threshold (for brevity henceforth we will be terming them simply level sets) as stated in Theorem \ref{theorem level}. Thus we seek to accomplish this. This leads us to  \cite{drz} who used  Slepian's lemma to compare the maximum of a general log-correlated $\phi_N$ to that of more tractable BRW type processes. In this vein, the second main observation for us is that, albeit a less direct and significantly more elaborate comparison framework can indeed be set up to obtain the sought control on level sets. It seems worth reiterating at this point the remark following the statement of Theorem \ref{theorem 1.1},  that  our comparison framework does not rely on any further fine information on the covariance and is rather robust and works even in settings where, say, the maximum of the field is not guaranteed to converge.

We now move on to a more detailed sketch of our proof, in particular expanding on how to set up such a comparison framework to prove Theorem \ref{theorem level} and indeed how to then, simply relying on a control on the level sets, carry out a program resembling that in \cite{bl1}, to prove Theorem \ref{theorem 1.1}.

The first key step in proving Theorem \ref{theorem 1.1} is to rigorously establish the already alluded to \emph{freezing phenomenon},  in the entire super-critical regime, i.e.,  the Gibbs measure \eqref{lqg} concentrates on  points with values close to
 \begin{align*}
 m_N:=  \sqrt{2d}\log N -\frac{3}{2\sqrt{2d}}\log \log N.
 \end{align*}
As was proven in \cite{drz}, this is precisely where the maximum of $\phi_N$ (to be denoted $\phi_{N,*}:=\max_{v\in V_N} \pv$) concentrates (see Section \ref{section 2.2} for details). Known apriori upper tail estimate of $\{\phi_{N,*}  - m_N\}_{N\geq 1}$ established in \cite{drz} and recorded in Section \ref{section 2.2} implies that no value $\pv$ for $v\in V_N$ is bigger than $m_N+t$ for some large but fixed $t>0$ with probability going to one as $t\to \infty.$ Therefore it remains to bound the contributions from $v\in V_N$ with  $\pv < m_N-t$ for  large $t>0$.

This naturally leads to  Theorem \ref{theorem level}. 

\noindent
\textbf{Sharp bound on level sets}. Momentarily focusing only on the upper bound since that is the direction of relevance, we establish 
\begin{align} \label{04}
\mathbb{P}( |\Gamma_{N}(t)|> e^{ (\sqrt{2d}+\e) t}  )\leq e^{-\e t/8}
\end{align}
recalling $\Gamma_{N}(t) := \{v\in V_N: \pv \geq m_N-t\}$
(see Theorem \ref{theorem 3.0}).
Note that this quickly implies (by simply summing) the total Gibbs weight of  $\Gamma_{N}(t)^c$, decays as $e^{(-\beta+\sqrt{2d}+\e)t}$ with probability at least $1-e^{-c t},$ for some constant $c=c(\e).$ In particular, that the level set is shown to be no more than essentially $e^{\sqrt{2d}t}$ is crucial for us in being able to analyze the \emph{entire supercritical regime}, i.e., $\beta > \sqrt{2d.}$

A similar statement as in \eqref{04}  was previously  obtained for the 2D GFF in \cite{bl2} with the proof making crucial use of the concentric decomposition of 2D GFF based on the  Gibbs-Markov property. 
The unavailability of any semblance of a Markovian structure in our setting unfortunately renders this approach ineffective. Further, the strategy in \cite{drz} proving tightness, after centering, of the maxima, $\phi_{N,*},$ for general log-correlated fields relying on Slepian's lemma does not extend to cardinality of level sets. In more detail,  in \cite{drz}, Slepian's lemma (see Section \ref{section 2.1} for the precise statement) was used to  compare (in the sense of stochastic domination) $\phi_{N,*},$ both from above and below, with the maxima  of more easily tractable processes such as a BRW or  modified versions of it (explicit descriptions appear in Section \ref{section 2.4}) and use sharp estimates known about the latter. Unfortunately, the  cardinality of the level sets do not enjoy such a comparison principle.

Nonetheless, exploring whether such a comparison strategy still has any promise, one may consider the sum of $\ell$ largest values of $\pv$, denoted by $S_{\ell}(\phi_N)$. Note that an upper bound on  $S_{\ell}(\phi_N)$ is useful to study level sets since  trivially,
$$\Big\vert\{v\in V_N: \pv \geq \frac{S_{\ell}(\phi_N)}{\ell}+1\} \Big\vert \le \ell.$$
Still, the random variable $S_{\ell}(\phi_N)$ cannot be compared across Gaussian processes (an example is provided in Section \ref{section 3.1} for illustration). However, fortunately, it turns out that the {expectation} of $S_{\ell}(\phi_N)$ \emph{can be}.
 
The quantity $S_{\ell}(\phi_N)$ was first introduced in \cite{dz} and then further exploited in \cite{bl1} to establish that for the 2D GFF there exists $c>0$ such that for any large enough $k>0$,
  \begin{align}  \label{05}
  \mathbb{P}( |\Gamma_{N,t}| > e^{kt} ) \leq e^{-ckt}.
  \end{align}
  This was done by first establishing an upper bound on $\mathbb{E}S_\ell$ and then translating it to an upper bound on the cardinality of level sets. Besides the Gibbs-Markov property, an important ingredient was a bootstrapping step which has by now featured in various guises in the study of log-correlated processes with an underlying tree structure (and will be important for our arguments as well). Note however that \eqref{05} is stated to hold only for sufficiently large $k$.  Subsequently, this was extended to an estimate of the form \eqref{04} for all  values of $k>\beta_c$ in \cite{bl2},  exploiting again the Markovian structure.
 
In order to obtain the optimal estimate \eqref{05} for general log-correlated fields, which holds for \it{any} $k > \beta_c$, we first deduce a  sharp upper bound on the expectation of $ S_{N}(\phi_N)$: For any $\e>0$,
 \begin{align}  \label{06}
 \mathbb{E}S_{\ell}(\phi_N) \leq  \ell\Big(m_N  - \Big( \frac{1}{\sqrt{2d}} - \e \Big) \log \ell\Big).
\end{align} While we will eventually prove a matching lower bound as well (which is significantly more complicated), this step is relatively straightforward. It follows from first adapting existing techniques to obtain a similar bound for BRW. 
Then the comparison principle in invoked. In brief, this step employs Kahane's    
convex inequality which can be used to deduce that for two Gaussian processes with covariance structure of one pointwise dominating the other, expectation of $S_{\ell}$ is smaller for the former.

Given \eqref{06} for all large $N$, the proof of the level set bound \eqref{04}, follows by way of contradiction. Namely, if the latter were not true, then constructing an auxiliary field $\psi_{m,N},$ of size $2^{m}N$ for a suitably chosen $m$, using  essentially $2^{m}$ i.i.d. copies of $\phi_N$ (with some further adjustments which we will ignore in this discussion), one can deduce that it is highly likely for $\psi_{m,N}$ to have a large $S_{\ell}.$ This is the crux of the bootstrapping argument which allows one to conclude that if the probability of a large level set is not too low for $\phi_N,$ then the underlying i.i.d structure in $\psi_{m,N}$ makes a large level set extremely likely for the latter and hence forces a large value of $S_{\ell}(\psi_{m,N})$ and thereby its expectation. A final application of the comparison principle shows $\E(S_{\ell}(\psi_{m,N})) \le \E(S_{\ell}(\phi_{2^mN}))$. Thus the obtained lower bound for the LHS also lower bounds the RHS and hence contradicts \eqref{06}.

We now provide a brief commentary on the proof of the matching lower bound for the level set analogous to \eqref{04}, showing that it is typically also at least $e^{(\sqrt{2d}-\e)t},$ albeit with a much weaker failure probability estimate completing the proof of Theorem \ref{theorem level}. Our proof begins with a lower bound on $\E(S_{\ell}(\phi_N)).$ This involves first proving this for MBRW which already is rather non-trivial, followed by yet another application of the comparison principle. 
Refraining from discussing further details, we end by remarking that a key ingredient to achieve the above is a uniform lower bound on the right tail of  $S_\ell$ for $\text{MBRW}_N$: There exist constants $c,c'>0$ such that for any $\ell \geq 1$, for sufficiently large $N$,
\begin{align*}
 \mathbb{P}\Big(S_{\ell}(\text{MBRW}_N)  \geq \ell \Big(m_N-\frac{1}{\sqrt{2d}}\log \ell - c'\Big)\Big) \geq c.
\end{align*}
This is obtained by a second moment argument together with an estimate about the Brownian bridge (this is where the explicit structure of MBRW is used). The remaining steps involve ideas resembling those appearing in the proof of \eqref{04} including the bootstrapping argument. \\

\noindent
\textbf{Structure of the Gibbs measure:} Given the sharp control on $|\Gamma_{N}(t)|,$ we now discuss the remaining ideas in the proof of Theorem \ref{theorem 1.1}. As already outlined, \eqref{04} allows us to essentially ignore $\Gamma_{N}(t)^c,$ for a fixed large $t$ (we will eventually send $t\to \infty$).
At this point \cite[Lemma 3.3]{drz} comes to our aid stating that two near maxima cannot be \emph{mesoscopically separated}, i.e., for any two points $u,v \in \Gamma_{N}(t),$ one must have with high probability $$|u-v| \notin \Big[e^{e^{ct}}, \frac{n}{e^{e^{ct}}}\Big]$$
($c>0$ is a constant). This  implies that $\Gamma_{N}(t)$ is shattered into macroscopically separated clusters of diameter $O(e^{e^{ct}})$ (in the arguments we will take this to be $r_N$ as stated in Theorem \ref{theorem 1.1}). The remainder of this section discusses the steps in proving that the sequence of Gibbs weights of such clusters properly normalized approximately follows the law $\textup{PD}(\beta_c/\beta).$

Towards this, as in \cite{bl1}, we will consider all the points in $\Gamma_{N}(t)$ which are additionally a local maxima in a ball of radius say $O(e^{e^{ct}})$ and view the clusters as centered around them (we will see why this is useful shortly). 
Thus let
  \begin{align}\label{local ext}
C_{N,r}:=\{v\in V_N: \phi_{N,v}=\max_{u\in B(v,r)} \phi_{N,u}\}
\end{align}
and for $R>0$,
{\begin{align} \label{sbar}
\bar{S}_{v,R} := \sum_{u\in B(v,R)} e^{\beta(\pu - \pv)}.
\end{align}
so that
$\mu_N^\beta( B(v,R)) \propto e^{\beta(\pv-m_N)}\cdot \bar{S}_{v,R} .$}

We now define the point process formed by the extreme points decorated by the cluster weights,
 \begin{align} \label{02}
\eta_{N,r_N}: =  \sum_{v\in C_{N,r_N} }    \delta_{\pv - m_N} (dy) \otimes \delta_ {\bar{S}_{v,r_N/2} } (dz).
\end{align}
We show that any subsequential limit $\eta$ of the above point process is distributed as a Cox process, i.e., a Poisson point process (PPP) with a random intensity measure  (see \cite{cox} for a comprehensive treatment of such processes).
In particular we show that it takes the form $\text{PPP}(e^{-\sqrt{2d} y}dy \otimes \nu(dz)) $ for some random ``cluster measure" $\nu$ on $(0,\infty).$ Without going into further details, we simply remark that  the obtained tensorization of the intensity measure is crucial in using the above to prove Theorem \ref{theorem 1.1}. This last step also crucially uses the level set information in obtaining vital information about the random measure $\nu$ needed to carry out the necessary measure theoretic arguments on the space of random measures.

 A result of this type, but significantly stronger was previously obtained for 2D GFF in \cite{bl2} where it was shown that   \begin{align} \label{08}
 \bar{\eta}_{N,r_N}:=   \sum_{v\in C_{N,r_N} }   \delta_{v/N}  (dx) \otimes  \delta_{\pv - m_N} (dy) \otimes \delta_ { \{\phi_{N,v+x}-\pv\}_{x\in \mathbb{Z}^2} } (dz)
 \end{align}
 converges to the Cox process $\text{PPP}(\mathcal{Z}(dx)\otimes e^{-\sqrt{2d}y}dy \otimes {\mathrm{v}}(dz))$ for an explicit random measure $\mathcal{Z}$ (goes by the name of derivative martingale in the literature) and a \emph{deterministic} measure $\mathrm v$ on the space height functions on $\Z^2$ pinned to be $0$ at the origin. 
 Note that unlike \eqref{02}, the above point process keeps track of the location of the extrema points, as well as the entire local height functions around them, and not simply their Gibbs weights.
 
In 2D GFF, the Gibbs-Markov property provides enough control on the conditional covariance to \emph{decouple} the location of a local extrema and the local height function (indeed this was made rigorous in \cite{bl2}) allowing one to work with \eqref{08}. 

 This is  a bit too much to ask for in the general setting of Gaussian fields we are in without any local covariance information.
Nonetheless, fortunately for us, if we forego the spatial location, one can still deduce a tensorization of the intensity measure of the value at the local extrema as well as the Gibbs weight of the local cluster (not the entire local height function) which is enough to imply Theorem \ref{theorem 1.1} (Note that it is only shown that the intensity of the cluster measure is random unlike the deterministic counterpart in \eqref{08}).  Our arguments also in fact allow us to work with the first and the second coordinates in \eqref{08} (while we don't record a formal statement, a more elaborate discussion on this point is presented in Remark \ref{spatial} later).  Thus, to summarize, while the first and third coordinates  might not have a tractable joint description in the generality of our setting for reasons cited above, the level set bound in \eqref{04} is enough to prove such invariance principles for any limit point of \eqref{02}, where the third coordinate only encodes the cluster measure.

Refraining from further commentary on the discrepancy between \eqref{02} and \eqref{08} we end with a brief review of the general idea that allows one to prove such Cox process limits. This goes back to a beautiful argument in \cite{bl1} who showed that  the law of the point process formed by the first two coordinates of \eqref{08} remains invariant as the underlying GFF is evolved along an Ornstein-Uhlenbeck flow which allows them to appeal to a powerful theory built by Liggett essentially characterizing all such invariant point processes as Cox processes. 
A similar argument also appeared in the already mentioned article \cite{sz2} about critical points for a class of  spherical spin glass models.

\subsection{Organization of the paper}
In Section \ref{prelim},  we review Slepian's lemma and other properties of Gaussian random variables, and the result in \cite{drz} about the maximum of log-correlated Gaussian fields. 
Section \ref{section 3} is devoted to level sets where we prove Theorem \ref{theorem level}.  Section \ref{section 4} is focused on the analysis of the point process \eqref{02}, establishing an invariance under Brownian flow. Such invariant processes are characterized  in Section \ref{sec limit}.  Given these ingredients, the proof of Theorem \ref{theorem 1.1}  is carried out in Section \ref{section main theorem}. 
Along the way several properties of the BRW and MBRW are relied upon which are collected in  Section \ref{section 6}. The theory of random measures and their convergence features centrally in this work, and the important facts are quoted in Section \ref{appendix a}. Finally certain technical proofs are furnished in the Appendix.
 \subsection{Acknowledgement} The authors warmly thank  Oren Louidor, R\'emi Rhodes and Ofer Zeitouni for useful discussions and comments on an earlier draft of the paper.
SG is partially supported by NSF grant DMS-1855688, NSF Career grant DMS-1945172, and a Sloan Fellowship.   KN is supported by  Samsung Science and Technology
Foundation under Project Number SSTF-BA2202-02.

\section{Preliminaries}  \label{prelim}

We start by recording key properties of Gaussian fields used throughout the paper.
First, about the conditional distributions of Gaussians.

\begin{lemma} \label{lemma 3.2}
Consider  a centered Gaussian vector $X=(X_1,X_2)$ in $\R^{2}$ with correlation $\rho$, {such that $\textup{Var} X_1 = \sigma_1^2$ and $\textup{Var} X_2 = \sigma_2^2$}.
Then,   given $X_2$,  the   conditional mean of $X_1$  is 
\begin{align*}
\frac{\sigma_1}{\sigma_2}  \rho  X_2,
\end{align*}
and the conditional variance  of $X_1$ is 
\begin{align*}
(1-\rho^2 ) \sigma_1^2.
\end{align*}
Note that the conditional covariance of $X_1$ does not depend on the  value of conditioning $X_2$.
\end{lemma}

\subsection{Comparison principle for Gaussian fields} \label{section 2.1}
We now come to  Kahane's convex inequality \cite{k}, which can be used to compare  the  expectations of convex functions of two Gaussian vectors (see \cite{rv,rv2} for the English statement) whose covariance structure admits a point-wise comparison. As alluded to in Section \ref{section 1.3}, this will be a key input for our proofs as it was in \cite{drz}.

\begin{proposition}  \label{prop 3.3}
Consider two centered Gaussian vectors $X=(X_1,\cdots,X_n)$ and $Y=(Y_1,\cdots,Y_n)$ in $\R^n$. Suppose that $f\in C^2(\R^n)$ is a function, whose second derivative have a sub-Gaussian growth (i.e. for each $\e>0$, there is $C>0$ such {that $|f(x)| \leq Ce^{\e \norm{x}^2}$ for} all $x\in \R^n$). Assume that for any $i,j=1,\cdots,n$,
\begin{align*}
\begin{cases}
\textup{Cov}(X_i,X_j) \geq \textup{Cov}(Y_i,Y_j) \Rightarrow \frac{\partial^2 f}{\partial x_i \partial x_j} \geq 0, \\
\textup{Cov}(X_i,X_j) \leq \textup{Cov}(Y_i,Y_j) \Rightarrow \frac{\partial^2 f}{\partial x_i \partial x_j} \leq 0.
\end{cases} 
\end{align*}
Then,
\begin{align*}
\mathbb{E}f(X)\geq \mathbb{E}f(Y).
\end{align*}
\end{proposition}
{
As a consequence of Proposition \ref{prop 3.3} (applied to $f^{(m)}(x_1,\cdots,x_n):= \prod_{i=1}^n g^{(m)}_i(x_i)$ with smooth, non-increasing, non-negative and bounded functions $g^{(m)}_i \downarrow \1_{(-\infty, t_i]}$ as $m\ri$), one can deduce the following lemma.}
\begin{lemma} \label{lemma 3.4}
Let $X=(X_1,\cdots,X_n)$ and $Y=(Y_1,\cdots,Y_n)$ be two centered Gaussian vectors in $\R^n$ such that  
\begin{align*}
\mathbb{E}(X_i^2)=\mathbb{E}(Y_i^2), \quad \forall i=1,\cdots,n
\end{align*}
and  
\begin{align*}
\textup{Cov}(X_i,X_j) \leq \textup{Cov}(Y_i,Y_j), \quad \forall i,j=1,\cdots,n .
\end{align*}
Then, for any $t_1,\cdots,t_n\in \R$,
\begin{align*}
\mathbb{P}(X_i\leq t_i,  i=1,\cdots,n) \leq \mathbb{P}(Y_i\leq t_i, i=1,\cdots,n).
\end{align*}
In particular,  for any $t\in \R$,
\begin{align*}
\mathbb{P}(\max_{i=1,\cdots,n} X_i\leq t) \leq \mathbb{P}(\max_{i=1,\cdots,n} Y_i\leq t) .
\end{align*}
\end{lemma}

\subsection{Branching random walk and modified branching random walk} \label{section 2.4}

In this subsection, we recall branching random walk (BRW) and a modified version thereof called modified branching random walk (MBRW).  Both of these models are canonical Gaussian processes on an underlying tree like geometry leading to log-correlation.  Their utility in the current context stems from the applications of the comparison principle outlined in Section \ref{section 1.3}.

We first need a notational setup. 
For a non-negative integer $j$, let $\mathcal{B}_j$ and $\mathcal{B}^*_j$ be the collection of boxes of size $2^j$ whose lower left corners are elements of $\mathbb{Z}^d$ and $2^j \mathbb{Z}^d$, respectively.  For $v\in V_N$, define $\mathcal{B}_j(v)$ and $\mathcal{B}^*_j(v)$ to be the collection of boxes in  $\mathcal{B}_j$ and $\mathcal{B}^*_j$ containing $v$, respectively. {Since   $\mathcal{B}^*_j(v)$  only consists of a single element, we will use the same notation and  let $\mathcal{B}^*_j(v)$ also denote that element.} Finally, for a positive integer $N$, let $\mathcal{B}^N_j$ be a collection of boxes in $\mathcal{B}_j$ whose lower left corners lie in $V_N$ (recall that $V_N = \{0,1,\cdots,N-1\}^d$).

Now, we define  BRW.
Let $\{g_{j,B}\}_{j\geq 0, B\in \mathcal{B}^*_j} $ be a collection of i.i.d. centered Gaussians of variance $\log 2$.   Then, for $N=2^n$ with $n\in \mathbb{N} $, BRW to be denoted by $\varphi_N = (\varphi_{N,v})_{v\in V_N}$ is defined by
\begin{align*}
\varphi_{N,v} := \sum_{j=0}^{n-1}   g_{j, \mathcal{B}^*_j(v)}.
\end{align*}

{Note that although BRW possesses a tree structure, owing to some rounding issues, it does not exhibit an exact log-correlation property in the sense of Assumption \ref{a2}. For instance, BRW at two  adjacent points in two disjoint dyadic boxes of  size $2^{n-1}$ have covariance $\log 2$, which is much less than $\log (N/1) = \log N$.   This however can be fixed by the following simple averaging procedure which leads to MBRW.
  Let $\{g_{j,B}\}_{j\geq 0, B\in \mathcal{B}_{j}^N} $ be  a collection of i.i.d. centered Gaussians with $\var g_{j,B} =  2^{-dj} \cdot \log 2$ for $B\in \mathcal{B}_{j}^N$.  For $j\geq 0$ and $B \in \cB_j$, define
  \begin{align*}
  g_{j,B}^N :=   \begin{cases}
    g_{j,B}, &B\in \mathcal{B}_j^N, \\
      g_{j,B'}, &B\notin \mathcal{B}_j^N,  B\sim_N B' \in \mathcal{B}_j^N,
  \end{cases}
\end{align*}}
where we write $B\sim_N B'$ if $B = B'+(i_1N,\cdots,i_dN)$ for some  $i_1,\cdots,i_d\in \mathbb{Z}$.
Then, for $N=2^n$, MBRW $\theta_N  = (\theta_{N,v})_{v\in V_N} $ is defined by
\begin{align} \label{mbrw}
\theta_{N,v} := \sum_{j=0}^{n-1}  \sum_{B\in \mathcal{B}_j(v)}  g_{j,B}^N.
\end{align}

{Since  $|\mathcal{B}_j(v)| = 2^{dj}$ and recalling  $\var g_{j,B} =  2^{-dj} \cdot \log 2$, we have $\var \theta_{N,v}  = \log N  $}.  Further, the following holds.

\begin{lemma}[\cite{drz}, Lemma 2.3] \label{lemma 2.3}
There exists  a constant $c_1>0$ (depending on $d$) such that for any $u,v\in V_N$,
\begin{align}\label{mbrwcov}
|\textup{Cov} ( \theta_{N,u},\theta_{N,v}) - (\log N - \log_+ |u-v|^{(N)}) | \leq c_1,
\end{align}
where $|u-v|^{(N)} = \min_{v'\sim_N v} |u-v'|$. Here, we say $v' \sim_N v$  if $v'-v\in N \mathbb{Z}^d$.
\end{lemma}

\subsection{Maximum of log-correlated Gaussian fields}
\label{section 2.2}
While the present paper focuses on the Gibbs measure, much of the literature is in fact devoted to the study of the ground state, i.e., the maximum of the underlying Gaussian process. For instance, the maximum of two-dimensional DGFF \cite{bdz,bz,d}, branching random walk \cite{ar,a,as,hs}, and the branching Brownian motion \cite{b}  are well-understood owing to their underlying Markovian properties inherited from the tree geometry.  Further, the maximum of $*$-scale invariant fields was studied in \cite{m}. Note, however that all of the above fields admit special representations and the question of universality had stayed open until the breakthrough paper
 \cite{drz}  which established a sharp fluctuation theory for the maximum of a general log-correlated Gaussian field.  More precisely, recalling that
$
\phi_{N,*} := \max_{v\in V_N} \pv
,$
 and
\begin{align} \label{mN}
m_N := \sqrt{2d}\log N -\frac{3}{2\sqrt{2d}}\log \log N,
\end{align}
  under the Assumptions \ref{a1} and \ref{a2},
\begin{align} \label{max} 
\mathbb{E}  \phi_{N,*} = m_N+O(1).
\end{align}

{In addition, it was shown in \cite{drz} that under some additional conditions on the covariance of the field $\phi_N$ (i.e. convergence of the covariance around the diagonal in a microscopic scale and convergence of the off-diagonal  covariance in a macroscopic scale), the sequence $\{  \phi_{N,*}- \mathbb{E} \phi_{N,*} \}_{N\geq 1}$ converges in distribution and  its  limiting law  can be expressed in terms of the Gumbel distribution with a random shift, thereby establishing a universal tail behavior.}

Next, we record some known tail estimates of  the maximum of  log-correlated Gaussian fields which we will make heavy use of.  

\begin{lemma}[\cite{drz}, Proposition 1.1]  \label{lemma 2.4}
Under Assumption \ref{a1}, there exists a constant $C>0$ such that for any $t\geq 1$,
\begin{align*}
\mathbb{P}(\max_{v\in V_N} \pv \geq m_N+t) \leq Cte^{-\sqrt{2d}t} e^{-C^{-1}t^2/n}.
\end{align*}
In addition,  for any $t\geq 1$, $s\geq 0$ and $A\subseteq V_N$,
\begin{align} \label{241}
\mathbb{P}\big(\max_{v\in A} \pv \geq m_N+t-s\big) \leq C \Big(\frac{|A|}{N^d}\Big)^{1/2} te^{-\sqrt{2d}(t-s)}.
\end{align}
\end{lemma}

One important consequence of the above is that 
for any $\lambda>0$,
{\begin{align} \label{410}
\lim_{\delta \downarrow 0} \limsup_{N\rightarrow \infty} \mathbb{P}( \max_{V_N \backslash  V_N^\delta} \pv  \geq  m_N-\lambda)=0.
\end{align}
This is obtained by taking $t=1$, $s=\lambda+1$ and $A=V_N \backslash  V_N^\delta$  in \eqref{410}. }

 The following lemma, which will be proved in  Section \ref{section 3}, provides the upper bound for the left tail of the maximum restricted to a sub-box of $V_N$. 
 
\begin{lemma} \label{lemma 2.8}
{Let $\phi_N$ be a centered Gaussian field satisfying Assumptions \ref{a1} and \ref{a2}.  Then,
there exist constants $C,c_2,c_2',t_0>0$ such that for  any  $t\geq t_0$, $L\geq 2$, $N \geq c_2' L$, and any  box $B\subseteq V_N^{1/10}$ of size $  \left \lfloor{N/L  }\right \rfloor  
$,}
\begin{align} \label{280}
 \mathbb{P}(\max_{v\in B} \phi_{N,v} \leq m_N -2\sqrt{2d}\log L - t ) \leq Ce^{-c_2t}
\end{align}
and
\begin{align} \label{2800}
 \mathbb{P}(\max_{v\in B} \theta_{N,v} \leq m_N -2\sqrt{2d}\log L - t ) \leq Ce^{-c_2t}.
\end{align}
\end{lemma}
The result for BRW can also be obtained by a comparison argument or directly but for our arguments we will only need the above two estimates.
{This lemma is a generalization of \cite[Lemma 2.1]{drz}  and \cite[Lemma 2.8]{drz}, where corresponding results for the \emph{whole} region $V_N$ was obtained: Under Assumption \ref{a2},
\begin{align} \label{289}
 \mathbb{P}(\max_{v\in V_N} \phi_{N,v} \leq m_N  - t ) \leq Ce^{-ct} ,\quad \mathbb{P}(\max_{v\in V_N} \theta_{N,v} \leq m_N  - t ) \leq Ce^{-ct}.
\end{align}
}

{A simple corollary of Lemma \ref{lemma 2.4} and  \eqref{289} is the tightness of the centered maximum for any centered Gaussian fields satisfying Assumptions \ref{a1} and \ref{a2} which we record for the purpose of reference.
\begin{corollary}[\cite{drz}, Theorem 1.2] \label{tight}
Let $\phi_N$ be a centered Gaussian field satisfying Assumptions \ref{a1} and \ref{a2}. Then, the sequence $ \{\phi_{N,*}  - m_N \}_{N\geq 1}$ is tight.
\end{corollary}}

\section{Level sets of log-correlated Gaussian fields} \label{section 3}
This section is devoted to proving Theorem \ref{theorem level}. {Throughout this section, Assumptions \ref{a1} and \ref{a2}
will remain in force without being explicitly stated repeatedly.}
Recall that for $t\in \R$,
{ \begin{align} \label{level set notation}
\Gamma_{N}(t) := \{v\in V_N:\pv \geq  m_N-t\}.
\end{align}}

We start with the upper bound.

 \begin{theorem} \label{theorem 3.0}
There exists a constant $c_0>0$ such that for any constant  $\e> 0 $,  for  large enough $t$ and any $N\geq c_0 e^{(\sqrt{2d}+\e)t/d}$,
 \begin{align} \label{3000}
  \mathbb{P}(|\Gamma_N(t) |  \geq  e^{(\sqrt{2d}+\e) t}) \leq  e^{-  \e t/8}.
 \end{align}
 \end{theorem}
 
 Recall from Section \ref{section 1.3} that the above allows us to essentially ignore the Gibbs mass outside $\Gamma_N(t)$ for a large but fixed $t$ (see Lemma \ref{lemma 7.7} for a precise statement).\\

 The next result states the lower bound.

\begin{theorem} \label{theorem lower}
{Let $\e,\iota>0$ be any constants. Then, for  any large enough (fixed) $t$, for sufficiently large $N$,}
\begin{align*}
\mathbb{P}(|\Gamma_N(t) | \geq  e^{ (\sqrt{2d}-\e )t} )  \geq 1-\iota.
\end{align*}
\end{theorem}

Note that in this case we merely establish a much weaker failure probability estimate; nonetheless together the above imply Theorem \ref{theorem level}. It is worth pointing out that we do establish an exponential lower bound on $|\Gamma_N(t)|$ with a different coefficient in the exponent with exponentially small failure probability (see the discussion after Lemma \ref{lemma 3.9}).

 \subsection{Expectation of sum of maximum values} \label{section 3.1} Recall that our strategy for proving the above theorems rely on obtaining sharp estimates on the sum of the $\ell$ largest values of $\phi_N$. We first do this at the level of expectation by first obtaining the same for an appropriate BRW and appealing to comparison.

 For any centered Gaussian field $X = (X_v)_{v\in V}$ with a   finite set $V$ and a positive integer $\ell$,
 define  $ S_{\ell}(X) $ to be the sum of largest $\ell$ values of  $X$:
 \begin{align} \label{sum}
 S_{\ell}(X) := \max\Big\{\sum_{v\in I}X_v:   I\subseteq V, |I|=\ell\Big\}.
 \end{align}
 We define $S_\ell(X) := - \infty $ when $\ell>|V|$.
For example, if $\ell=1$, then  $S_{\ell}(X)=S_1(X) = \max_{v\in V} X_v$. As mentioned in Section \ref{section 1.3}, while the maximum can be compared (as random variables) across Gaussian processes, this is no longer true for $S_{\ell}$ once $\ell>1$. For example, for two independent standard Gaussian random variables $Z$ and $W$, consider two Gaussian vectors  $X=(Z,W)$ and $Y=(Z,Z)$. Clearly, the covariance of $X$ is dominated by that of $Y,$ with the variances agreeing. However, since, $S_{2}(X)=Z+W$  and  $S_{2}(Y)=2Z$ and $Z+W \overset{\text{law}}{\sim} \mathcal{N}(0,2)$, $2Z \overset{\text{law}}{\sim} \mathcal{N}(0,4)$, for any $t>0$,
\begin{align*}
\mathbb{P}(Z+W > t ) < \mathbb{P}(2Z>t).
\end{align*}
which violates the order expected from Lemma \ref{lemma 3.4}.

{However, fortunately, as was shown in \cite{dz}, a comparison can be established at the level of expectation.} {This involves considering new Gaussian fields formed by the summation of all possible $\ell$ distinct elements in $X$ and $Y$, along with an application of Proposition \ref{prop 3.3}.}
 
 \begin{lemma}[\cite{dz}, Lemma 2.7] \label{lemma 3.3}
 Let $X=(X_1,\cdots,X_n)$ and $Y=(Y_1,\cdots,Y_n)$ be two centered Gaussian vectors in $\R^n$ such that 
\begin{align*}
\mathbb{E}(X_i^2)=\mathbb{E}(Y_i^2),\quad \forall i=1,\cdots,n
\end{align*}
and
\begin{align*}
\textup{Cov}(X_i,X_j) \leq \textup{Cov}(Y_i,Y_j),\quad \forall i,j=1,\cdots,n.
\end{align*}
Then, for any $\ell=1,\cdots,n$,
 \begin{align*}
 \mathbb{E}S_\ell(X)  \geq  \mathbb{E}S_\ell(Y) .
 \end{align*}
 \end{lemma}
 
We next for the record state the following simple but useful fact.

 \begin{lemma} \label{lemma 3.5}
For a finite set $I$, suppose that $X = (X_i)_{i\in I}$ and $Y =  (Y_i)_{i\in I}$ are independent random vectors such that  $\mathbb{E}X_i=\mathbb{E}Y_i=0$ for all $i\in I$. Then,
\begin{align*}
\mathbb{E} S_\ell(X) \leq \mathbb{E} S_\ell (X+Y).
\end{align*}
  \end{lemma}

 \begin{proof} 
 The proof is straightforward by conditioning on $X$ and thereby fixing its $\ell$ largest entries and then averaging over $Y.$

 \end{proof}

\subsection{Expectation bounds}
The main result of this subsection is 
 a sharp estimate  on the quantity $\mathbb{E} S_\ell(\phi_N) $.

 \begin{proposition} \label{prop 3.6}
For any constant $\lambda_1<\frac{1}{\sqrt{2d}}$, for   large enough $\ell$ and $N\geq 1$,
\begin{align} \label{363}
  \mathbb{E} S_\ell(\phi_N)   \leq   \ell(m_N-\lambda_1\log \ell).
\end{align}
In addition, let
 $\lambda_2> \frac{1}{\sqrt{2d}}$ be a constant. Then,  for large enough $\ell$, for sufficiently large $N$,
\begin{align} \label{364}
   \mathbb{E} S_\ell(\phi_N)  \geq \ell(m_N-\lambda_2\log \ell)  .
\end{align}
 \end{proposition}

The proof of the upper bound is simpler and this is what we begin with. Recall that for any $A\subseteq \mathbb{Z}^d$ and a positive integer $k$, we denote $kA := \{x \in \mathbb{Z}^d: x/k \in A\}.$

\subsubsection{Upper bound in Proposition \ref{prop 3.6}}

The upper bound \eqref{363} is  a direct consequence of   a comparison with a suitably dilated BRW. 
For $\kappa\in \mathbb{N}$, define the dilation map $\Phi_N^\kappa: V_N \rightarrow 2^\kappa V_N \subseteq V_{2^k N}$ by
\begin{align*}
\Phi_N^\kappa (v) = 2^\kappa v.
\end{align*}
Let $X$ be a standard Gaussian random variable independent of everything else. Then, there exists a sufficiently large integer $\kappa>0$ and positive numbers $a_v>0  \ (v\in V_N)$ such that
\begin{align} \label{367}
\var (\pv + a_v X) = \var (\varphi_{2^\kappa N, \Phi_N^\kappa(v)}),\quad \forall v\in V_N
\end{align}
and
\begin{align} \label{368}
\cov (\pu + a_u X,\pv + a_v X) \geq  \cov ( \varphi_{2^\kappa N, \Phi_N^\kappa(u)},\varphi_{2^\kappa N, \Phi_N^\kappa(v)}),\quad \forall u,v\in V_N 
\end{align}
Note that owing to the right hand side in \eqref{367} being the same for any two vertices $u,v$, it follows that $|a_v^2 - a_u^2|=|\var \pu - \var \pv|.$
{Now, while \eqref{368} was shown in  \cite[Lemma 2.5]{drz}, we provide the proof for completeness. Since $ \var (\varphi_{2^\kappa N, \Phi_N^\kappa(v)}) =\log ( 2^\ka N)$, by the condition \eqref{a0}, one can take $a_v>0$ satisfying \eqref{367}  for large enough $\kappa>0$. Next,   since
\begin{align*}
\mathbb{E}(\pu + a_u X - \pv  - a_v X)^2 &= \mathbb{E} (\pu-\pv)^2 + (a_v-a_u)^2 \\
& \leq   \mathbb{E} (\pu-\pv)^2 + |a_v^2 - a_u^2|  \\
&\overset{\eqref{a}}{ \leq}  (2\log_+|u-v| - |\var \pu - \var \pv| + 4\alpha_0  )   +  |\var \pu - \var \pv| \\
&=  2\log_+|u-v| +4\alpha_0 
\end{align*} 
and 
\begin{align*}
\mathbb{E} ( \varphi_{2^\kappa N, \Phi_N^\kappa(u)} - \varphi_{2^\kappa N, \Phi_N^\kappa(v)})^2  \geq  2\log 2^\kappa  +  2\log_+|u-v|  - C_0
\end{align*}
($C_0$ is an absolute constant), \eqref{368} holds for sufficiently large $\kappa>0$.
}

Thus,
by Lemma   \ref{lemma 3.3},
\begin{align} \label{360}
\mathbb{E} S_\ell( (\phi_{N,v} + a_vX )_{v\in V_N} )   \leq    \mathbb{E} S_\ell(\varphi_{2^\kappa N} |_{2^\kappa V_N}),
\end{align}
where  $\varphi_{2^\kappa N} |_{2^\kappa V_N}$ is the field $ \varphi_{2^\kappa N}$  restricted to the lattice  $2^\kappa V_N$. Also, by Lemma \ref{lemma 3.5},
\begin{align} \label{365}
\mathbb{E} S_\ell(\phi_N)  \leq \mathbb{E} S_\ell( (\phi_{N,v} + a_vX )_{v\in V_N} )   .
\end{align}
Since  $S_\ell(\varphi_{2^\kappa N} |_{2^\kappa V_N}) \leq   S_\ell(\varphi_{2^\kappa N})$, by the above two inequalities,
\begin{align} \label{361}
\mathbb{E} S_\ell(\phi_N)    \leq   \mathbb{E} S_\ell(\varphi_{2^\kappa N}).
\end{align}

{Using the tree structure of BRW, one can  show that  (see  {Lemma \ref{lemma 6.2}}  for the precise statement and its proof) for any  constant $\lambda' \in (\lambda_1, \frac{1}{\sqrt{2d}})$, for  large enough $\ell$,}
 \begin{align} \label{366}
\mathbb{E} S_\ell(\varphi_{2^\kappa N}) \leq \ell(m_{2^\kappa N} -\lambda' \log \ell).
\end{align}
Note that by the expression of $m_N$ in \eqref{mN},  for any $N_1\geq N_2\geq 4$,
{\begin{align} \label{462}
m_{N_1}- m_{N_2} = \sqrt{2d}\log (N_1/N_2) - \frac{3}{2\sqrt{2d} }\log (\log N_1 / \log N_2)  \leq \sqrt{2d}\log (N_1/N_2)  .
\end{align}}
 Thus, by \eqref{366} and \eqref{462} {(with $N_1 = 2^\kappa N$ and $N_2  = N$)}, for   large enough $\ell$,
\begin{align} \label{362}
 \mathbb{E} S_\ell(\varphi_{2^\kappa N}) \leq \ell(m_N  - \lambda_1 \log \ell ).
\end{align}
Therefore, \eqref{361} and \eqref{362} conclude the proof of  \eqref{363}.

\qed

We next move on to the proof of the lower bound in Proposition \ref{prop 3.6}, which is significantly more complicated and will already see the bootstrapping argument in action.

\subsubsection{Lower bound in Proposition \ref{prop 3.6}} \label{section aux}

 There will be broadly three steps.
We will employ a comparison by constructing a field whose covariance dominates that of $\phi_N.$ 
\begin{itemize}
\item We will construct a field made of i.i.d copies of MBRW $\theta_N$ with some independent perturbation to ensure Lemma \ref{lemma 3.3} can be invoked. {Note that  this is where MBRW turns out to be more handy than BRW, since the former possesses more correlations than the latter as noted in the  discussion in Section \ref{section 2.4}).}
\item We will use an apriori estimate on the probability that $S_\ell(\theta_N)$ is not too small, and the i.i.d. structure in our construction, to deduce that $S_{\ell}$ for the latter is not too small with overwhelming probability allowing us to lower bound the expectation of the same.
\item Finally a comparison allows us to conclude the same for $\phi_N.$
\end{itemize}

Recalling the definition of MBRW from Section \ref{section 2.4}, let us now proceed to the construction. Let $K \in \mathbb{N}$ be a large constant to be chosen later and further let $N'=2^{n'}$ with  $n'\in \mathbb{N}$ such that $N\geq 6KN'$. 
 Consider the collection $\mathcal{D} = \{D_i\}_{i\in I}$    such that for all $i\in I$,
 \begin{enumerate} \label{condition}
 \item  $D_i \subseteq V_{N}^{1/10}$,
 \item   $D_i  = 2KN' x_i + KV_{N'} $ for some $x_i \in \mathbb{Z}^d$.
\end{enumerate}  
Thus $D_i$s are obtained by dilating $V_{N'}$ by $K$ and translating them by multiples of $2KN'$ ($D_i =  2KN'x_i + \{ (Ky_1,\cdots,Ky_d): y_1,\cdots,y_d  \in \{0,1,\cdots,N'-1\}\}$) and checking whether they are fully contained in $V_{N}^{1/10}.$
{Since the side length of  $V_{N}^{1/10}$ is at least $\frac{8}{10}N-1$, for any large $N$, 
\begin{align} \label{3480}
|I|\geq \left\lfloor{\frac{(8/10) N-1}{2KN'} }\right \rfloor   ^d   \geq  \left \lfloor{\frac{N}{3KN'} }\right \rfloor   ^d .
\end{align}}
Also, by the second condition, $d_\infty (D_i,D_j) \geq KN'$ for any $i\neq j$ and hence in particular they are \emph{disjoint}.     Set $D: = \cup_{i\in I}D_i$  and for $v\in D$, denote by $D(v)$ a unique element $D_i$ in $ \mathcal{D}$ containing $v$.

Noting that $\theta_{N'} =  ( \theta_{N',v'} ) _{v'\in V_{N'}} $ denotes the MBRW  on $V_{N'}$, let $\{\theta_{D_i}\}_{i \in I}$ be i.i.d. Gaussian fields defined on the $D_i$s simply by taking i.i.d. copies of $\theta_{N'}$, one for each $D_i$ and simply dilating and translating the one corresponding to $D_i$ to identify $V_{N'}$ with $2KN'x_i+KV_{N'}$.Thus, formally, $\theta_{D_i}$ is defined on  $D_i$  such that
\begin{align} \label{aux1}
(\theta_{D_i,2KN'x_i+Kv'})_{v' \in V_{N'}} \overset{\text{law}}{\sim}  ( \theta_{N',v'} )_{v' \in V_{N'}}
\end{align} 

{Note that, by the statement right before Lemma \ref{lemma 2.3},  for any $v\in D$, $\var \theta_{D(v),v} =  \log N' $ and further by Assumption \ref{a2} and noting that $v\in  V_{N}^{1/10}$,  $$\var \phi_{N,v}  \geq  \log N - \alpha(1/10) \geq  \log (6KN') - \alpha(1/10).$$
 Hence, for large enough $K$, $ \var \phi_{N,v} >  \var \theta_{D(v),v}$ for all $v\in D$. For any such $K$, define $a_v>0$ by
 \begin{align} \label{av0}
 a_v^2 := \var \phi_{N,v} -  \var \theta_{D(v),v}.
 \end{align}
Finally, we define the field $\bar{\theta}_N^{N',K} =  ( \bar{\theta}_{N,v}^{N',K})_{v\in D}$  by
\begin{align} \label{aux0}
 \bar{\theta}_{N,v}^{N',K} :=  \theta_{D(v),v} + a_vY,\quad  v\in D,
\end{align}
where  $Y$ denotes  a  standard Gaussian random variable  independent of everything else. 
}

The next lemma follows by comparing the covariances and appealing to Lemma \ref{lemma 3.3}.

 \begin{lemma} \label{lemma 3.8}
There exists a constant  $K \in \mathbb{N}$ such that for any  $N,N', \ell \in \mathbb{N}$ with $N \geq 6KN'$,
\begin{align}
\mathbb{E} S_\ell(\bar{\theta}_N^{N',K})  \leq \mathbb{E} S _{\ell}(\phi_N)  .
\end{align}
\end{lemma}
Postponing the computations involving covariances required for the above proof until later, note that this reduces the proof of the lower bound in Proposition \ref{prop 3.6} to lower bounding the LHS.

We now set on accomplishing this. This involves a bit of preparation. The first step as already indicated at the beginning of the section is a lower bound on the right tail of $S_\ell (\theta_N).$   The simplest case of $\ell=1$, i.e. the lower bound for the right tail of the maximum of MBRW,  was already examined in \cite[Propositon 5.3]{bz} using the exact structure of MBRW and related random walk estimates and {we record this below for expository reasons. } 

\begin{lemma}[\cite{bz}, Propositon 5.3] \label{lemma mbrw}
There exist   constants $c_3,c_4>0$   such that  for any $N\geq 1$,
\begin{align*}
 \mathbb{P} (  \max_{v\in V_N} \theta_{N,v} >  m_N-  c_3 ) \geq c_4.
\end{align*}
\end{lemma}

The next key ingredient generalizes the above to larger values of $\ell$. Note that the centering is in accordance with the prediction that for any $t>0$, one has $|\Gamma_{N}(t)|\approx e^{\sqrt{2d}t}.$
\begin{lemma} \label{lemma 3.7}
There exist   constants $c_3,c_4>0$   such that for any   $\ell\geq 1$, for sufficiently large $N$,
\begin{align*}
 \mathbb{P}\Big( S_\ell (\theta_N) > \ell\Big(m_N-\frac{1}{\sqrt{2d}}\log \ell - c_3\Big)\Big) \geq c_4.
\end{align*}
\end{lemma}
{The proof of this is similar to that of the preceding lemma involving technical random walk estimates and is deferred to Section \ref{section 6}.}

Note next that to prove a lower bound for the expectation, it is not enough to show that a random variable is reasonably large with high probability, one in fact has to control the \emph{entire} lower tail to ensure that large negative numbers, albeit occurring with low probability, does not affect the expectation significantly. While momentarily it is enough to obtain such a result for $\theta_N,$ we record simultaneously a version for $\phi_N$ which will be relied upon later, in the proof of Theorem \ref{theorem 3.0}.

\begin{lemma} \label{lemma 3.9}
{There exist constants $C,c_5,c_5',t_1>0$ such that for any $t\geq t_1$, $\ell \geq 1$ and $N\geq c_5'\ell^{1/d}$,}
\begin{align} \label{3900}
  \mathbb{P}\Big(  S_{\ell} (\phi_N) \leq   \ell \Big( m_N-2\sqrt{\frac{2}{d}}\log \ell -t\Big) \Big) \leq C\ell e^{-c_5t}
\end{align}
and
\begin{align} \label{391}
  \mathbb{P}\Big(  S_{\ell} (\theta_N) \leq   \ell \Big( m_N-2\sqrt{\frac{2}{d}}\log \ell -t\Big) \Big) \leq C\ell e^{-c_5t}.
\end{align}
\end{lemma}
Note that the centering here is larger than that in  Lemma \ref{lemma 3.7}.   

\begin{remark}
{Lemma \ref{lemma 3.9} also provides an exponential bound on the left tail of the cardinality of level sets. In fact, setting
\begin{align}\label{407}
\ell:= \lfloor e^{ 2 \kappa t}  \rfloor
\end{align}
where $\kappa>0$ is a small universal constant which will be chosen later, 
 the event $\{|\Gamma_N(t)| < e^{ \kappa t}\} \cap \{\max_{v\in V_N}  \pv < m_N+t\} $ implies 
\begin{align} \label{408}
S_\ell (\phi_N) &\leq e^{ \kappa t}   (m_N+ t  ) + (\ell - e^{ \kappa t} ) (m_N-t) \nonumber = \ell m_N   -t\ell  +2t e^{ \kappa t} \nonumber \\
& \overset{\eqref{407}}{ \leq}  \ell m_N   - \frac{1}{2\kappa}  \ell \log \ell +\frac{1}{\kappa} \sqrt{\ell+1}\log ( \ell+1) \leq   \ell m_N   - \frac{1}{3\kappa} \ell \log \ell 
\end{align} 
for large enough $\ell$.
 By Lemma \ref{lemma 3.9}, for small enough $\kappa>0$, there exists $c(\kappa)>0$ such that
\begin{align} \label{409}
  \mathbb{P}\Big(  S_{\ell} (\phi_N) \leq  \ell m_N   -  \frac{1}{3\kappa}  \ell \log \ell  \Big) \leq  C \ell  e^{-c_5  ( 1/3\kappa - \sqrt{8/d})\log \ell } \overset{\eqref{407}}{ \leq}  C e^{ -c(\kappa) t}.
\end{align}
Therefore, for such  small constant $\kappa>0$, using  Lemma \ref{lemma 2.4}, for large enough $\ell$,
\begin{align*}
\mathbb{P}(|\Gamma_N(t)| < e^{ \kappa t})& \leq \mathbb{P}\big(\{|\Gamma_N(t)| < e^{ \kappa t}\} \cap \{\max_{v\in V_N}  \pv < m_N+t\}\big) + \mathbb{P}\big( \max_{v\in V_N}  \pv \geq m_N+t\big) \\
& \overset{\eqref{408}}{ \leq}   \mathbb{P}\Big(  S_{\ell} (\phi_N) \leq  \ell m_N   -  \frac{1}{3\kappa}  \ell \log \ell  \Big)+ Cte^{-\sqrt{2d}t} \\
& \overset{\eqref{409}}{ \leq}   C e^{ - c(\kappa) t}+  Cte^{-\sqrt{2d}t} \leq  Ce^{-c'(\kappa) t}
\end{align*}
for some constant $c'(\kappa)>0$.
}
\end{remark}
~

We now proceed to finishing the proof of \eqref{364} using Lemmas \ref{lemma 3.7} and \ref{lemma 3.9}. Let us define
\begin{align} \label{shat}
{\hat{S}_\ell: = {S}_\ell(\{{\theta}_{D(v),v} \}_{v\in D} )} = \max\Big\{\sum_{v\in J}  {\theta_{D(v),v}} : J \subseteq D, |J|=\ell\Big\}
\end{align}
The following simple identity will be useful:
    for any random variable $X \in L^1(\mathbb{P})$ and $M\in \R$,
\begin{align} \label{useful}
\mathbb{E}X= M - \int_{-\infty}^{M} \mathbb{P} (X \leq t)dt  + \int_{M}^\infty \mathbb{P} (X \geq t)dt.
\end{align}

\begin{proof}[Proof of  the lower bound in Proposition \ref{prop 3.6}]

Take any 
\begin{align} \label{eta}
\eta \in \Big(0, \frac{1}{2\sqrt{2d}}\Big( \lambda_2 - \frac{1}{\sqrt{2d}}\Big)\Big)
\end{align}
 and a large constant $K\in \mathbb{N}$  for which Lemma \ref{lemma 3.8} holds. 
{Then, define $N'=2^{n'}$ by
\begin{align}  \label{341}
n '=    \left \lfloor{ \log_2  \Big(\frac{N}{6K \ell^\eta}\Big)  }\right \rfloor . 
\end{align}
The choice of $N'$  (with  a large enough $N/N'$) is such that the number of boxes $D_i$s in $D$ is sufficiently large to employ a bootstrap argument.}
Note that $n'\geq 1$ for sufficiently large $N$. Using the fact $\lfloor x \rfloor \geq x/2$ for $x\geq 1$,
\begin{align} \label{348}
|I| \overset{ \eqref{3480}}{ \geq} \left \lfloor{\frac{N}{3KN'} }\right \rfloor   ^d \overset{\eqref{341}}{ \geq } \ell^{d\eta}.
\end{align}
Note that by the definition of $\hat{S}_\ell$  in \eqref{shat},
\begin{align}
\hat{S}_\ell \geq \max_{i\in I}  S_\ell( \theta_{D_i}).
\end{align}
Recalling  $\theta_{D_i} \overset{\text{law}}{\sim} \theta_{N'}$,  by Lemma \ref{lemma 3.7}, for each $\ell\geq 1$, for sufficiently large $N'$,
\begin{align} \label{340}
\mathbb{P}\Big( S_\ell( \theta_{D_i})  \leq \ell\Big( m_{N'} - \frac{1}{\sqrt{2d}}\log \ell-c_3\Big) \Big) \leq 1-c_4,\quad \forall i\in I.
\end{align}
Thus, using the independent of  the fields $ \theta_{D_i}$s ($i\in I$),
\begin{align} \label{342}
\mathbb{P}\Big(\hat{S}_\ell \leq \ell\Big( m_{N'} - \frac{1}{\sqrt{2d}}\log \ell-c_3\Big)\Big)  \leq  (1-c_4)^{\ell^{d\eta}}.
\end{align}
On the other hand, by the crudest possible bound, since  $\hat{S}_\ell \geq    S_\ell( \theta_{D_i})$  for any $i \in I$ and $ \theta_{D_i} \overset{\text{law}}{\sim}  \theta_{N'}$,   for any $M\in \R$,   
$$
  \mathbb{P}(\hat{S}_{\ell}  \leq M) \leq   \mathbb{P}(  S_{\ell}(\theta_{N'}) \leq   M ).
$$
Thus,  by Lemma \ref{lemma 3.9},  {for  large enough $\ell$} and  $N'\geq c_5' \ell^{1/d}$,
   \begin{align}  \label{343}
     &\int_{-\infty}^{\ell(m_{N'}-(2\sqrt{\frac{2}{d}} + \frac{1}{c_5})\log \ell)} \mathbb{P}(\hat{S}_{\ell}  \leq t)dt \nonumber  \\
&\leq   \ell \int_0^\infty    \mathbb{P}\Big(  S_{\ell} (\theta_{N'}) \leq   \ell \Big( m_{N'}-\Big(2\sqrt{\frac{2}{d}} + \frac{1}{c_5}\Big)\log \ell -s \Big)\Big)  ds   \nonumber \\
     &\leq  C \ell \int_0^\infty \ell e^{-c_5(  \frac{1}{c_5} \log \ell + s)}ds \leq C\ell.
   \end{align}
{Note that a largeness condition on $\ell$ is needed to satisfy the condition $t \geq t_1$ (with $ t:=\frac{1}{c_5}\log \ell + s \geq  \frac{1}{c_5}\log \ell$) in Lemma \ref{lemma 3.9}. }

{Now, we lower bound $\mathbb{E} \hat{S}_{\ell}  $ using the identity   \eqref{useful}, with the aid of  \eqref{342} and \eqref{343}.   To be precise, setting $c' :=  2\sqrt{\frac{2}{d}} + \frac{1}{c_5}$,}
{\begin{align}   \label{345}
\mathbb{E} \hat{S}_{\ell}  &\geq  \ell(m_{N'}-c' \log \ell) - \int_{-\infty}^{ \ell(m_{N'}- c'\log \ell) }  \mathbb{P}(\hat{S}_{\ell}   \leq t)dt +  \int_{\ell(m_{N'}- c'\log \ell)}^{\ell(m_{N'}- \frac{1}{ \sqrt{2d}}\log \ell-c_3)} \mathbb{P}(\hat{S}_{\ell}   \geq t)dt  \nonumber \\
&\geq  \ell(m_{N'}-c' \log \ell) - C \ell +  \Big(c' - \frac{1}{\sqrt{2d}} - \frac{c_3}{\log \ell}\Big) \big(1 - (1-c_4)^{\ell^{d\eta}}  \big) \ell\log \ell  \nonumber \\
&=  \ell \Big (m_{N'}-  \Big(\frac{1}{\sqrt{2d}}  +\frac{c_3}{\log \ell}\Big)  \log \ell \Big) - C \ell -  \Big(c' - \frac{1}{\sqrt{2d}} - \frac{c_3}{\log \ell}\Big) (1-c_4)^{\ell^{d\eta}}  \ell\log \ell  \nonumber \\
&\geq   \ell\Big(m_{N'} -  \Big(\frac{1}{\sqrt{2d}} + \sqrt{2d}\eta\Big) \log \ell\Big) - C\ell - \Big(c' - \frac{1}{\sqrt{2d}}\Big)   (1-c_4)^{\ell^{d\eta}}  \ell\log \ell  .
\end{align} }
Using \eqref{462} and \eqref{341}, for some constant $a>0$,
{\begin{align} \label{344}
m_N \leq  m_{N'}  + \sqrt{2d}\log ( 6K \ell^\eta) + \sqrt{2d}\log 2 \leq  m_{N'}  + \sqrt{2d} \eta \log \ell +a.
\end{align}}
Thus, applying \eqref{344} to \eqref{345}, taking $\iota>0$ such that  $\frac{1}{\sqrt{2d}} +  2 \sqrt{2d} \eta + \iota < \lambda_2$ (see \eqref{eta}), for sufficiently large $\ell$,
 \begin{align} \label{346}
\mathbb{E} \hat{S}_{\ell}  &\geq   \ell\Big(m_N -  \Big(    \frac{1}{\sqrt{2d}} +  2 \sqrt{2d} \eta \Big) \log \ell  - a \Big)  - C\ell - \Big(c' - \frac{1}{\sqrt{2d}}\Big)   (1-c_4)^{\ell^{d\eta}}  \ell\log \ell  \nonumber \\
&\geq   \ell\Big(m_N -  \Big(    \frac{1}{\sqrt{2d}} +  2 \sqrt{2d} \eta \Big) \log \ell  \Big)  -( C+a)\ell - \frac{\iota}{2}  \ell\log \ell   \geq \ell(m_N - \lambda_2 \log \ell).
\end{align}
On the other hand, by Lemmas \ref{lemma 3.5} and \ref{lemma 3.8},
\begin{align}  \label{347}
 \mathbb{E} \hat{S}_{\ell}    \leq \mathbb{E} S_{\ell} (\bar{\theta}_N^{N',K}) \leq  \mathbb{E} S_\ell(\phi_N) .
\end{align}
Thus, by \eqref{346} and \eqref{347}, proof of \eqref{364} is done.

\end{proof}

All that remains is to provide the outstanding proofs of Lemmas \ref{lemma 3.8}, \ref{lemma 3.9} and \ref{lemma 3.7}. Recalling that the last one appears in Section \ref{section 6}, we furnish the first two.

\begin{proof} [Proof of Lemma \ref{lemma 3.8}]
By \eqref{av0}, and the definition of the field $\bar{\theta}_N^{N',K}$,  
\begin{align} \label{3750}
\var  \bar{\theta}_{N,v}^{N',K}  = \var \pv\quad \forall v\in D.
\end{align}
{We  next prove that for a large constant $K\in \mathbb{N}$,}
\begin{align} \label{3760}
\cov  (  \bar{\theta}_{N,u}^{N',K},  \bar{\theta}_{N,v}^{N',K}) \geq \cov (\pu, \pv),\quad \forall u \neq v\in D.
\end{align} By Lemma \ref{lemma 3.3}, this will finish the proof. 
We consider two cases depending on whether $u$ and $v$ belong to the same or different $D_i$s.

\textbf{Case 1: $u\neq v$ belong to the same $D_i$.}
 Assume that $ u = 2KN'x_i+K u'$ and $ v = 2KN' x_i+K v' $ with $x_i \in \mathbb{Z}^d$ and $u'\neq v'\in V_{N'}$. {We lower bound $\cov (\theta_{D_i,u},\theta_{D_i,v})$ using  Lemma \ref{lemma 2.3} and the fact $|u'-v'|^{(N')} \leq |u'-v'  |  = |u-v|/K$. Also, we  lower bound $a_u$ and $a_v$ (defined in \eqref{av0}) using  Assumption \ref{a2}. Indeed,}
\begin{align} \label{351}
\cov  (  \bar{\theta}_{N,u}^{N',K},  \bar{\theta}_{N,v}^{N',K}) & \geq  \Big(\log  \frac{N'}{|u-v|/K}  - c_1  \Big) +( \var \phi_{N,u} -  \var \theta_{D_i,u})^{1/2} ( \var \phi_{N,v} -  \var \theta_{D_i,v})^{1/2} \nonumber \\
&\geq  \Big( \log \frac{N'}{|u-v|/K}  - c_1  \Big)+ {(\log N- \alpha(1/10) -  \log N'  )} \nonumber  \\
& =    \log \frac{N}{|u-v|} +\log K -c_1 - \alpha(1/10) ,
\end{align}
{where the second inequality follows from the fact that   $\var \phi_{N,u} ,\var \phi_{N,v}  \geq \log N  - \alpha(1/10)$ (recall $u,v\in V_N^{1/10}$) and $\var \theta_{D_i,u} =\var \theta_{D_i,v} = \log N'$.   Note that this explains why MBRW instead of BRW is used to construct the auxiliary field, since MBRW possesses more correlations.}
By Assumption \ref{a2} again,
\begin{align*}
\cov (\pu,\pv ) \leq \log \frac{N}{|u-v|}+ \alpha(1/10).
\end{align*}
Thus, we have \eqref{3760} for sufficiently large $K$.

\textbf{Case 2: $u\in D_i$ and $v\in D_j$ with $i\neq j$.}
By the independence of $\theta_{D_i}$ and $\theta_{D_j}$,
\begin{align} \label{352}
\cov  (  \bar{\theta}_{N,u}^{N',K},  \bar{\theta}_{N,v}^{N',K})  &=  ( \var \phi_{N,u} -  \var \theta_{D_i,u})^{1/2} ( \var \phi_{N,v} -  \var \theta_{D_j,v})^{1/2}  \nonumber \\
&\geq  \log N - \alpha(1/10)- \log N' .
\end{align}
Since $|u-v|\geq KN'$, by Assumption \ref{a2},
\begin{align*}
 \cov (\pu, \pv) \leq \log \frac{N}{KN'} + \alpha(1/10) = \log N - \log N' - \log K  + \alpha(1/10).
\end{align*} 
Thus, \eqref{3760} holds for sufficiently large $K$.
\end{proof}

Thus it remains to prove Lemma \ref{lemma 3.9}. A crucial input is Lemma \ref{lemma 2.8} which we will prove first. Since the latter is a statement about the maximum, we will employ comparison. The comparison needed for  \eqref{280} (between $\phi_N$ and the auxiliary construction) in Lemma \ref{lemma 2.8} is already accomplished in the just concluded proof. A very similar argument provides the corresponding comparison required for \eqref{2800} in  Lemma \ref{lemma 2.8} (between $\theta_N$ and an analogously constructed field).

Replacing $\phi_N$ by $\theta_N$ is the only change one needs to perform in the construction of $ \bar{\theta}_{N}^{N',K}$ described in the beginning of  subsection \ref{section aux} .  More precisely, setting
\begin{align}\label{353}
 \tilde{a}_v^2 := \var \theta_{N,v} -  \var \theta_{D(v),v} = \log N - \log N' \geq 0,
\end{align}
define the auxiliary field   $ \tilde{\theta}_{N}^{N',K} = (\tilde{\theta}_{N,v}^{N',K})_{v\in D}$ by
\begin{align*}
\tilde{\theta}_{N,v}^{N',K} := \theta_{D(v),v} + \tilde{a}_v Y.
\end{align*}
(In other words,  the only difference from the previously constructed field $ \bar{\theta}_{N}^{N',K}$ is the amount of perturbation $ \tilde{a}_v$.) 
Note that by construction, 
\begin{align*}
\var  \tilde{\theta}_{N,v}^{N',K}  = \var   \theta_{N,v}, \quad \forall v\in D.
\end{align*}
We now have the statement analogous to \eqref{3760}.

\begin{lemma} \label{lemma 3.11}
There exists a large enough constant $K \in \mathbb{N}$ such that 
\begin{align}  \label{285}
\textup{Cov}  ( \tilde{\theta}_{N,u}^{N',K},  \tilde{\theta}_{N,v}^{N',K}) \geq  \textup{Cov} (\theta_{N,u}, \theta_{N,v}),\quad \forall u \neq v\in D.
\end{align}
\end{lemma}
{The proof is essentially the same as that of Lemma \ref{lemma 3.8} and is thus moved to the Appendix (we refer to Section \ref{appendix b}).}

 ~
 
We are now in a position to finish the proof of Lemma \ref{lemma 2.8}, and we do so before closing off this section with the proof of Lemma \ref{lemma 3.9}. 
The basic idea in the proof of Lemma \ref{lemma 2.8} involves using  comparison with the constructed auxiliary fields and exploiting the i.i.d structure embedded in the latter. 

\begin{proof}[Proof of Lemma \ref{lemma 2.8}]

Let us first prove  \eqref{280}. 
{Consider the field $\bar{\theta}_N^{N',K}$ restricted to $\bar{D}:=\cup_{i \in \bar{I}} D_i \subseteq D$, where $\bar{I}:=  \{i \in I : D_i \subseteq B\}$.  We assume that $K \in \mathbb{N}$ is a  large constant such that   the covariance comparison \eqref{3760} holds for all $u,v\in \bar{D}$. }

 Assume that $N\geq 8KL$.   For any  $t'\in [ 2\sqrt{2d}\log (4KL),  2\sqrt{2d} \log (N/2)  ]$ (which is a valid interval for $N\geq 8KL$),  take  $N' = 2^{n'}$   with
\begin{align} \label{281}
n '=    \left \lfloor{ \log_2  ( N e^{-\frac{t'}{2\sqrt{2d}} } )  }\right \rfloor  .
\end{align} 
 The condition on $t'$ ensures that $n'\geq 1$ and $\frac{N}{L} \geq 4KN'$.
 Since $B$ is of size  $  \left \lfloor{\frac{N}{L} }\right \rfloor $,    using again the fact $\lfloor x \rfloor \geq x/2$ for $x\geq 1$,
\begin{align} \label{282}
|\bar{I}|\geq    \left \lfloor{\frac{  \lfloor N / L \rfloor}{2KN'} }\right \rfloor   ^d
 \geq  \frac{1}{(8KL)^d} e^{ \sqrt{\frac{d}{2}}\frac{t'}{2}}.
\end{align}  
 Also, by the definition of $a_v$ in \eqref{av0}, for any $v\in \bar{D}$,  
 \begin{align}
 |a_v^2 - (\log N - \log N'  ) |  \leq C.
\end{align} This  together with   \eqref{281} implies 
 \begin{align} \label{295}
 \Big\vert a_v^2 -\frac{t'}{2\sqrt{2d}}  \Big\vert \leq  C.
\end{align} 
{Also,  by \eqref{462} and \eqref{281},
\begin{align}  \label{290}
m_N \leq m_{N'} + \frac{t'}{2} + \sqrt{2d}\log 2 .
\end{align}}
Since each $\theta_{D_i}  \overset{\text{law}}{\sim} \theta_{N'}$,    setting $\bar{c} := c_3 + \sqrt{2d}\log 2 $, by  Lemma  \ref{lemma mbrw}, for each $i\in \bar{I}$,
\begin{align} \label{284}
\mathbb{P}\Big( \max_{v\in D_i } \theta_{D_i,v} \leq  m_{N} - \frac{t'}{2} -\bar{c} \Big) \overset{\eqref{290}}{ \leq } \mathbb{P}( \max _{v\in V_{N'} }  \theta_{N',v} \leq  m_{N'}   - c_3)  \leq  1-c_4.
\end{align}
Thus, using the independence of the fields $\theta_{D_i}$s ($i\in \bar{I}$) and noting that $t'\geq 1$, there exists  $c>0$ such that 
\begin{align*}
\mathbb{P}(\max_{v\in \bar{D}} \bar{\theta}_{N,v}^{N',K} &\leq  m_N -t' -\bar{c})\\
 &\overset{\eqref{295}}{\leq}  \mathbb{P}\Big(\max_{i\in \bar{I}} \max_{v\in  D_i } \theta_{D_i,v} \leq  m_{N} - \frac{t'}{2} -\bar{c} \Big) + \mathbb{P}\Big( \Big( \frac{t'}{2\sqrt{2d}} +C  \Big)^{1/2}Y\leq -\frac{t'}{2} \Big) \\
& \overset{\eqref{284}}{\leq}  (1-c_4)^{|\bar{I}|} +  Ce^{-c t'}  \overset{\eqref{282}}{\leq}  \exp\Big( -\frac{c}{(KL)^d}e^{ \sqrt{\frac{d}{2}}\frac{t'}{2} } \Big)+  Ce^{-ct'}.
\end{align*}
{Setting $t: =t'-2\sqrt{2d}\log L+ \bar{c}$, using the fact $e^x \geq 1+x$, we deduct that  there exist $t_0,c'>0$ such that   for $t\in [t_0, 2\sqrt{2d} (\log (N/2)- \log L)+\bar{c}]  $,
\begin{align} \label{287}
\mathbb{P}(\max_{v\in \bar{D}} \bar{\theta}_{N,v}^{N',K} &\leq  m_N -2\sqrt{2d}\log L - t  ) \leq Ce^{-c't}.
\end{align}}
By \eqref{3750} and \eqref{3760}, and Slepian's lemma (see Lemma \ref{lemma 3.4}), for any $M\in \R$,
\begin{align}  \label{283}
\mathbb{P}(\max_{v\in B} \phi_{N,v} \leq  M )\leq  \mathbb{P}(\max_{v\in \bar{D}} \phi_{N,v} \leq  M )\leq  
\mathbb{P}(\max_{v\in \bar{D}} \bar{\theta}_{N,v}^{N',K} \leq  M )  .
\end{align}  
{Therefore, by the preceding estimates, we get,}
for $t\in [t_0, 2\sqrt{2d} (\log (N/2) - \log L) + \bar{c}]  $,
\begin{align} \label{286}
\mathbb{P}(\max_{v\in B} \phi_{N,v} &\leq  m_N -2\sqrt{2d}\log L - t ) \leq Ce^{-c't}.
\end{align}
{In the case $t\geq 2\sqrt{2d} (\log (N/2)  - \log L)+\bar{c}$, noting that  $m_N -2\sqrt{2d}\log L - t $ is already a large negative quantity, even the one-point bound already provides an exponential upper bound. In fact,  by Assumption \ref{a1},    for any $w\in B$, 
\begin{align} \label{288}
\mathbb{P}(\max_{v\in B} \phi_{N,v} \leq  m_N -2\sqrt{2d}\log L - t ) &\leq \mathbb{P}( \phi_{N,w} \leq  m_N -2\sqrt{2d}\log L - t )  \nonumber \\
&\leq Ce^{-\frac{(2\sqrt{2d}\log L + t - m_N)^2}{2(\log N +\alpha_0)}} \leq Ce^{-c''t},
\end{align}
where  $c''>0$ a constant.}
Therefore, \eqref{286} and \eqref{288} finish the proof of \eqref{280}. The argument above verbatim with $\tilde{\theta}_{N,v}^{N',K}$ (see Lemma \ref{lemma 3.11}) replacing $\bar{\theta}_{N,v}^{N',K}$  yields \eqref{2800} with the required comparison delivered by \eqref{285}.

\end{proof}

We conclude this section by proving  Lemma \ref{lemma 3.9} using a simple union bound.
\begin{proof}[Proof of Lemma \ref{lemma 3.9}]
We first prove \eqref{3900}.  Assume that $N\geq 3c_2'\ell^{1/d}$ ($c_2'>0$ is a constant from Lemma \ref{lemma 2.8}).
 Take disjoint boxes $B_1,\cdots,B_\ell \subseteq V_N^{1/10}$ of size $ \left \lfloor{   N/ 3\ell^{1/d}  }\right \rfloor  
$. Note that it is possible to take such $\ell$ boxes since
$
\lfloor \frac{(8/10)N}{\lfloor N/3\ell^{1/d} \rfloor} \rfloor \geq \frac{1}{2} \cdot \frac{24}{10}\ell^{1/d} \geq  \ell^{1/d}.
$
 For each $i=1,\cdots,\ell$, define
\begin{align*}
M_i = \max_{v\in B_i} \pv.
\end{align*}
Then, $S_\ell(\phi_N)  \geq M_1 + \cdots + M_\ell$. Thus, by a union bound and  Lemma \ref{lemma 2.8},
there exists $t_1 > 0$ such that for any $t\geq t_1$,
\begin{align*}
\mathbb{P}\Big(  S_\ell(\phi_N)  \leq   \ell \Big( m_N-2\sqrt{\frac{2}{d}}\log \ell -t\Big)\Big )& \leq \mathbb{P}\Big(M_1+\cdots+M_\ell\leq   \ell \Big( m_N-2\sqrt{\frac{2}{d}}\log \ell -t\Big)\Big ) \\
&  \leq \sum_{i=1}^\ell \mathbb{P}\Big(M_i\leq  m_N-2\sqrt{\frac{2}{d}}\log \ell -t\Big)  \\
& = \sum_{i=1}^\ell \mathbb{P}\Big(M_i\leq  m_N- 2\sqrt{2d}\log (3\ell^{1/d}) + 2\sqrt{2d}\log 3 -t\Big) \\
&\leq C\ell e^{-c_2t}.
\end{align*}
The inequality  \eqref{391} follows in the same way with the aid of  Lemma \ref{lemma 2.8}.

\end{proof}

 \subsection{Proof of Theorem \ref{theorem 3.0} }

In this subsection, we prove Theorem \ref{theorem 3.0}. The proof bears resemblance to the proof of lower bound for $\mathbb{E}S_\ell$ in Proposition \ref{prop 3.6}. {\begin{enumerate}
\item We construct a field with an i.i.d structure much like our previous constructions, but now with blocks of $\phi_N$ instead of $\theta_N.$ This is done in a way so that the covariance structure of the construction dominates that of $\phi$ of the same size. In particular, this implies that $\mathbb{E}S_\ell$ of the auxiliary field is less than that of the original field.
\item Proof by contradiction: As indicated in Section \ref{section 1.3}, if \eqref{3000} fails, then the bootstrapping implies that the cardinality of level sets is large for the constructed field with overwhelming probability which is enough to force a lower  bound on $\mathbb{E} S_\ell$ for the latter.
\item This together with the upper bound on $\mathbb{E} S_\ell$ for the original field (\eqref{363} in  Proposition \ref{prop 3.6}) produces a contradiction.
\end{enumerate}}

We first proceed with the construction. As indicated, all the steps except one will be exactly the same as before. For  a small constant $\eta>0$ and a  large positive integer $K$ which will be chosen later,  let
\begin{align} \label{nprime}
N' :=\left \lfloor{\ell^\eta K N }\right \rfloor .
\end{align}  
 Consider the collection of boxes $\mathcal{D} = \{D_i\}_{i\in I}$ such that 
  \begin{enumerate}
 \item  $D_i \subseteq V_{N'}^{1/10}$,
 \item   $D_i  = 2KN x_i + KV_{N} $ for some $x_i \in \mathbb{Z}^d$.
\end{enumerate}  
 {As before, $D_i$s are obtained by dilating $V_{N}$ by $K$ and translating them by multiples of $2KN$ and checking whether they are fully contained in $V_{N'}^{1/10}.$  Then,  for large $N'$, } 
\begin{align} \label{350}
|I|\geq  \left \lfloor{\frac{(8/10) N'-1}{2KN} }\right \rfloor   ^d \geq    \left \lfloor{\frac{N'}{3KN} }\right \rfloor   ^d.
\end{align}
Note that by the second condition,  $d_\infty(D_i,D_j) \geq KN$ for $i\neq j$ and hence all the boxes are disjoint.  Set $D:= \cup_{i \in I} D_i$,  and for $v\in D$, denote $D(v)$ by a unique element  in $\mathcal{D}$ containing $v$.

We now come to the distinct step where $\{\phi_{D_i}\}_{D_i\in \mathcal{D}}$ is taken to be i.i.d. Gaussian fields defined by translating and dilating i.i.d. copies of $\phi_N$as follows:  $\phi_{D_i}$ is defined on  $D_i$ such that
\begin{align*}
( \phi_{D_i,2KNx_i+Kv'}  )_{v'\in V_N} \overset{\text{law}}{\sim} (\phi_{N,v'} )_{v'\in V_N}.
\end{align*} 

{Note that by  Assumption \ref{a1}  for any $v\in D$, $  \var \phi_{D(v),v} \leq \log N + \alpha_0$, and further by Assumption \ref{a2} and noting that $v\in  V_{N'}^{1/10}$,  $$\var \phi_{ N',v}  \geq \log N'  - \alpha(1/10)  \overset{\eqref{nprime}}{ \geq} \log N + \log K - \alpha(1/10).$$
 Hence, for large enough $K$, $\var \phi_{ N',v} >\var \phi_{D(v),v} $ for all $v\in D$. For any such $K$, define $a_v>0$ by}
 \begin{align} \label{av}
 a_v^2 := \var \phi_{ N',v} -  \var \phi_{D(v),v},\quad v\in D.
 \end{align}

Not surprisingly, by now, for a standard Gaussian random variable $Y$ independent of everything else,  we finally define the field $ \psi_{N}^{K,\ell,\eta}= ( \psi_{N,v}^{K,\ell,\eta})_{v\in D}$   by
\begin{align} \label{aux}
 \psi_{N,v}^{K,\ell,\eta} :=  \phi_{D(v),v} + a_vY.
\end{align}
Let 
{ \begin{align*}
\tilde{S}_\ell  := {S}_\ell(\{{\phi}_{D(v),v} \}_{v\in D} )  =  \max\Big\{\sum_{v\in J} \phi_{D(v),v}  : J \subseteq D, |J|=\ell \Big\}.
 \end{align*}
} 
 
Similar to Lemma \ref{lemma 3.8}, we first state that the desired comparison indeed holds.   
 \begin{lemma} \label{lemma 4.5} 
There exists a large constant $K\in \mathbb{N}$ such that for any  $\eta>0$ and $N, \ell\geq 1$,
\begin{align}
\mathbb{E}  S_\ell (  \psi_{N}^{K,\ell,\eta} ) \leq \mathbb{E} S_{\ell}(\phi_{N'}).
\end{align}
\end{lemma}
Given Lemma  \ref{lemma 4.5} whose proof is again a computational one which we momentarily postpone, one can quickly conclude the  proof of  Theorem \ref{theorem 3.0}.

\begin{proof} [Proof of   Theorem \ref{theorem 3.0}]
Let  $\e'\in (0,\frac{1}{8})$ be any constant.  Suppose that for some $\ell\geq 1$ and $N\geq c_5'\ell^{1/d}$ ($c_5'>0$ is the constant from Lemma \ref{lemma 3.9}),
 \begin{align}  \label{465}
 \mathbb{P}\Big(\Big\vert \Gamma_N \Big(\frac{1-8\e'}{\sqrt{2d}}\log \ell\Big ) \Big\vert   \geq  \ell \Big) > \ell^{-\e'}.
 \end{align}
Thus, we also have
\begin{align} \label{463}
\mathbb{P}\Big(S_\ell(\phi_N)  \geq \ell\Big(m_N - \frac{1-8\e'}{\sqrt{2d}}\log \ell\Big)\Big) >  \ell^{-\e'}.
\end{align}

Consider now the field  $  \psi_{N}^{K,\ell,\eta}$ with
$
\eta := \frac{3\e'}{d}
$ 
 and a   large constant $K\in \mathbb{N}$  for which  Lemma \ref{lemma 4.5} holds.
Then,
\begin{align} \label{306}
|I|  \overset{\eqref{350}}{\geq}    \left \lfloor{\frac{N'}{3KN} }\right \rfloor   ^d
\overset{\eqref{nprime}}{ \geq } \frac{1}{6^d}\ell^{3\e'}.
\end{align}
 Thus,  by the bootstrapping effect, due to the independence of the fields $\phi_{D_i}$s ($i\in I$) and since $ \phi_{D_i} \overset{\text{law}}{\sim} \phi_{N}$, there exists $c>0$ such that
 \begin{align} \label{303}
  \mathbb{P}\Big( \tilde{S}_\ell   < \ell\Big ( m_N-\frac{1-8\e'}{\sqrt{2d}}\log \ell \Big) \Big) &\leq  \mathbb{P}\Big(\max_{i\in I} S_\ell(\phi_{D_i}) <  \ell \Big( m_N-\frac{1-8\e'}{\sqrt{2d}}\log \ell \Big) \Big)   \nonumber\\
  & \overset{\eqref{463}}{\leq}  (1-\ell^{-{\e'}})^{|I|} \overset{\eqref{306}}{ \leq} e^{- c \ell^{2 \e'}}.
\end{align}    

Next, since  $ \tilde{S}_\ell \geq    S_\ell( \phi_{D_i})$  for any $i \in I$ and $ \phi_{D_i} \overset{\text{law}}{\sim} \phi_{N}$, for any $M\in \R$,   
$$
  \mathbb{P}(\tilde{S}_\ell   \leq  M) \leq   \mathbb{P}(  S_\ell(\phi_N)  \leq   M ).
$$ Thus,
by Lemma \ref{lemma 3.9},   for large enough $\ell$,
   \begin{align} \label{304}
     \int_{-\infty}^{\ell(m_N-(2\sqrt{\frac{2}{d}} + \frac{1}{c_5})\log \ell)} \mathbb{P}( \tilde{S}_\ell    \leq  t)dt   &\leq   \ell \int_0^\infty    \mathbb{P}\Big(  S_\ell(\phi_N)  \leq   \ell \Big( m_N-\Big(2\sqrt{\frac{2}{d}} + \frac{1}{c_5}\Big)\log \ell -s \Big)\Big)  ds \nonumber \\
     &\leq  C \ell \int_0^\infty \ell  e^{-c_5 (  \frac{1}{c_5} \log \ell + s)}ds \leq C\ell.
   \end{align}
Define $c' :=  2\sqrt{\frac{2}{d}} + \frac{1}{c_5}$.
Using \eqref{useful} along with the above bounds, we get
\begin{align} \label{302}
\mathbb{E} \tilde{S}_\ell  
&\geq  \ell(m_N-c' \log \ell) - \int_{-\infty}^{ \ell(m_N- c'\log \ell) }  \mathbb{P}(\tilde{S}_\ell   \leq  t)dt +  \int_{\ell(m_N- c'\log \ell)}^{\ell(m_N- \frac{1-8\e'}{ \sqrt{2d}}\log \ell)} \mathbb{P}( \tilde{S}_\ell    \geq t)dt  \nonumber \\
&\geq  \ell(m_N-c' \log \ell) - C \ell +  \Big(c' - \frac{1-8\e'}{\sqrt{2d}}\Big) \ell\log \ell  \cdot  \big(1 - e^{- c \ell^{2 \e'}}  \big) \nonumber \\
&=  \ell\Big(m_N - \frac{1-8\e'}{\sqrt{2d}}\log \ell\Big) - C\ell -  \Big(c' - \frac{1-8\e'}{\sqrt{2d}}\Big) e^{- c \ell^{2 \e'}}   \ell\log \ell   .
\end{align}
 On the other hand, by Proposition \ref{prop 3.6}, recalling $N' =\left \lfloor{\ell^\eta K_1 N }\right \rfloor  $ and $\eta =  \frac{3\e'}{d}$,  for large enough $\ell$,
\begin{align} \label{301}
\mathbb{E} S_{\ell}(\phi_{N'})  &\leq   \ell\Big(m_{N'}- \frac{1-\e'}{\sqrt{2d}} \log \ell\Big) \nonumber \\
&\overset{\eqref{462}}{\leq}  \ell\Big(m_{N} -\Big (\frac{1-\e'}{\sqrt{2d}} -\sqrt{2d}\eta\Big)\log \ell+C\Big) = \ell\Big(m_N - \frac{1-7\e'}{\sqrt{2d}}\log \ell\Big) + C\ell.
\end{align}
In addition, by Lemma \ref{lemma 3.5} and \ref{lemma 4.5},
\begin{align} \label{305}
 \mathbb{E} \tilde{S}_\ell  \leq \mathbb{E} S_\ell (  \psi_{N}^{K,\ell,\eta} )   \leq  \mathbb{E} S_{\ell}(\phi_{N'}).
\end{align}
  \eqref{302} together with \eqref{301} contradict \eqref{305} for large  enough $\ell$.  {We conclude that for any  $\e'\in (0,\frac{1}{8})$, there exists $\ell_0$ such that for  any $\ell \geq \ell_0$ and  $N\geq c_5'\ell^{1/d}$,  \eqref{465} is false, i.e.,}
\begin{align*}
\mathbb{P}\Big(\Big\vert \Gamma_N \Big(\frac{1-8\e'}{\sqrt{2d}}\log \ell\Big ) \Big\vert   \geq  \ell \Big) \leq  \ell^{-\e'}.
\end{align*}
Setting $t:=\frac{1-8\e'}{\sqrt{2d}}\log \ell$ and $\e :=  \sqrt{2d} ( \frac{1}{1-8\e'}-1 )$, one can deduce that for any constant $\e >0$, for sufficiently large $t$ and $N\geq c_5' e^{(\sqrt{2d}+\e)t/d}$,
\begin{align*}
\mathbb{P}( |\Gamma_N(t)| \geq  e^{(\sqrt{2d}+\e)t}) \leq  e^{-\e'  ( \sqrt{2d}+\e) t}   = e^{-\e t/8}.
\end{align*}

\end{proof}

We now conclude this section by providing the outstanding proof of Lemma \ref{lemma 4.5} which again simply involves comparing covariances.

\begin{proof}   [Proof of Lemma \ref{lemma 4.5}]
By \eqref{av}, for $v\in D$,
\begin{align} \label{910}
\var \psi_{N,v}^{K,\ell,\eta} =  \var \phi_{ N',v}.
\end{align}
By Lemma \ref{lemma 3.3}, it remains to show that there exists a  large constant $K\in \mathbb{N}$ such that  for $u\neq v\in D$,
\begin{align} \label{911}
\cov ( \psi_{N,u}^{K,\ell,\eta}, \psi_{N,v}^{K,\ell,\eta})  \geq \cov ( \phi_{N',u}, \phi_{N',v}).
\end{align}

\textbf{Case 1: $u\neq v \in D$ belong to the same $D_i$.} Assume that $ u = 2KNx_i+K u' $ and $ v = 2KNx_i+K v' $ with $x_i \in \mathbb{Z}^d$ and $ u'\neq v' \in V_N$. Since $|u'-v' |  = |u-v|/K$, by Assumptions \ref{a1} (\eqref{aa} applied to $\phi_N$) and \ref{a2} and \eqref{av},
\begin{align} \label{912}
\cov ( \psi_{N,u}^{K,\ell,\eta}, \psi_{N,v}^{K,\ell,\eta})  &= \cov(\phi_{N,u'} ,\phi_{N,v'} ) + a_ua_v \nonumber \\
&\geq   \max\{\var \phi_{N,u'},\var \phi_{N,v'}\} - \log |u-v  |  + \log K - 2\alpha_0\nonumber \\
&+ (\log N' - \alpha(1/10) - \var \phi_{N,u'})^{1/2} (\log N'  - \alpha(1/10) - \var \phi_{N,v'})^{1/2} \nonumber \\
&\geq  - \log |u-v  |  + \log K - 2\alpha_0 + \log N' - \alpha(1/10)  .
\end{align}
On the other hand,
since $u,v\in V_{N'}^{1/10}$, by Assumption \ref{a2},
\begin{align} \label{913}
 \cov ( \phi_{N',u}, \phi_{N',v}) \leq  \log  \frac{N'}{|u-v|} + \alpha(1/10).
\end{align}
Thus  \eqref{911} holds  for sufficiently large $K$.

\textbf{Case 2: $u\in D_i$ and $v\in D_j$ with $i\neq j$.}\
By Assumptions \ref{a1} and \ref{a2},
\begin{align} \label{915}
\cov ( \psi_{N,u}^{K,\ell,\eta}, \psi_{N,v}^{K,\ell,\eta})   
&= ( \var \phi_{N',u} -  \var \phi_{D_i,u})^{1/2} ( \var \phi_{N',v} -  \var \phi_{D_j,v})^{1/2} \nonumber \\
&\geq \log   N'  -  \alpha(1/10)  - \log N - \alpha_0.
\end{align}
Since $|u-v| \geq KN$, by Assumption \ref{a2},
\begin{align} \label{916}
\cov(  \phi_{N',u}, \phi_{N',v})  \leq \log N' - \log (KN) + \alpha(1/10).
\end{align}
Thus, again, we have \eqref{911} for sufficiently large $K$.
Hence, $
\mathbb{E}S_\ell ( \psi_{N}^{K,\ell,\eta}) \leq  \mathbb{E}S_\ell (\phi_{N'} |_D) \leq   \mathbb{E}S_\ell (\phi_{N'} ).
$
\end{proof}

\subsection{Lower bound on the level set} \label{section 3.3}
In this section, we prove Theorem \ref{theorem lower}.

\begin{proof}[Proof of Theorem \ref{theorem lower}]
{The proof consists of the following broad steps.
\begin{enumerate}
\item Using the upper bound (Theorem \ref{theorem 3.0}), we upper bound the right tail of  $S_\ell(\phi_N)$. This together with a control on the maximum offered by  Lemma \ref{lemma 2.4}, allow to  upper bound the  contribution $\mathbb{E}[ S_\ell (\phi_N): S_\ell (\phi_N) > \ell(m_N - (\frac{1}{\sqrt{2d}} - \kappa) \log \ell) ]$  ($\kappa>0$).

\item  The lower bound on  $\mathbb{E}S_\ell (\phi_N)$ (\eqref{364} in  Proposition \ref{prop 3.6}) together with (1) provide a lower bound  for the probability $ \mathbb{P}(  S_\ell(\phi_N)  \geq  \ell(m_N - (\frac{1}{\sqrt{2d}} + \kappa) \log \ell) ) $  ($\kappa>0$).

\item The lower bound on the probability of $S_{\ell}(\phi_N)$ being large along with again using the control on the maximum is enough to transfer this to a lower bound on $\Gamma_{N}(t)$ being large with $t:= (\frac{1}{\sqrt{2d}} + \kappa) \log \ell$. 
\end{enumerate}
}

\textbf{Step 1.} We first show that for any small enough constant $\e>0$, for large enough $\ell$, for sufficiently large $N$,
\begin{align} \label{397}
\mathbb{P}\Big(S_\ell (\phi_N) > \ell m_N - \frac{1}{\sqrt{2d}+4\e}\ell \log \ell\Big)\leq  2  \ell^{-0.1 \e / \sqrt{2d}}.
\end{align}
Setting 
\begin{align}\label{k0}
s:=  \frac{1}{\sqrt{2d}+2\e} \log \ell,
\end{align} 
 define  the events
\begin{align*}
A:= \{ |\Gamma_N( s) | \leq   e^{ (\sqrt{2d}+\e )s} \}
\end{align*}
and
\begin{align} \label{402}
B:= \Big\{\max_{v\in V_N} \phi_{N,v} <  m_N +  \e  \log \ell  \Big\}.
\end{align}
{For any small enough $\e>0$, by Theorem \ref{theorem 3.0}, for large enough $s$ (hence for large enough $\ell$)},
\begin{align} \label{395}
\mathbb{P}(A^c) \leq  e^{- \e s/8}  \leq   \ell^{-0.1 \e / \sqrt{2d}}.
\end{align} 
Also, by  Lemma \ref{lemma 2.4}, for large enough $\ell$,
\begin{align} \label{396}
\mathbb{P}(B^c) \leq  C \e \log \ell \cdot   e^{- \sqrt{2d} \e \log \ell} \leq \ell^{-0.1 \e / \sqrt{2d}}.
\end{align}
In addition, under the event $A\cap B$, 
\begin{align*}
S_\ell (\phi_N) & \leq    e^{ (\sqrt{2d}+\e )s}  (m_N + \e \log \ell) + (\ell-   e^{ (\sqrt{2d}+\e )s} ) (m_N -s)  \\
&\overset{\eqref{k0}}{=}   \ell^{1-\frac{\e}{\sqrt{2d}+2\e} } (m_N + \e \log \ell) + (\ell-  \ell^{1-\frac{\e}{\sqrt{2d}+2\e} }) \Big (m_N - \frac{1}{\sqrt{2d}+2\e}\log \ell  \Big)   \\
&= \ell m_N  -  \frac{1}{\sqrt{2d}+2\e}\ell \log \ell +  \Big ( \e +  \frac{1}{\sqrt{2d}+2\e} \Big)  \ell^{1-\frac{\e}{\sqrt{2d}+2\e}} \log \ell \\
&\leq \ell m_N - \frac{1}{\sqrt{2d}+4\e}\ell \log \ell
\end{align*}
for large enough $\ell$.
Thus, by  \eqref{395} and \eqref{396}, we obtain \eqref{397}.

\textbf{Step 2.} We establish that  for any $\kappa,\iota>0$, for large enough $\ell$,
\begin{align} \label{405}
 \mathbb{P}\Big( S_\ell (\phi_N) \leq  \ell \Big (m_N - \Big(\frac{1}{\sqrt{2d}} +\kappa \Big)\log \ell\Big ) \Big )  \leq  \iota.
\end{align}
For a small enough $\gamma\in (0,1)$ which will be chosen later, let
\begin{align}
p&:= \mathbb{P}\Big( S_\ell (\phi_N) \leq  \ell \Big (m_N - \Big(\frac{1}{\sqrt{2d}} +\kappa \Big)\log \ell\Big ) \Big ),  \label{406} \\
q_1&:= \mathbb{P} \Big(  \ell \Big (m_N - \frac{1}{\sqrt{2d}+4\gamma} \log \ell\Big)  < S_\ell (\phi_N) \leq   \ell (m_N + \gamma \log \ell) \Big), \nonumber \\ 
q_2&:= \mathbb{P}(S_\ell (\phi_N)  > \ell (m_N + \gamma \log \ell)) \nonumber.
\end{align}
Then,
\begin{align*}
\mathbb{P} \Big( \ell \Big (m_N - \Big(\frac{1}{\sqrt{2d}} +\kappa \Big)\log \ell\Big) < S_\ell (\phi_N)  \leq  \ell \Big (m_N - \frac{1}{\sqrt{2d}+4\gamma}\log \ell\Big)\Big) = 1-p-q_1-q_2.
\end{align*}
Also, by \eqref{397},
\begin{align} \label{399}
q_1+q_2\leq   2\ell^{-0.1 \gamma / \sqrt{2d}}.
\end{align}
We write
\begin{align} \label{398}
\mathbb{E} S_\ell (\phi_N) & = \mathbb{E}    \Big[ S_\ell (\phi_N)  :   S_\ell (\phi_N) \leq  \ell \Big (m_N - \Big(\frac{1}{\sqrt{2d}} +\kappa \Big)\log \ell \Big)  \Big] \nonumber \\
& +   \mathbb{E}   \Big [ S_\ell (\phi_N)  :  \ell \Big (m_N - \Big(\frac{1}{\sqrt{2d}} +\kappa \Big)\log \ell\Big) < S_\ell (\phi_N)  \leq  \ell \Big (m_N - \frac{1}{\sqrt{2d}+4\gamma}\log \ell\Big)  \Big ]  \nonumber \\
&+ \mathbb{E}  \Big  [ S_\ell (\phi_N)  :  \ell \Big (m_N - \frac{1}{\sqrt{2d}+4\gamma} \log \ell\Big)  <  S_\ell (\phi_N) \leq   \ell (m_N + \gamma  \log \ell)  \Big]  \nonumber \\
&+  \mathbb{E}   [ S_\ell (\phi_N)  : \ell (m_N +\gamma  \log \ell) < S_\ell (\phi_N)    ] .
\end{align}
We bound each term above. The first term is bounded by
\begin{align} \label{3990}
p \ell \Big (m_N - \Big(\frac{1}{\sqrt{2d}} +\kappa \Big)\log \ell\Big).
\end{align}
The second term is bounded by
\begin{align}  
(1-p-q_1-q_2)  \ell \Big (m_N - \frac{1}{\sqrt{2d}+4\gamma}\log \ell\Big) .
\end{align}
The third term is bounded by
\begin{align}
q_1\ell (m_N + \gamma \log \ell) .
\end{align}
Note that  $S_\ell (\phi_N) \geq \ell(m_N+x)$ implies $\max_{v\in V_N} \pv \geq m_N+x$. Thus, using Lemma \ref{lemma 2.4}, for large enough $\ell$,
the fourth term is bounded by
\begin{align} \label{3992}
& \ell (m_N + \gamma \log \ell)  \mathbb{P}(S_\ell(\phi_N) > \ell (m_N + \gamma   \log \ell) )  +  \int_{\ell (m_N + \gamma   \log \ell)}^\infty \mathbb{P}( S_\ell (\phi_N) \geq t)dt \nonumber \\
&\leq q_2 \ell (m_N + \gamma  \log \ell)  +\ell \int_0^\infty  \mathbb{P} (\max_{v\in V_N} \pv \geq m_N + \gamma \log \ell+s)ds \nonumber \\
&\leq  q_2 \ell (m_N + \gamma \log \ell)  + \ell \int_0^\infty C(\gamma \log \ell + s)e^{- \sqrt{2d} (\gamma \log \ell + s)}ds \nonumber \\
& \leq q_2 \ell (m_N + \gamma \log \ell)  + C   ( \log \ell ) \ell^{1-\sqrt{2d}\gamma} \leq    q_2 \ell (m_N + \gamma \log \ell)  +  \ell^{1-\gamma}.
\end{align}
Therefore, applying \eqref{3990}-\eqref{3992} to \eqref{398},
\begin{align} \label{385}
\mathbb{E} S_\ell (\phi_N) &\leq p \ell \Big (m_N -\Big (\frac{1}{\sqrt{2d}} +\kappa \Big)\log \ell \Big )  +(1-p- q_1-q_2) \ell \Big(m_N - \frac{1}{\sqrt{2d}+4\gamma} \log \ell \Big)\nonumber   \\
&+q_1 \ell (m_N + \gamma  \log \ell)  + q_2 \ell (m_N + \gamma  \log \ell)  + \ell^{1-\gamma} \nonumber  \\ 
&=  \ell m_N - \ell \log \ell \Big( p \Big(\frac{1}{\sqrt{2d}} +\kappa  - \frac{1}{\sqrt{2d}+4\gamma}   \Big)  + \frac{1}{\sqrt{2d}+4\gamma} - (q_1 +q_2)\Big( \frac{1}{\sqrt{2d}+4\gamma} +\gamma \Big)\Big) \nonumber \\
&+ \ell^{1-\gamma}.
\end{align}
By Proposition \ref{prop 3.6}, for large enough $\ell$, for sufficiently large $N$,
\begin{align} \label{386}
\mathbb{E}S_\ell (\phi_N) \geq  \ell\Big(m_N - \Big (\frac{1}{\sqrt{2d}} + \gamma\Big) \log \ell\Big).
\end{align}
Rearranging, {using the fact $  \frac{1}{\sqrt{2d}+4\gamma} \geq \frac{1}{\sqrt{2d}}-\frac{2\gamma}{d} $,} for large enough $\ell$,
\begin{align*}
p\Big( \frac{1}{\sqrt{2d}} +\kappa  -  \frac{1}{\sqrt{2d}+4\gamma} \Big)& \leq \frac{1}{\sqrt{2d}} + \gamma  + (q_1+q_2) \Big ( \frac{1}{\sqrt{2d}+4\gamma} +\gamma  \Big) - \frac{1}{\sqrt{2d}+4\gamma}  + \frac{1}{\ell^\gamma \log \ell} \\
&{\overset{\eqref{399}}{\leq}   \frac{1}{\sqrt{2d}} +\gamma  + 2\ell^{-0.1 \gamma / \sqrt{2d}}    \Big(  \frac{1}{\sqrt{2d}+4\gamma} +\gamma   \Big)-  \Big(\frac{1}{\sqrt{2d}} - \frac{2\gamma}{d} \Big)} + \gamma\\
&\leq  C(\ell^{-0.1 \gamma / \sqrt{2d}} +\gamma).
\end{align*}
Therefore,
\begin{align} \label{404}
p \leq \frac{C}{\kappa}( \ell^{-0.1 \gamma / \sqrt{2d}} + \gamma).
\end{align}
Taking $\gamma>0$ sufficiently small, we obtain \eqref{405}.

\textbf{Step 3.}
We show that if the events $B$ (defined in \eqref{402}) and
\begin{align} \label{403}
  S_\ell (\phi_N) \geq \ell \Big (m_N -  \frac{1}{\sqrt{2d}-\e}\log \ell\Big )  
\end{align}
hold, then  for large enough $\ell$,
\begin{align} \label{401}
\Big\vert   \Big\{v\in V_N: \phi_{N,v} \geq m_N - \frac{1}{\sqrt{2d}-2\e}\log \ell \Big\}\Big\vert \geq  \ell^{\frac{\sqrt{2d}-\e}{\sqrt{2d}-0.5\e}}.
\end{align}
Indeed, assuming the complement of the above event, then under the event $B$,
\begin{align*}
S_\ell (\phi_N) &\leq  \ell^{\frac{\sqrt{2d}-\e}{\sqrt{2d}-0.5\e}} (m_N + \e \log \ell) + (\ell - \ell^{\frac{\sqrt{2d}-\e}{\sqrt{2d}-0.5\e}})  \Big(m_N - \frac{1}{\sqrt{2d}-2\e}\log \ell \Big) \\
&= \ell m_N -  \frac{1}{\sqrt{2d}-2\e}\ell \log \ell + \frac{2}{\sqrt{2d}} \ell^{\frac{\sqrt{2d}-\e}{\sqrt{2d}-0.5\e}} \log \ell  <   \ell m_N -  \frac{1}{\sqrt{2d}-\e}\ell \log \ell  ,
\end{align*}
which contradicts \eqref{403}.

Therefore,  by \eqref{396} and \eqref{405}, for any $\iota>0$,  for large enough $\ell$,
\begin{align*}
\mathbb{P} \Big(\Big\vert   \Big\{v\in V_N: \phi_{N,v} \geq m_N - \frac{1}{\sqrt{2d}-2\e}\log \ell \Big\}\Big\vert \geq  \ell^{\frac{\sqrt{2d}-\e}{\sqrt{2d}-0.5\e}}\Big) \geq  1 - \iota.
\end{align*}
Setting $t:= \frac{1}{\sqrt{2d}-2\e}\log \ell$, 
\begin{align*}
\mathbb{P}(|\Gamma_N(t) | \geq  e^{ (\sqrt{2d}-3\e )t} )  \geq   1-\iota .
\end{align*}
By the arbitrariness of $\e>0$, we conclude the proof.

\end{proof}

 {We conclude this section by quoting the result from \cite[Lemma 3.3]{drz}, which states that no two points in $\Gamma_N(t)$ are mesoscopically separated.}

\begin{lemma} \label{lemma 4.4}
For any $\delta>0$, there exists a constant $c = c(\delta)>0$
 such that for any centered Gaussian field $\phi_N$ satisfying Assumption \ref{a2},
\begin{align}
\lim_{r \ri} \limsup_{N\ri} \mathbb{P}\Big(\exists u,v\in V_N^\delta \ \textup{such that} \  \phi_{N,u},\phi_{N,v} > m_N- c\log \log r,  |u-v| \in \Big(r,\frac{N}{r}\Big)\Big) = 0.
\end{align}
\end{lemma}

Combining this lemma with  \eqref{410},  under Assumptions \ref{a1} and \ref{a2},  for any $\lambda>0$,
\begin{align} \label{meso}
 \lim_{r\ri} \limsup_{N\ri} \mathbb{P}\Big(\exists u,v\in V_N \ \textup{such that} \  \phi_{N,u},\phi_{N,v}  \geq  m_N - \lambda,  |u-v| \in \Big[ r,\frac{N}{r}\Big]\Big) = 0.
\end{align}

\section{Point process convergence} \label{section 4}
 
Recalling from Section \ref{section 1.3} that the strategy for the proof of Theorem \ref{theorem 1.1} has several steps, we develop them in this section. { As indicated earlier, and is presented later in Lemma \ref{lemma 7.7}, we upper bound the Gibbs weight of $\Gamma^c_{N}(t)$.}  
For the remaining steps, let us begin by recalling some of the key notations already introduced in Section \ref{section 1.3}.
 $C_{N,r}$ denotes the local extreme points
  \begin{align}\label{local ext}
C_{N,r}:=\{v\in V_N: \phi_{N,v}=\max_{u\in B(v,r)} \phi_{N,u}\}
\end{align}

We next define the local cluster weight.
For $R>0$, and $v \in V_N,$
{\begin{align} \label{sbar}
\bar{S}_{v,R} := \sum_{u\in B(v,R)} e^{\beta(\pu - \pv)}.
\end{align}
Note that 
$
\mu_N^\beta( B(v,R)) \propto e^{\beta(\pv-m_N)}\cdot \bar{S}_{v,R}.$
We next define the central point process (or a random measure) that will feature in our analysis.
 \begin{align} \label{point process}
\eta_{N,r}: = \sum_{v\in C_{N,r} }  \delta_{\pv - m_N} (dy) \otimes \delta_ {\bar{S}_{v,r/2} } (dz),
\end{align}
Note that above we only consider the cluster weight of balls of radius $r/2$ to ensure that the balls $B(v,r/2)$ are disjoint for $v\in C_{N,r}$.   }

~

The main result towards the proof of Theorem \ref{theorem 1.1} is the convergence of the point process $\{\eta_{N,r_N}\}_{N\geq 1}$ (with a suitable sequence  $\{r_N\}_{N\geq 1}$) to a Cox process where the  random intensity measure exhibits a certain key tensorization. {To talk about such convergences, we endow the space of measures on $ \R\times [0,\infty)$ with the topology of vague convergence where a sequence of measures $\{\mu_N\}_{N\geq 1}$ on the Polish space $X$ converges vaguely to the measure $\mu$ if $\int_X f d\mu_N \rightarrow \int_X f d\mu$ as $N\ri$ for all continuous  and compactly supported real valued functions $f$ on $X$.} 

Throughout this section, we assume that the centered
 Gaussian field $\phi_N$ satisfies  Assumptions \ref{a1} and \ref{a2}. Also, for a sigma-finite  measure $\mu$,  denote by $\text{PPP}(\mu)$ the Poisson point process with  the intensity measure $\mu$. Given the above preparation, the following is the main convergence result of this section.

\begin{theorem} \label{theorem 4.1}
  Let $\{r_N\}_{N\geq 1}$ be any sequence such that $r_N\ri$ and $r_N/N\rightarrow 0$ as $N\rightarrow \infty$. Then, any subsequential limit of  $\{\eta_{N,r_N}\}_{N\geq 1}$ can be written as
\begin{align}
  \textup{PPP}( e^{-\sqrt{2d}y}dy \otimes \nu (dz))
\end{align}
for some random Borel measure $\nu$ on $[0,\infty)$ such that $\nu((0,\infty)) \in (0,\infty)$ almost surely.
 
\end{theorem}
 
As in Section \ref{section 1.3}, it is worth drawing the reader's attention again to the fact that in the specific  case of 2D GFF,  Biskup-Louidor \cite{bl2} established the significantly stronger result about the convergence of  
 \begin{align} \label{augment}
\bar{\eta}_{N,r_N}:=    \sum_{v\in C_{N,r_N} }  \delta_{v/N}  (dx) \otimes  \delta_{\pv - m_N} (dy) \otimes \delta_ { \{\phi_{N,v+x}-\pv\}_{x\in \mathbb{Z}^2} } (dz)
\end{align} 
 to the Cox process $\text{PPP}(\mathcal{Z}(dx)\otimes e^{-\sqrt{2d}y}dy \otimes {\mathrm{v}}(dz))$ for an explicit random measure $\mathcal{Z}$  and a \emph{deterministic} measure $\mathrm v$ on the space height functions on $\Z^2$ pinned to be $0$ at the origin. 
In particular, the Gibbs-Markov property of 2D GFF is heavily relied upon to prove a tensorization of the cluster from the first two coordinates.
However, in the setting of general log-correlated fields $\phi_N$, we have to be content with one of two things. We can either forego the spatial location (the first coordinate), which leads to Theorem \ref{theorem 4.1}, which though weaker, is fortunately \emph{enough} to deduce our main result, Theorem \ref{theorem 1.1}. The other possibility is to project on the first two coordinates and disregard the third coordinate which will encode the spatial locations of the local extrema  and their values but not the local cluster weights around them. A similar Cox process limit for this can also be established and  Remark \ref{spatial} expands on this.

The proof of Theorem \ref{theorem 4.1} follows the same program as in \cite{bl1} \emph{but often our key input is the apriori control on the level set obtained in Theorem \ref{theorem 3.0}, whereas for the 2D GFF, it was the Gibbs-Markov property.}

The program has two broad parts.
\begin{enumerate}
\item {Establishing the invariance of the limit points of \eqref{point process} under a certain Brownian flow.}
\item Using the abstract theory of Liggett in \cite{liggett}, characterize the limit points as  Cox processes, as well as establishing the desirable tensorization properties of the intensity measure. 
\end{enumerate}   

We first elaborate on (1). Let $B_t$ be a standard one-dimensional Brownian motion. For any  measurable function  $f: \R \times [0,\infty) \rightarrow [0,\infty)$   and $t\geq 0$, define
\begin{align} \label{700}
f_t(y,z) :=  -\log \mathbb{E}(e^{-f(y+B_t-  \sqrt{d/2} t  ,z ) }),
\end{align}
In other words, for the Markov kernel $P_t$, defined as
\begin{align}\label{markov kernel}
P_t((y,z), A) := \mathbb{P}((y+B_t - \sqrt{d/2}  \ t, z) \in A) ,\quad y\in \R,  \ z\in [0,\infty),  \  A\subseteq \R\times  [0,\infty) \  \text{Borel},
\end{align}
$f_t$ above, also can be written as
\begin{align*}
f_t(y,z) =  - \log [(P_t e^{-f})(y,z)].
\end{align*}

Given the above notational setup, we are now is a position to state the main invariance principle we prove for general log-correlated Gaussian fields.

\begin{theorem} \label{theorem 7.1}
Suppose  that $\eta$ is any subsequential limit of $\{\eta_{N,r_N}\}_{N\geq 1}$  for $\{r_N\}_{N\geq 1}$ with $r_N\rightarrow \infty$ and $r_N/N \rightarrow 0$ as $N\ri$. Then, for any continuous function $f: \R \times [0,\infty) \rightarrow [0,\infty)$ with compact support and $t\geq  0$,
\begin{align}\label{710}
\mathbb{E}(e^{-<\eta,f>}) = \mathbb{E}(e^{-<\eta,f_t>}).
\end{align}
\end{theorem}

The Brownian flow in Theorem \ref{theorem 7.1} is obtained by perturbing the underlying field $\phi_N$ along an Ornstein-Uhlenbeck flow.
More precisely, let $\phi'_N$ and $\phi''_N$ be independent copies of $\phi_{N}$.  For $t>0$, define
\begin{align} \label{ou0}
\hat{\phi}'_N: = \sqrt{1-\frac{t}{\log N}}\phi'_N,\quad \hat{\phi}''_N := \sqrt{\frac{t}{\log N}}\phi''_N.
\end{align}
Then, the field $\hat{\phi}'_N + \hat{\phi}''_N$ has the same law as $\phi_N$ and for notational brevity, we we will denote the former by latter, i.e.,
\begin{align} \label{ou}
\phi_N   :=  \hat{\phi}'_N + \hat{\phi}''_N.
\end{align}

Define $C_{N,r}$ and $C_{N,r}'$  to be the sets the $r$-local extrema of $\phi_N$ and $\phi_N'$, respectively:
\begin{align*}
C_{N,r}:=\{v\in V_N: \phi_{N,v}=\max_{u\in B(v,r)} \phi_{N,u}\},\\
C'_{N,r}:=\{v\in V_N: \phi_{N,v}'=\max_{u\in B(v,r)} \phi_{N,u}' \}.
\end{align*}
Finally, define the level sets for $\phi_N$ and $\phi_N'$:
\begin{align*}
\Gamma_N(\lambda):=\{v\in V_N: \phi_{N,v}\geq m_N-\lambda\},\\
\Gamma_N'(\lambda):=\{v\in V_N: \phi'_{N,v}\geq m_N-\lambda\}.
\end{align*}

The following lemma, analogous to \cite[Lemma 4.6]{bl1}, is a perturbative result stating that the level sets of $\phi_N$ and $\phi'_N$ are close to each other.

\begin{lemma}  \label{lemma 7.4}
For any fixed $t>0,$ we have
\begin{align} \label{451}
\lim_{\lambda \ri} \liminf_{N\ri } \mathbb{P}(\Gamma'_N(\lambda)\subseteq \Gamma_N(2\lambda)) = 1
\end{align}
and
\begin{align} \label{452}
\lim_{\lambda \ri} \liminf_{N\ri } \mathbb{P}(\Gamma_N(\lambda)\subseteq \Gamma'_N(2\lambda)) = 1.
\end{align}
\end{lemma}

\begin{proof}
Since the pair $(\phi_N,\phi_N')$ is exchangeable, and the two statements are clearly symmetric, we will only prove the first one.
{Recalling $\pv =  \sqrt{1-\frac{t}{\log N}}\phi'_{N,v}  + \hat{\phi}''_{N,v}$ and noting that $ \sqrt{1-\frac{t}{\log N}} (m_N-\lambda) > m_N-\lambda-c$ for some constant $c>0$ (depending on $t$),
\begin{align*}
\Gamma_N'(\lambda) \setminus \Gamma_N(2\lambda) \subseteq \{v\in V_N: {\phi}_{N,v}' \geq m_N-\lambda,  \hat{\phi}_{N,v}''<c-\lambda\}.
\end{align*}}
Thus, using the independence of ${\phi}_{N}' $ and $ \hat{\phi}_{N}''$,  for large enough $\lambda$,
{\begin{align*}
\mathbb{P}(\Gamma_N'(\lambda) \setminus \Gamma_N(2\lambda)\neq \emptyset)& \leq   \mathbb{P}(|\Gamma_N'(\lambda)| > e^{2\sqrt{2d} \lambda} )+  \mathbb{E}  [ \mathbb{P} (\exists v\in \Gamma_N'(\lambda), \hat{\phi}_{N,v}''<c-\lambda | \phi_N') \mathbf{1}( |\Gamma_N'(\lambda)| \leq  e^{2\sqrt{2d} \lambda})] \\
& \leq e^{-\sqrt{2d}\lambda/8}  + e^{2\sqrt{2d} \lambda} e^{-(\lambda-c)^2/(4t)},
\end{align*}
where we used  Theorem \ref{theorem 3.0} and $\var( \hat{\phi}_{N,v}'') <2t$ (for large $N$) in the last inequality. Thus,  we obtain \eqref{451}.}
\end{proof}

Now, consider the a.s. well-defined mappings
\begin{align}
\Pi(v) :=\argmax_{u\in B(v,2r)} \phi_{N,u},\quad \Pi'(v) :=\argmax_{u\in B(v,2r)} \phi'_{N,u}.
\end{align}
The following proposition claims that the image $\Pi(C_{N,r})$ lands inside $C'_{N,r}$ and establishes the closeness of the cluster weights of $\phi_N$ around points in $C_{N,r}$ to that of $\phi'_N$ around their images in $\Pi(C_{N,r})$.

\begin{proposition} \label{lemma 7.6}
For any $\e , \lambda>0$,
the following implications hold with probability tending to one as $N\ri$ followed by $r\ri$:
\begin{align} \label{431}
\bs
v&\in C_{N,r}\cap \Gamma_N(\lambda) \cap \Gamma_N'(\lambda) \\
\Rightarrow
& \ \Pi'(v)\in C_{N,r}', \quad |\Pi'(v)-v|\leq \frac{r}{2},\quad      0\leq \pv - \phi_{N,\Pi'(v)}   \leq  \e, \\
 &\Big\vert \sum_{u\in B(v,2r)} e^{\beta (\pu-\pv)}  -   \sum_{u'\in B(\Pi'(v),2r)} e^{\beta (\phi'_{N,u'}-\phi'_{N,\Pi'(v)})}  \Big\vert < 10\e
\es.
\end{align}
{In addition,}
\begin{align} \label{432}
\bs
v&\in C_{N,r}'\cap \Gamma_N(\lambda) \cap \Gamma_N'(\lambda) \\
\Rightarrow
& \ \Pi(v)\in C_{N,r}, \quad |\Pi(v)-v|\leq \frac{r}{2},\quad      0\leq  \phi_{N,\Pi(v)}  -  \pv  \leq \e, \\
 &\Big\vert \sum_{u\in B(v,2r)} e^{\beta (\pu'-\pv')}  -   \sum_{u'\in B(\Pi(v),2r)} e^{\beta (\phi_{N,u'}-\phi_{N,\Pi(v)})}  \Big\vert < 10\e.
\es
\end{align}
\end{proposition}

\begin{remark}
A previous version of this result for the 2D GFF appearing as \cite[Lemma 4.8]{bl1} did not claim anything about the cluster weights. This is where we crucially use the control on the cardinality of level sets (Theorem \ref{theorem 3.0}).
\end{remark}

The proof of this involves a bit of preparation. The first result we need is the already alluded to negligibility of the Gibbs weight of $\Gamma_N(\lambda)^c$ for large $\lambda>0$.
\begin{lemma}\label{lemma 7.7}
 Suppose that $\beta>\sqrt{2d}$. Then, there exist constants $C_\beta,\tau_\beta>0$ such that the event
 \begin{align} \label{event e1}
E^\lambda_1 := \Big\{   \sum_{v\in V_N \cap {\Gamma_N(\lambda)^c} } e^{\beta(\pv - m_N)}   < C_\beta (e^{-\tau_\beta \lfloor \lambda \rfloor} + N^{-\tau_\beta }  ) \Big\}
\end{align} 
satisfies
\begin{align*}
\lim_{\lambda\ri} \liminf_{N\ri} \mathbb{P} (E^\lambda_1) = 1.
\end{align*}

\end{lemma}

Next, for  any random field $\psi_N = (\psi_{N,v})_{v\in V_N}$ and any subset $A\subseteq V_N$, define
the osscillation on $A$ as {\begin{align*}
\text{osc}_A (\psi_N) := \max_{v\in A} \psi_{N,v} - \min_{v\in A} \psi_{N,v}.
\end{align*}}
The following lemma bounds the oscillations of the fields $\phi_N'$ and  $ \hat{\phi}''_N$ near the  points whose $\phi_N'$-value is close to $m_N$.
\begin{lemma} \label{lemma 7.8}
Let $\lambda>0$ and $r\geq 1$. The event
\begin{align} \label{437}
E_2 := \Big\{\max_{v\in \Gamma_N'(\lambda)} \textup{osc}_{B(v,2r)} (\hat{\phi}''_N) <  \frac{1}{(\log N)^{1/4}} \Big\},
\end{align} 
satisfies
\begin{align} \label{461}
\lim_{N \ri } \mathbb{P}(E_2)=1.
\end{align}
In addition,  the event
\begin{align} \label{466}
E_2' := \{ \max_{v\in \Gamma_N'(\lambda)} \textup{osc}_{B(v,2r)}  ({\phi}_N') < 2\sqrt{\log N}  \}. 
\end{align}
satisfies
\begin{align}   \label{460}
\lim_{N \ri } \mathbb{P}(E_2' )=1.
\end{align}
{Furthermore,}
\begin{align} \label{464}
\lim_{N \ri } \mathbb{P}( \max_{v\in \Gamma_N'(\lambda)}  \max_{u\in B(v,2r)}   |\pu'  - \sqrt{2d}\log N |< 3 \sqrt{\log N}  )=1.
\end{align}
\end{lemma} 
Note that  this lemma implies that for any $\lambda>0$ and $r\geq 1$, the event
\begin{align} \label{4700}
E_0:= \Big\{\max_{ v\in \Gamma_N'(\lambda)} \max_{u\in B(v,2r)}   | \pu  -( \pu'  -   \sqrt{d/2} t  +\hat{\phi}''_{N,v} )  |   \leq  \frac{2}{(\log N)^{1/4}}\Big\}
\end{align}
satisfies  
\begin{align} \label{4711}
\lim_{N\ri}\mathbb{P}(E_0) = 1.
\end{align}
To see this,
under the events listed in Lemma \ref{lemma 7.8}, 
for any $v\in  \Gamma_N'(\lambda)$ and $u\in B(v,2r)$,
\begin{align*}
\pu= \sqrt{1-\frac{t}{\log N}}\pu' +  \hat{\phi}''_{N,u} & \overset{ \eqref{437}, \eqref{464}}{=} \Big(1-\frac{t}{2\log N} \Big)   \pu' + \hat{\phi}''_{N,v} + \tau \\
 &\overset{\eqref{464}}{=}  \pu'  -     \sqrt{d/2} t  + \hat{\phi}''_{N,v}+ \tau,
\end{align*}
{where in the above $\tau$ is a random {quantity} which can vary line to line such that $|\tau| \leq  \frac{10}{(\log N)^{1/4}}$.}

Postponing the proofs of Lemmas \ref{lemma 7.7} and \ref{lemma 7.8}, we first prove Proposition \ref{lemma 7.6} assuming them.

\begin{proof}[Proof of Proposition  \ref{lemma 7.6}]
{Let $\rho \in (\lambda,\infty)$ be a   constant. We consider the following series of events appearing in the preceding lemmas.
\begin{itemize}
\item  Event $E^{\rho}_1$  defined in Lemma \ref{lemma 7.7} (from now on, we suppress the parameter $\rho$ and simply write $E_1$):
\begin{align*}
E_1 := \Big\{   \sum_{v\in V_N\cap \Gamma_N(\rho)^c} e^{\beta(\pv - m_N)}   < C_\beta (e^{-\tau_\beta \lfloor \rho \rfloor} + N^{-\tau_\beta }  ) \Big\}.
\end{align*}  
\item  Event $E_1'$, the version of $E_1$  for the field $\phi_N'$:
 \begin{align*}
E_1' := \Big\{   \sum_{v\in V_N\cap \Gamma_N(\rho)^c } e^{\beta(\pv' - m_N)}   < C_\beta (e^{-\tau_\beta \lfloor \rho \rfloor} + N^{-\tau_\beta }  ) \Big\}.
\end{align*} 
\item   Event $E_2$ defined in  \eqref{437} in  Lemma \ref{lemma 7.8}:
\begin{align*}
E_2 := \Big\{\max_{v\in \Gamma_N'(\lambda)} \textup{osc}_{B(v,2r)} (\hat{\phi}''_N) <  \frac{1}{(\log N)^{1/4}} \Big\}.
\end{align*}
\item Event $E_2'$ defined in \eqref{466} in   Lemma \ref{lemma 7.8}:
\begin{align*}
E_2' := \{ \max_{v\in \Gamma_N'(\lambda)} \textup{osc}_{B(v,2r)} ({\phi}_N') < 2\sqrt{\log N}  \}. 
\end{align*}
\item  Event $E_0$ defined  in \eqref{4700}:
\begin{align*}
E_0:= \Big\{\max_{ v\in \Gamma_N'(\lambda)} \max_{u\in B(v,2r)}   | \pu  -( \pu'  -   \sqrt{d/2} t  +\hat{\phi}''_{N,v} )  |   \leq  \frac{2}{(\log N)^{1/4}}\Big\}.
\end{align*}
\end{itemize}
In addition, define the additional events
\begin{align}  
E_3:= \{  |u-v| \notin [r/2,N/r], \  \forall u,v\in \Gamma_N(\rho)\}, \label{435}\\
E_3':= \{  |u-v| \notin [r/2,N/r], \  \forall u,v\in \Gamma'_N(\rho)\} \label{4355}.
\end{align}
Now, let us consider the event}
\begin{align}\label{434}
E:=
E_1\cap E_1'\cap E_2\cap E_2'  \cap E_3\cap E_3' \cap E_0.
\end{align}
By  Lemmas \ref{lemma 7.7}, \ref{lemma 7.8}, \eqref{meso} and \eqref{4711}, 
\begin{align} \label{433}
\lim_{\rho \ri}  \liminf_{r\ri}    \liminf_{N\ri}  \mathbb{P}(E) = 1.
\end{align}  

We now arrive at the key claim.

\noindent
\textbf{Claim.}
There exists $\rho_0=\rho_0(\lambda)>\lambda$ such that for any $\rho > \rho_0$ and $r\geq 1$,  the implication \eqref{431} holds
under the  event $E$ for  sufficiently large $N$. 

~

Let  $v\in C_{N,r}\cap \Gamma_N(\lambda)\cap \Gamma'_N(\lambda)$. Then,   $ \phi_{N, \Pi'(v)}' \geq \pv' \geq  m_N-\lambda$. Since $v, \Pi'(v) \in  \Gamma'_N(\lambda)\subseteq \Gamma'_N(\rho)$ (recall $\rho>\lambda$) and $|v-\Pi'(v)| \leq 2r<N/r$, owing to the event  $E_3'$ we have  $ |\Pi'(v)-v|\leq r/2$. Combining this with  $\Pi'(v) =\argmax_{u\in B(v,2r)} \phi'_{N,u}$, we obtain $\Pi'(v)\in C_{N,r}'$. Also,
 since $v\in C_{N,r}\cap \Gamma_N'(\lambda)$ and $\Pi'(v)\in C_{N,r}'$, 
\begin{align*}
0\leq  \pv - \phi_{N,\Pi'(v)} \leq  \hat{\phi}_{N,v}''-\hat{\phi}_{N,\Pi'(v)}'' \leq  \textup{osc}_{B(v,2r)} ( \hat{\phi}''_N) < (\log N)^{-1/4}.
\end{align*}
The first inequality uses $v\in C_{N,r}$ and similarly the second inequality uses the expression \eqref{ou}, and that $\Pi'(v)\in C_{N,r}'$ and the last inequality uses the event $E_2.$
{By a similar reasoning, for sufficiently large  $N$,
\begin{align} \label{430}
0  \leq  \phi'_{N,\Pi'(v)} - \pv' \leq   2 ( \hat{\phi}_{N,v}'' - \hat{\phi}_{N,\Pi'(v)}''  )  \leq   2\textup{osc}_{B(v,2r)}  (\hat{\phi}''_N )<  2(\log N)^{-1/4} ,
\end{align}
where the second inequality follows from the expression \eqref{ou0} and \eqref{ou}:
$
\phi_{N} = \sqrt{1-\frac{t}{\log N}}\phi'_N + \hat{\phi}_N''
$
together with  the fact $\pv \geq   \phi_{N,\Pi'(v)}$ and  $   \phi'_{N,\Pi'(v)} \geq \pv'$.}

{Finally, we verify the last implication of \eqref{431}. Let $D:= B(v,2r)\cap B(\Pi'(v),2r)$. Since $|\Pi'(v)-v| \leq r/2$, for sufficiently large $N$, 
\begin{align*}
u\in  ( B(v,2r) \cup B(\Pi'(v),2r) )      \setminus D \Rightarrow  r<|u-v| <3r  \Rightarrow \pu' <  m_N-\rho ,
\end{align*}
where the last implication holds by the event $E_3'$. Hence, under the event $E'_1 \cap E'_3$,  for  any $v\in \Gamma'_N(\lambda)$ (recall $ \Pi'(v) \in   \Gamma'_N(\lambda)$),}
\begin{align} \label{439}
\sum_{u\in ( B(v,2r) \cup B(\Pi'(v),2r) )   \setminus D}  e^{\beta (\pu'- \phi_{N, \Pi'(v)}')} & \leq   e^{\beta \lambda}   \sum_{u\in ( B(v,2r) \cup B(\Pi'(v),2r) )   \setminus D}  e^{\beta (\pu'- m_N )}  \nonumber \\
&\leq e^{\beta \lambda} \sum_{u\in V_N, \pu' < m_N-\rho}  e^{\beta (\pu'-m_N)}   \nonumber \\
&\overset{\eqref{event e1}}{\leq} C_\beta   e^{\beta \lambda} (e^{-\tau_\beta \lfloor \rho \rfloor} + N^{-\tau_\beta }  )  < \e
\end{align}
for large enough $\rho$ and $N$.
In addition,  under the event $E$, for any $v\in C_{N,r}\cap \Gamma_N(\lambda)\cap \Gamma'_N(\lambda)$,
\begin{align}  \label{436}
\sum_{u\in B(v,2r)} e^{\beta (\pu-\pv)} &\overset{\eqref{4700}}{=}  \Big(1+ O\Big( \frac{1}{(\log N)^{1/4}}\Big)\Big)  \sum_{u\in B(v,2r)} e^{\beta (\pu' -  \pv')}  \nonumber \\
&\overset{\eqref{430}}{=}  \Big(1+ O\Big( \frac{1}{(\log N)^{1/4}}\Big)\Big)  \sum_{u\in B(v,2r)} e^{\beta (\pu' -   \phi_{N, \Pi'(v)}')} .
\end{align}
Therefore, setting $K: = \sum_{u\in D} e^{\beta (\pu'-  \phi_{N, \Pi'(v)}')}  $,
\begin{align} \label{438}
\Big\vert\sum_{u\in B(v,2r)} & e^{\beta (\pu-\pv)} -  \sum_{u'\in B(\Pi'(v),2r)} e^{\beta (\phi_{N,u'}' -   \phi_{N, \Pi'(v)}')}  \Big\vert  \nonumber \\
&\overset{\eqref{436}}{=}\Big\vert\Big(1+ O\Big( \frac{1}{(\log N)^{1/4}}\Big)\Big)  \sum_{u\in B(v,2r)} e^{\beta (\pu' -   \phi_{N, \Pi'(v)}')}  -   \sum_{u\in B(\Pi'(v),2r)} e^{\beta (\phi_{N,u}' -   \phi_{N, \Pi'(v)}')}  \Big\vert  \nonumber  \\
&\overset{\eqref{439}}{=}\Big\vert  \Big(1+ O\Big( \frac{1}{(\log N)^{1/4}}\Big)\Big)   (K+\zeta_1) - (K+\zeta_2)\Big\vert
\end{align}
for some $\zeta_1,\zeta_2 \in (0,\e)$.
{In order to show the smallness of this quantity, it suffices to control the quantity $K$ uniformly in $N$. Recalling   $\Pi'(v) =\argmax_{u\in B(v,2r)} \phi'_{N,u}$,  we have $ \phi_{N, \Pi'(v)}' \geq \pu' $ for $u\in  B(v,2r)$ and thus for $u\in D$. Hence, using the fact that $|D|\leq Cr^d$,
\begin{align*} 
K = \sum_{u\in D} e^{\beta (\pu'-  \phi_{N, \Pi'(v)}')}  \leq Cr^d.
\end{align*}}
Thus, the quantity  \eqref{438} is bounded by $10\e$ for large enough $N$.  This verifies the claim.

Therefore, by \eqref{433}, the implication \eqref{431}  holds with probability tending to one.
{The implication \eqref{432} follows similarly, noting that $(\phi_N,\phi_N')$ is a  exchangeable pair  and the statements are symmetric, except the third statement $0\leq \phi_{N,\Pi(v)} - \pv \leq \e$. This is obtained from the fact that for $v\in C'_{N,r}$,
\begin{align*}
0\leq \phi_{N,\Pi(v)} - \pv   \leq  \hat{\phi}_{N,\Pi(v)}''-\hat{\phi}_{N,v}'' \leq  \textup{osc}_{B(v,2r)} ( \hat{\phi}''_N) < (\log N)^{-1/4}.
\end{align*}
}
\end{proof}
We now prove Lemmas \ref{lemma 7.7} and \ref{lemma 7.8}.
\begin{proof}[Proof of Lemma \ref{lemma 7.7}]
{Let $\kappa>0$ be a small constant such that $\sqrt{2d}+2\kappa<\beta$. Define
\begin{align} \label{4300}
\ell_0:= \Big\lfloor  \frac{d}{\sqrt{2d}+\kappa}\log \Big(\frac{N}{c_0}\Big)\Big \rfloor,
\end{align}
 where $c_0>0$ is a constant from Theorem \ref{theorem 3.0}. The definition of $\ell_0$ is dictated by the fact that the contribution of Gibbs weights from  $\pv \leq  m_N - \ell_0$ can be controlled via a trivial bound. To be precise, under    the  event  
\begin{align} \label{111}
 \cap_{\ell =[\lambda] }^{\ell_0}   \{|\Gamma_N(\ell)| \leq  e^{(\sqrt{2d}+\kappa)\ell}\},
\end{align}
we have
\begin{align} \label{110}
\sum_{v\in V_N, \pv \leq m_N-\lambda} e^{\beta(\pv - m_N)} &\leq \sum_{\ell=[\lambda]}^{  \ell_0  } e^{- \beta  (\ell-1)} |\Gamma_N(\ell)|   +N^d e^{-\beta  \ell_0 }  \nonumber \\
&\leq Ce^{-(\beta - (  \sqrt{2d} + \kappa))[\lambda]} + CN^{-d(\frac{\beta}{\sqrt{2d}+\kappa}-1)}  .
\end{align}
On the other hand, by Theorem \ref{theorem 3.0}, for large enough  $\ell$ and  $N \geq c_0 e^{(\sqrt{2d}+\kappa) \ell/d}$,
\begin{align*}
 \mathbb{P} (  |\Gamma_N(\ell)|  \geq  e^{(\sqrt{2d}+\kappa)\ell} ) \leq e^{-\kappa \ell/8}.
\end{align*} 
By a union bound (recall  \eqref{4300} which implies  $N \geq c_0 e^{(\sqrt{2d}+\kappa) \ell_0/d}$),
\begin{align*}
\lim_{\lambda\ri} \liminf_{N\ri} \mathbb{P}\big( \cap_{\ell =[\lambda] }^{\ell_0}   \{|\Gamma_N(\ell)| \leq  e^{(\sqrt{2d}+\kappa)\ell}\}\big) =1.
\end{align*}}
This finishes the proof.
\end{proof}

\begin{proof}[Proof of Lemma \ref{lemma 7.8}]
First,  we prove  \eqref{461}.
Note that the complement of the  event in  \eqref{461} is contained in
\begin{align*}
\cup_{v\in \Gamma_N'(\lambda)} \Big\{\max_{u\in B(v,2r)} |\hat{\phi}''_{N,u} -\hat{\phi}''_{N,v}| \geq  \frac{1}{2(\log N)^{1/4}} \Big\}.
\end{align*}
Thus, by Chebyshev's inequality and a union bound, noting that  ${\phi}_N'$ and $\hat{\phi}_N''$ are independent,  by conditioning on   ${\phi}_N'$, for  any $M,\delta>0$,
\begin{align} \label{457}
\mathbb{P}&\Big( \Big\{  \max_{v\in \Gamma'(\lambda)} \textup{osc}_{B(v,2r)} ( \hat{\phi}_N'' ) \geq  \frac{1}{(\log N)^{1/4}} \Big\} \cap \{|\Gamma_N'(\lambda) |<M\}  \cap  \{\Gamma_N'(\lambda)\subseteq V_N^\delta\} \Big) \nonumber \\
& \leq CMr^d (\log N)^{1/2} \sup_{u,v\in V_N^{\delta/2}, |u-v|\leq 2r} \mathbb{E}(\hat{\phi}''_{N,u} -\hat{\phi}''_{N,v})^2  .
\end{align}
By Assumption \ref{a2}, there exists $C=C(\delta,r)>0$ such that
\begin{align} \label{458}
\sup_{u,v\in V_N^{\delta/2},|u-v|\leq 2r} \mathbb{E}(\hat{\phi}''_{N,u} -\hat{\phi}''_{N,v})^2   = \frac{t}{\log N} \sup_{u,v\in V_N^{\delta/2},|u-v|\leq 2r} \mathbb{E}(\phi''_{N,u} -\phi''_{N,v})^2\leq  C\frac{t}{\log N}.
\end{align}
On the other hand, by Theorem \ref{theorem 3.0} and \eqref{410},
\begin{align} \label{459}
\lim_{M\ri} \liminf_{N\ri} \mathbb{P}( \{|\Gamma_N'(\lambda) |<M\}) =1,\quad \lim_{\delta \downarrow 0} \liminf_{N\ri} \mathbb{P}( \{\Gamma_N'(\lambda)\subseteq V_N^\delta\}) =1.
\end{align}
Thus, by \eqref{457}-\eqref{459}, we conclude the proof of \eqref{461}.

~

{Next, we establish  \eqref{460}  via a union bound together with the upper bound on  the two-point estimate
\begin{align*}
\mathbb{P}(  \pv' \geq m_N-\lambda,  | {\phi}_{N,u}' - {\phi}_{N,v}'| \geq \sqrt{\log N})
\end{align*}  
for $v\in V_N$ and $u\in B(v,2r)$}.
{We bound this probability by conditioning on the value of  $\pv'$. To do this, we  analyze the conditional mean and variance of $\pu'$ given $\pv'$, by computing the correlation between $\pu'$ and $\pv'$. 
First, we observe that for any $\delta,r>0$, there exists a constant $C_1 = C_1 (\delta,r)>0$ such that the following holds:}
For any $u,v \in V_N^{\delta/2}$ with  $|u-v| \leq 2r$,  the correlation coefficient between $\pu'$ and $\pv'$ is  $$\rho:= \frac{\cov (\pu',\pv')}{\sqrt{\text{Var} \pu'}\sqrt{\text{Var} \pv'}} = 1-\frac{c_{u,v}}{\log N}$$
for $|c_{u,v}| \leq C_1(\delta,r)$. {This follows from the fact $$| \cov (\pu',\pv') - \log N|  \leq \log_+ |u-v| + \alpha(\delta/2)  \leq \log (2r) + \alpha(\delta/2)$$ and $$|\var \pu'- \log N|< \alpha(\delta/2),\quad   |\var \pv'- \log N| < \alpha(\delta/2),$$    obtained by  Assumption \ref{a2}. Thus,  the conditional mean and variance of $\pu'$, given the value of $\pv'$ such that} $|\pv'| \leq 10\sqrt{2d}\log N$, are
\begin{align*}
\text{Conditional mean} &= \rho \pv' = \pv'+ \e_1 , \\
\text{Conditional variance} &= (1- \rho^2)\text{Var} \pu' = \e_2
\end{align*}
for $|\e_1|,|\e_2 |  \leq C_2 = C_2(\delta,r)$.
This implies that for some constant $\gamma=\gamma(\delta,r)>0$,
\begin{align*}
\mathbb{P}( |\pu'-\pv'| \geq \sqrt{\log N} | \pv')\1_{|\pv'| \leq 10\sqrt{2d}\log N } \leq \frac{1}{N^\gamma}.
\end{align*}
{Since $\var \pv'  \leq  \log N + \alpha_0$ by Assumption \ref{a1}, using the standard Gaussian tail estimate,
\begin{align*}
\mathbb{P}(\pv' \geq m_N-\lambda) \leq C  \frac{1}{\sqrt{\log N}}e^{-\frac{(m_N - \lambda)^2}{2 (\log N +\alpha_0 )} } \leq C \frac{\log N}{N^d} .
\end{align*}}
Hence, for any $v\in V_N^\delta$ and $u\in B(v,2r)$, 
\begin{align} \label{450}
\mathbb{P}(\pv' \geq m_N-\lambda, |\pu'-\pv'| \geq \sqrt{\log N})& \leq    C \frac{\log N}{N^d} \frac{1}{N^\gamma} + \mathbb{P}(|\pv' | > 10\sqrt{2d}\log N) \nonumber\\
&
 \leq    C \frac{\log N}{N^d} \frac{1}{N^\gamma} + \frac{1}{N^{10d}}.
\end{align}
Note that we have the implication
\begin{align*}
&\max_{v\in \Gamma_N'(\lambda) \cap V_N^\delta}  \textup{osc}_{B(v,2r)}  ({\phi}_N' ) \geq 2\sqrt{\log N}  \\
& \Rightarrow \exists v\in V_N^\delta, \exists u\in B(v,2r)  \ \text{such that} \ \pv' \geq m_N-\lambda,  | {\phi}_{N,u}' - {\phi}_{N,v}'| \geq \sqrt{\log N}.
\end{align*} 
  Hence, by \eqref{450} and the union bound over all possible $v\in V_N^\delta$ and $|u-v| \leq 2r$, 
\begin{align*}
\lim_{N \ri } \mathbb{P}(\max_{v\in \Gamma_N'(\lambda) \cap V_N^\delta} \textup{osc}_{B(v,2r)} ({\phi}_N')  \geq 2\sqrt{\log N} )=0.
\end{align*} 
Combining this with \eqref{410}, we obtain
\eqref{460}.

{Finally, this implies  \eqref{464}  in a  straightforward way. Indeed, since $v\in \Gamma_N'(\lambda) \Rightarrow \pv' \geq  \sqrt{2d}  \log N- 0.1 \sqrt{\log N}$ for large  enough $N$,  by \eqref{460} and Corollary \ref{tight},  we establish \eqref{464}.}
\end{proof}

We are now ready to prove Theorem \ref{theorem 7.1}. 
Throughout the proof, we will use the following fact: Let $X$ and $Y$ be any non-negative random variables. If $|X-Y| \leq a$ on the event $A$, then
\begin{align} \label{key}
|\mathbb{E}e^{-X} - \mathbb{E}e^{-Y} | \leq a + 2\mathbb{P}(A^c).
\end{align}
This follows by noting
{\begin{align*}
|\mathbb{E}e^{-X} - \mathbb{E}e^{-Y} | \leq \mathbb{E}  ( |e^{-X} - e^{-Y} |  \1_A + |e^{-X} - e^{-Y} |  \1_{A^c} )  \leq   \mathbb{E}  ( |X-Y|  \1_A ) + \mathbb{E} (2 \1_{A^c}) \leq  a + 2\mathbb{P}(A^c),
\end{align*}
where we used $|e^{-x} - e^{-y}| \leq |x-y|$ for $x,y \geq 0$ in the second inequality, a consequence of the mean value theorem.
}
\begin{proof}[Proof of Theorem \ref{theorem 7.1}]
For $v\in V_N$ and $r  \geq 1$,  {define the version of \eqref{sbar} for the field $\phi_N'$ as well:}
\begin{align*}
  \bar{S}_{v,r}' := \sum_{u\in B(v,r)} e^{\beta(\pu' - \pv')} . 
\end{align*}
We first  prove that 
for any  constant $\iota>0$ and a continuous function  $f: \R \times [0,\infty) \rightarrow [0,\infty)$  of compact support,
\begin{align}  \label{711}
\lim_{r \ri } \limsup_{N \ri } \mathbb{P}\Big( \Big\vert\sum_{v\in C_{N,r}} f(\pv-m_N ,  \bar{S}_{v,2r}  )-  \sum_{v\in C'_{N,r}} f(\pv-m_N,  \bar{S}'_{v,2r}  ) \Big\vert  >\iota\Big)=0.
\end{align} 
 Since $f$ is continuous and compactly supported, for any $\lambda>0$,
{there exists $\e >0$} such that 
\begin{align} \label{731}
|y-y'|,|z-z'|\leq 10\e \Rightarrow |  f(y, z) - f(y',z') | < \iota e^{-2\sqrt{2d}\lambda}.
\end{align}
{By Proposition \ref{lemma 7.6}, there exists an event  $E$, with probability tending to one as $N\ri$ followed by $r\ri$, under which   the following implications  hold:}
\begin{align*} 
v&\in C_{N,r}\cap \Gamma_N(\lambda) \cap \Gamma_N'(\lambda) \\
\Rightarrow
& \ \Pi'(v)\in C_{N,r}', \quad |\Pi'(v)-v|\leq \frac{r}{2},\quad      0\leq \pv - \phi_{N,\Pi'(v)}   \leq  \e,  \quad | \bar{S}_{v,2r} - \bar{S}'_{\Pi'(v),2r}  | < 10\e , \\
v&\in C_{N,r}'\cap \Gamma_N(\lambda) \cap \Gamma_N'(\lambda) \\
\Rightarrow
& \ \Pi(v)\in C_{N,r}, \quad |\Pi(v)-v|\leq \frac{r}{2},\quad      0\leq  \phi_{N,\Pi(v)}  -  \pv  \leq \e,  \quad  | \bar{S}'_{v,2r} - \bar{S}_{\Pi(v),2r}  | < 10\e .
\end{align*}
Then, by \eqref{731}, under the event $E$, for any  $v\in C_{N,r} \cap \Gamma_N(\lambda) \cap \Gamma'_N(\lambda) $,
\begin{align} \label{732}
|  f(\pv- m_N, \bar{S}_{v,2r})   - f(\phi_{N,\Pi'(v)}-m_N,   \bar{S}'_{\Pi'(v),2r} ) | &< \iota e^{-2\sqrt{2d}\lambda}.
\end{align}
Define the event
\begin{align*}
A:=E \cap \{\Gamma_N(\lambda/2)\subseteq \Gamma'_N(\lambda)\}\cap \{|\Gamma'_N(\lambda)|\leq  e^{2\sqrt{2d}\lambda} \}.
\end{align*}
Then,  by  Theorem \ref{theorem 3.0},  Lemma \ref{lemma 7.4} and  Proposition \ref{lemma 7.6},
 \begin{align} \label{712}
 \lim_{\lambda \ri} \limsup_{r\ri}  \limsup_{N\ri} \mathbb{P}(A^c) = 0.
 \end{align}
{Since $f$ is compactly supported,   for sufficiently large $\lambda>0$,  $ f(\pv-m_N  , \bar{S}_{v,2r}) = 0$ for any  $v\notin \Gamma_N(\lambda/2)$. Hence, for large enough $\lambda$, under the event $A$,}
\begin{align} \label{714}
\sum_{v\in C_{N,r}} f(\pv-m_N  , \bar{S}_{v,2r}) 
&= \sum_{v\in C_{N,r}\cap  \Gamma_N(\lambda) \cap    \Gamma_N'(\lambda) }  f(\pv-m_N ,  \bar{S}_{v,2r})  \nonumber \\
&\overset{\eqref{732}}{ \leq }  \iota e^{-2\sqrt{2d}\lambda} |\Gamma_N'(\lambda)| + \sum_{v\in C_{N,r}\cap  \Gamma_N(\lambda) \cap    \Gamma_N'(\lambda) } f (\phi_{N,\Pi'(v)}-m_N,   \bar{S}'_{\Pi'(v),2r} ) \nonumber \\
&\leq \iota +   \sum_{v\in C'_{N,r}} f(\pv-m_N , \bar{S}'_{v,2r}  ) .
\end{align}
Thus, by \eqref{712} and \eqref{714},
\begin{align*}
\lim_{r \ri } \limsup_{N \ri } \mathbb{P}&\Big(  \sum_{v\in C_{N,r}}  f(\pv-m_N ,\bar{S}_{v,2r}  ) >\iota+ \sum_{v\in C'_{N,r}} f(\pv-m_N ,\bar{S}'_{v,2r}  ) \Big)=0.
\end{align*}
{Now, we deduce the other side of the above inequality using the similar argument as above.  Under the event $A$,
for any  $v\in C'_{N,r} \cap \Gamma_N(\lambda) \cap \Gamma'_N(\lambda) $,
\begin{align*} 
|   f(\phi_{N,\Pi(v)}-m_N,   \bar{S}_{\Pi(v),2r} ) - f(\pv- m_N, \bar{S}'_{v,2r})  | &< \iota e^{-2\sqrt{2d}\lambda},
\end{align*}
which implies that for large enough $\lambda$,
\begin{align*}
 \sum_{v\in C'_{N,r}} f(\pv-m_N , \bar{S}'_{v,2r}  ) & = \sum_{v\in C'_{N,r}\cap  \Gamma_N(\lambda) \cap    \Gamma_N'(\lambda) }  f(\pv-m_N ,  \bar{S}'_{v,2r})  \\
& \leq   \iota e^{-2\sqrt{2d}\lambda} |\Gamma_N'(\lambda)| + \sum_{v\in C'_{N,r}\cap  \Gamma_N(\lambda) \cap    \Gamma_N'(\lambda) } f (\phi_{N,\Pi(v)}-m_N,   \bar{S}_{\Pi(v),2r} )  \\
&\leq \iota +   \sum_{v\in C_{N,r}} f(\pv-m_N , \bar{S}_{v,2r}  ) .
\end{align*} Thus, we obtain \eqref{711}.}

Next, we  show that
\begin{align} \label{722}
\lim_{\lambda \ri}   \limsup_{\delta \downarrow 0} \limsup_{N\ri} &\mathbb{P}\Big( \Big\vert \sum_{v\in C'_{N,r}}  f(\pv-m_N ,\bar{S}'_{v,2r}  )  \nonumber \\
&- \sum_{v\in C'_{N,r}\cap \Gamma'_N(\lambda)\cap V_N^\delta}   f (\pv' - m_N -  \sqrt{d/2} t + \hat{\phi}''_{N,v} , \bar{S}'_{v,2r} ) \Big\vert  >2  \iota\Big) = 0 .
\end{align}
Since $f$ is compactly supported, by   \eqref{410},   {Corollary \ref{tight}} and Lemma \ref{lemma 7.4} 
\begin{align} \label{713}
\lim_{\lambda \ri}   \limsup_{\delta \downarrow 0} \limsup_{N\ri} \mathbb{P}\Big(\sum_{v\in  (\Gamma'_N(\lambda)\cap V_N^\delta)^c}  f(\pv-m_N ,\bar{S}'_{v,2r}  )  > \iota\Big) = 0 .
\end{align}
{In addition,  by \eqref{4700} and  \eqref{4711}, we obtain
\begin{align*} 
\lim_{\lambda\ri}\limsup_{N\ri} \mathbb{P}&\Big( \sum_{ v\in C'_{N,r}\cap \Gamma'_N(\lambda)\cap V_N^\delta} \Big\vert  f(\pv-m_N , \bar{S}_{v,2r}')  -   f (\pv' - m_N -  \sqrt{d/2} t + \hat{\phi}''_{N,v} , \bar{S}'_{v,2r})\Big\vert  > \iota\Big)  = 0,
\end{align*}
since the number of terms in the summation is bounded by $ |\Gamma_N'(\lambda)| \leq   e^{2\sqrt{2d}\lambda}$  with probability tending to 1 as $\lambda\ri$ (see Theorem \ref{theorem 3.0}).}
This together with  \eqref{713} implies \eqref{722}.
 Therefore, by
 \eqref{711} and  \eqref{722},
\begin{align} 
\lim_{\lambda \ri}& \limsup_{r\ri} \limsup_{\delta \downarrow 0} \limsup_{N\ri} \mathbb{P}\Big(\Big\vert  \sum_{v\in C_{N,r}}  f(\pv-m_N ,\bar{S}_{v,2r}  ) \nonumber \\
&- \sum_{v\in C'_{N,r}\cap \Gamma'_N(\lambda)\cap V_N^\delta}  f (\pv' - m_N -  \sqrt{d/2} t + \hat{\phi}''_{N,v} , \bar{S}'_{v,2r} )\ \Big\vert > 3 \iota\Big) = 0.
\end{align}
By the arbitrariness of $\iota>0$, this together with \eqref{key} imply that
\begin{align} \label{717}
\lim_{\lambda \ri}& \limsup_{r\ri} \limsup_{\delta \downarrow 0} \limsup_{N\ri}   \Big\vert  \mathbb{E}\exp \Big(  -\sum_{v\in C_{N,r}}  f(\pv-m_N ,\bar{S}_{v,2r}  )\Big) \nonumber \\
&-\mathbb{E}\exp \Big( - \sum_{v\in C'_{N,r}\cap \Gamma'_N(\lambda)\cap V_N^\delta}  f (\pv' - m_N -  \sqrt{d/2} t + \hat{\phi}''_{N,v} , \bar{S}'_{v,2r} ) \Big )  \Big\vert  = 0.
\end{align}

{Define the event $E_3'= \{  |u-v| \notin [r/2,N/r], \  \forall u,v\in \Gamma'_N(\lambda)\}$ as in \eqref{4355} (although the parameter in \eqref{4355} is $\rho$ instead of $\lambda$, we continue using the same notation to avoid introducing new notation, since the parameter $\rho$ will no longer be used).} Under the event $E_3'$, 
 any $u,v\in C_{N,r}' \cap \Gamma'_N(\lambda)\cap V_N^\delta$ with $u\neq v$ satisfies $|u-v| > N/r$. Thus, by Assumption \ref{a2}, for any such $u,v$,
\begin{align*}
\var (\hat{\phi}''_{N,u} ) = t+o(1),\quad   \cov ( \hat{\phi}''_{N,u},\hat{\phi}''_{N,v} ) = o(1).
\end{align*}
The above shows that the variables $\{\hat{\phi}''_{N,v}\}_{v\in C_{N,r}' \cap \Gamma'_N(\lambda)\cap V_N^\delta}$  form a multivariate Gaussian whose coordinates are ``almost'' independent. We introduce a useful fact about such vectors next. This is a slightly generalized version of \cite[Lemma 4.9]{bl1} {and the detailed proof will be presented in Appendix (see Section \ref{appendix b}).} 
\begin{lemma} \label{tensor}
Let   $m\in \mathbb{N}$, $\sigma>0$ be  constants and $g : \R^2 \rightarrow \R$ be a  continuous and compactly supported  function. Then, for any $\e > 0$, there is a constant $\delta > 0$ such that the following holds: For any $n\in \mathbb{N}$ with $n \leq m$, $z_1,\cdots,z_n,w_1,\cdots,w_n\in  \R$ and a  centered multivariate Gaussian vector $X = (X_1,..., X_n)$ satisfying
\begin{align} \label{gaussian condition}
\max_{1\leq i,j\leq n} |\textup{Cov}(X_i,X_j) - \sigma^2 \delta_{ij} | <\delta,
\end{align}
it holds that 
\begin{align}
\Big\vert \log \frac{\mathbb{E} \exp(\sum_{i=1}^n g(z_i+X_i,w_i))}{ \mathbb{E}\exp(  \sum_{i=1}^n g(z_i+Y_i,w_i))} \Big\vert  < \e,
\end{align}
where  $Y = (Y_1,\cdots,Y_n)$  denotes  the multivariate   Gaussian vector consisting of  i.i.d.  $\mathcal{N}(0,\sigma^2)$.
\end{lemma}

Using this fact and by the independence of $\phi'_N$ and $\hat{\phi}''_N$, under the event $$F:=E_3' \cap \{|\Gamma_N'(\lambda) | \leq e^{2\sqrt{2d}\lambda} \} ,$$
we have
\begin{align}   \label{716}
\mathbb{E}&\Big[\exp\Big(-  \sum_{v\in C'_{N,r}\cap \Gamma'_N(\lambda)\cap V_N^\delta}   f (\pv' - m_N -  \sqrt{d/2} t + \hat{\phi}''_{N,v} , \bar{S}'_{v,2r}  )\ \Big) \Big\vert \phi_N'\Big] \nonumber \\
&= e^{o(1)} \mathbb{E}\exp\Big(-\sum_{v\in C'_{N,r}\cap \Gamma'_N(\lambda)\cap V_N^\delta} f(\pv' - m_N -   \sqrt{d/2} t + B_t^v , \bar{S}'_{v,2r}) \Big) \nonumber \\
&=  e^{o(1)} \exp\Big( - \sum_{v\in C'_{N,r}\cap \Gamma'_N(\lambda)\cap V_N^\delta} f_t( \pv'- m_N, \bar{S}'_{v,2r}) \Big),
\end{align}
{where   $\{B_t^v: v\in C'_{N,r}\cap \Gamma'_N(\lambda)\cap V_N^\delta\}$ denote i.i.d. centered Gaussians with variance $t$ independent of everything else (in the above expressions, $\E$ is used to denote the conditional expectation given $\phi_N'$), and $o(1)$ is a random quantity which tends to 0 uniformly in $\phi_N'$. Noting that  $$ \lim_{\lambda\ri} \liminf_{r\ri} \liminf_{N\ri} \mathbb{P}(F )= 1,$$ using the fact $f,f_t\geq 0$,
\begin{align*}
\mathbb{E}&\exp\Big(-  \sum_{v\in C'_{N,r}\cap \Gamma'_N(\lambda)\cap V_N^\delta}   f (\pv' - m_N -  \sqrt{d/2} t + \hat{\phi}''_{N,v} , \bar{S}'_{v,2r}  )\ \Big) \\
&=  \mathbb{E} \Big[ \mathbb{E}\Big[\exp\Big(-  \sum_{v\in C'_{N,r}\cap \Gamma'_N(\lambda)\cap V_N^\delta}   f (\pv' - m_N -  \sqrt{d/2} t + \hat{\phi}''_{N,v} , \bar{S}'_{v,2r}  )\ \Big) \Big\vert \phi_N'\Big] \1_{F}  \Big]  + \e_1 \\
&\overset {\eqref{716}}{=} \mathbb{E} \Big[  e^{o(1)} \exp\Big( - \sum_{v\in C'_{N,r}\cap \Gamma'_N(\lambda)\cap V_N^\delta} f_t( \pv'- m_N, \bar{S}'_{v,2r}) \Big) \1_{F}  \Big] +  \e_1\\
&=\mathbb{E} \exp\Big( - \sum_{v\in C'_{N,r}\cap \Gamma'_N(\lambda)\cap V_N^\delta} f_t( \pv'- m_N, \bar{S}'_{v,2r}) \Big) +  \e_2
\end{align*}
where  $\e_i$s denote the error terms such that $\lim_{\lambda\ri} \liminf_{r\ri} \liminf_{N\ri}  |\e_i| =0$.
}

Combining this  with  \eqref{717},
\begin{align} \label{718}
\lim_{\lambda \ri} \limsup_{r\ri} \limsup_{\delta \downarrow 0} \limsup_{N\ri} \Big\vert & \mathbb{E} \exp\Big(-\sum_{v\in C_{N,r}}  f(\pv-m_N ,\bar{S}_{v,2r}  )\Big) \nonumber \\
&- \mathbb{E}  \exp\Big( - \sum_{v\in C'_{N,r}\cap \Gamma'_N(\lambda)\cap V_N^\delta} f_t (\pv'- m_N, \bar{S}'_{v,2r}) \Big) \Big\vert= 0.
\end{align}
{In order to conclude the proof, we verify that the second summation above is close to the summation over all $v\in C'_{N,r}$, by showing that for any $\iota>0$,}
\begin{align}   \label{719}
\lim_{\lambda \ri}   \limsup_{\delta \downarrow 0} \limsup_{N\ri} \mathbb{P}\Big(\sum_{v\in  (\Gamma'_N(\lambda)\cap V_N^\delta)^c}  f_t(\pv'-m_N ,\bar{S}'_{v,2r}  )  > \iota\Big) = 0 .
\end{align}
{Although $f_t$ is no longer compactly supported in the first coordinate (but still compact in the second coordinate), by the expression \eqref{700}, for each $t$,  one can deduce that there exist $C,c>0$ (depending on $t$) such that $f_t(y,z)\leq Ce^{-cy^2}$.}  Thus, under the event 
\begin{align} \label{744}
 \cap_{\ell =[\lambda] }^{\ell_0}   \{|\Gamma'_N(\ell)| \leq  e^{2\sqrt{2d} \ell}\}
\end{align} 
($\ell_0:=  [ \frac{d}{\sqrt{2d}+\kappa}\log (\frac{N}{c_0}) ]$, where $c_0>0$ is a constant  from Theorem \ref{theorem 3.0})
whose probability tends to 1 as $N\ri$ followed by $\lambda\ri$,
\begin{align} \label{745}
\sum_{v\in V_N: \pv' \leq m_N - \lambda}  f_t(\pv'-m_N ,\bar{S}'_{v,2r} )  \leq C \Big (  \sum_{\ell=[\lambda]}^{  \ell_0  } e^{-c ( \ell-1)^2} |\Gamma_N(\ell)|   +N^d e^{-c  \ell_0^2 } \Big ) \overset{\lambda,N \ri}{\rightarrow} 0.
\end{align}
 Therefore,    by   \eqref{410} and \eqref{745},  we obtain \eqref{719}. Since $\phi_N$ and $\phi'_N$ have the same law,
\begin{align} \label{715}
\mathbb{E} \exp\Big( - \sum_{v\in C'_{N,r} } f_t (\pv'- m_N, \bar{S}'_{v,2r}) \Big) =  \mathbb{E} \exp\Big( - \sum_{v\in C_{N,r} } f_t (\pv- m_N, \bar{S}_{v,2r}) \Big) .
\end{align}
{Hence, this together with \eqref{718} and \eqref{719}  imply}
\begin{align*}
\lim_{r\ri} \limsup_{N\ri}&  \Big\vert \mathbb{E}\exp\Big(-\sum_{v\in C_{N,r}}  f(\pv-m_N ,\bar{S}_{v,2r}  )\Big) - \mathbb{E} \exp\Big(-\sum_{v\in C_{N,r}}  f_t(\pv-m_N ,\bar{S}_{v,2r}  )\Big)  \Big\vert =0. 
\end{align*}
The last remaining piece is the following lemma which allows us to go from $r$ balls to $r_N$ balls.

\end{proof}

\begin{lemma}\label{lemma 7.2}
{Let $f: \R\times [0,\infty) \rightarrow [0,\infty)$ be a continuous function of compact support, or generally $f(y,z) \leq Ce^{-cy^2}$ for some $C,c>0$ and $f$ is compactly support in the second coordinate.} Then, 
for any $\e>0$ and $\{r_N\}_{N\geq 1}$ with $r_N\ri$ and $r_N/N\rightarrow 0$ as $N\ri$,  
\begin{align} \label{720}
\lim_{r\ri} \limsup_{N\ri} \mathbb{P}\Big( \Big\vert \sum_{v\in C_{N,r}}  f(\pv-m_N ,\bar{S}_{v,2r}  ) - \sum_{v\in C_{N,r_N}}  f(\pv-m_N ,\bar{S}_{v,r_N/2}   ) \Big\vert >\e \Big) = 0.
\end{align}
\end{lemma}

\begin{proof}
Let us first assume that $f(y,z)$ has a compact support.
{Let $\lambda>0$ be  a  large  number such that
 $f(y,z) = 0$ for $y\notin [-\lambda,\lambda]$.
Since $f$ is uniformly continuous, there exists $\kappa>0$ such that 
 \begin{align}\label{746}
y\in \R, |z-z'| \leq \kappa \Rightarrow |f(y,z) - f(y,z')| \leq e^{-2\sqrt{2d}\lambda} \e.
 \end{align}
Let $\delta>0$ and  $g(r):= \log \log \log r$, intended to be a slowly growing function. Similarly as in \eqref{435}, define the event
$$\hat{E}_3:= \{  |u-v| \notin [r,r_N], \  \forall u,v\in  V_N^\delta\cap \Gamma_N(g(r))  \}.$$
Then, for large enough $r$ and $N$, under the event $\{\Gamma_N(\lambda) \subseteq V_N^{2\delta}\}\cap \hat{E}_3$,   
\begin{align} \label{743}
 \forall v\in \Gamma_N(\lambda), \  \forall u\in B(v,r_N/2) \setminus B(v,2r) \Rightarrow \pu < m_N - g(r).
\end{align} 
To see this, first note that $v\in \Gamma_N(\lambda) \subseteq V_N^{2\delta} $ and  $|u-v| \leq r_N/2 < \delta N$ (for large $N$)  imply   $u\in V_N^\delta$. Since $g(r) > \lambda$ for large $r$, $v\in V_N^\delta \cap  \Gamma_N(g(r))$. By the definition of the event $\hat{E}_3$, we obtain \eqref{743}.}
{In addition, since  $r_N/N\rightarrow 0$ as $N\ri$, by Lemma  \ref{lemma 4.4} (note that  $g(r) <  c(\delta) \log \log r$ for large $r$), 
\begin{align} \label{740}
  \lim_{r\ri} \liminf_{N\ri} \mathbb{P}(\hat{E}_3) = 1.
\end{align}

Now, we verify the following implication which says that under $\{\Gamma_N(\lambda) \subseteq V_N^{2\delta}\}\cap \hat{E}_3$,  $r$-local maximizers are also $r_n$-local maximizers, i.e.,  for large enough $r$ and $N$,
\begin{align} \label{741}
\{\Gamma_N(\lambda) \subseteq V_N^{2\delta}\}\cap  \hat{E}_3 \Rightarrow  \{C_{N,r} \cap \Gamma_N(\lambda) =  C_{N,r_N} \cap \Gamma_N(\lambda)\}.
\end{align}
Suppose that there exists $v\in \Gamma_N(\lambda)\subseteq V_N^{2\delta} $  such that $v\in C_{N,r}$ but $v\notin  C_{N,r_N}$. Then,  there exists $u$ with $r< |u-v| \leq r_N$ such that  $\pu \geq m_N-\lambda$. Since $u,v\in  V_N^\delta \cap \Gamma_N(\lambda)$ and $g(r) > \lambda$ for large enough $r$,  this contradicts the definition of  the event $\hat{E}_3$. Thus, we establish \eqref{741}.

In addition,   under the event $$G:=\{\Gamma_N(\lambda) \subseteq V_N^{2\delta}\}\cap \hat{E}_3 \cap E_1^{g(r)}$$ (recall that the event $E_1^{g(r)}$ is defined in Lemma \ref{lemma 7.7}), for any $v\in  C_{N,r} \cap   \Gamma_N(\lambda)$,
\begin{align*}
0\leq   \bar{S}_{v,r_N/2}-\bar{S}_{v,2r} &\leq  e^{\beta \lambda}    \sum_{u\in B(v,r_N/2) \setminus B(v,2r)} e^{\beta(\pu  - m_N)} \nonumber  \\
  &\overset{\eqref{743}}{\leq}   e^{\beta \lambda}    \sum_{u\in V_N: \pu < m_N-g(r)} e^{\beta(\pu - m_N)}  \overset{\eqref{event e1}}{ \leq}  C_\beta e^{\beta \lambda}  (e^{-\tau_\beta \lfloor g(r) \rfloor} + N^{-\tau_\beta } ).
\end{align*}
If $r$  is large enough so that $C_\beta e^{\beta \lambda}  \cdot e^{-\tau_\beta \lfloor g(r) \rfloor}  < \kappa/2$, then for sufficiently large $N$,
\begin{align} \label{747}
0\leq   \bar{S}_{v,r_N/2}-\bar{S}_{v,2r} &\leq \kappa.
\end{align}
This implies that  under the event  $G \cap \{|\Gamma_N(\lambda)| \leq e^{2\sqrt{2d}\lambda}\},$
\begin{align*}
 \Big\vert &\sum_{v\in C_{N,r} \cap \Gamma_N(\lambda)}  f(\pv-m_N ,\bar{S}_{v,2r}  ) - \sum_{v\in C_{N,r} \cap \Gamma_N(\lambda) }  f(\pv-m_N ,\bar{S}_{v,r_N/2}   ) \Big\vert   \\
 & \overset{\eqref{746},\eqref{747}}{\leq} |C_{N,r} \cap \Gamma_N(\lambda)| \cdot e^{-2\sqrt{2d}\lambda} \e \leq  e^{2\sqrt{2d}\lambda}   \cdot e^{-2\sqrt{2d}\lambda} \e = \e.
\end{align*}
Note that  by \eqref{410}, Theorem \ref{theorem 3.0}, Lemma \ref{lemma 7.7} and \eqref{740},
 $$\lim_{\lambda \rightarrow \infty } \liminf_{\delta \downarrow 0}  \liminf_{r\ri} \liminf_{N\ri} \mathbb{P}(G \cap \{|\Gamma_N(\lambda)| \leq e^{2\sqrt{2d}\lambda}\}) =1.$$
  Since $f$ is supported on $[-\lambda,\lambda]$ in the first variable, together  with  \eqref{741},  we conclude the proof.
}

~

Now,  let us consider the general case $f(y,z) \leq Ce^{-cy^2}$ for some $C,c>0$ and $f$ is compactly support in the second coordinate. By the argument in  \eqref{745} (note that $\phi_N$ and $\phi'_N$ have the same law),  
\begin{align*}
&\lim_{\lambda \ri} \limsup_{N\ri}  \mathbb{P}\Big( \Big\vert \sum_{v\in C_{N,r}}  f(\pv-m_N ,\bar{S}_{v,2r}  )  -   \sum_{v\in C_{N,r}}  (f(y,z) \1_{|y| \leq \lambda} ) (\pv-m_N ,\bar{S}_{v,2r}  )  \Big\vert >  \e \Big) = 0, \\
&\lim_{\lambda \ri} \limsup_{N\ri}  \mathbb{P}\Big( \Big\vert \sum_{v\in C_{N,r_N}}  f(\pv-m_N ,\bar{S}_{v,r_N/2}  )  -   \sum_{v\in C_{N,r_N}}  (f(y,z) \1_{|y| \leq \lambda} ) (\pv-m_N ,\bar{S}_{v,r_N/2}  )  \Big\vert >  \e \Big) = 0.
\end{align*}
Therefore, by \eqref{720} applied to $f(y,z)\1_{|y| \leq \lambda} $ with large enough $\lambda$,  we conclude the proof.
 
\end{proof}

\section{Characterizing limit points} \label{sec limit}
Given the invariance principle for the point process (Theorem \ref{theorem 7.1}), we furnish the proof of Theorem \ref{theorem 4.1} in this section appealing to the general theory of Liggett \cite{liggett}.

\begin{proof}[Proof of Theorem \ref{theorem 4.1}]
Let $\eta$  be any subsequential limit of  $\{\eta_{N,r_N}\}_{N\geq 1}$. 

\begin{enumerate}
\item We will first establish a non-degeneracy and a local finiteness condition satisfied by $\eta$ needed to be able to appeal to Liggett's theory. This is done by relying on the already established control on the level sets (Theorem \ref{theorem 3.0}).
\item Then using the above, we deduce that $\eta$ is a Cox process.
\item Finally, we characterize the law of the intensity measure by computing appropriate Laplace transforms.
\end{enumerate}

{Assume that  
\begin{align} \label{sub converge}
 \lim_{k\ri}  \eta_{n_k,r_{n_k}} = \eta \quad \text{in law}
\end{align}  (w.r.t. vague topology) for some subsequence $\{n_k\} _{k\in \mathbb{N}}$ in $\mathbb{N}$.}

\textbf{Step 1.}
We first show that  for any $k>0$,  almost surely,
\begin{align} \label{481}
\eta(   [-k,\infty) \times [0,\infty)) < \infty,\quad \eta(\R \times (0,\infty))>0.
\end{align}
For the first statement,  note that by  Theorem \ref{theorem 3.0}, for large enough $\lambda$,
\begin{align} \label{488}
\limsup_{N\ri} \mathbb{P}(\eta_{N,r_N}(  [-\lambda,\infty) \times [0,\infty))  >  e^{ 1.5\sqrt{2d} \lambda})  \leq e^{-\sqrt{2d} \lambda/16}.
\end{align}
To show that the same estimate holds for the limit point $\eta,$  for $M>0$, let  $f_M : \R \times [0,\infty) \rightarrow [0,1]$ be a continuous function such that $f_M=1$ on   $[-\lambda,M] \times [0,M]$ and $f_M=0$ outside  $ [-\lambda-1,M+1] \times [0,M+1] $. Then, for large enough $\lambda$,
\begin{align*}
\mathbb{P}(\eta(  [-\lambda,M] \times  [0,M])  >  e^{2\sqrt{2d}    \lambda}) & \leq \mathbb{P}(  \langle \eta,f_M\rangle >  e^{ 2\sqrt{2d} \lambda})
 \leq \liminf_{ k\ri}  \mathbb{P}(  \langle  \eta_{n_k,r_{n_k}},f_M\rangle >  e^{ 2\sqrt{2d} \lambda})\\ 
&\leq   \limsup_{ k\ri}  \mathbb{P}( \eta_{n_k,r_{n_k}} (  [-\lambda-1,\infty) \times [0,\infty))  >  e^{ 2\sqrt{2d} \lambda}) \overset{\eqref{488}}{  \leq  } e^{-\sqrt{2d} \lambda/16},
\end{align*}
{where  we used the fact $ \langle  \eta_{n_k,r_{n_k}},f_M\rangle \rightarrow  \langle  \eta,f_M\rangle$ in law (which is indeed equivalent to \eqref{sub converge}, see e.g., Section \ref{appendix a} for details), together with Portmanteau theorem, in the second inequality above.}
Since this holds uniformly in $M>0$, for large enough $\lambda > k$,
\begin{align*}
\mathbb{P}(\eta(  [-k,\infty) \times [0,\infty)) >  e^{ 2\sqrt{2d}    \lambda}) \leq  e^{-\sqrt{2d}    \lambda/16}.
\end{align*}
{This  implies the first part of \eqref{481}.}

For the second part, {similarly as in \eqref{435}, define the event 
\begin{align} \label{499}
\bar{E}_3:= \{  |u-v| \notin [r/2,2r_N], \  \forall u,v\in \Gamma_N(\lambda^2)\}.
\end{align} We claim that for any   $\lambda,r>1$,   for sufficiently large $N$, under the event  $\bar{E}_3 \cap E_1^{\lambda^2}$,}
any $v\in \Gamma_N(\lambda)$ satisfies that
\begin{align} \label{489}
1\leq \bar{S}_{v,r_N/2} \leq    (2r)^d + C_\beta  e^{\beta \lambda -\tau_\beta \lfloor  \lambda^2 \rfloor } +1
\end{align}
($C_\beta>0$ is a constant from Lemma \ref{lemma 7.7}).
To see this, note that since $v\in \Gamma_N(\lambda)\subseteq \Gamma_N(\lambda^2)$, under the event $ \bar{E}_3$,  any $u\in  B(v,r_N/2) \setminus B(v,r)$ satisfies $\pu < m_N-\lambda^2$. Thus, for sufficiently large $N$,  for  any $v\in \Gamma_N(\lambda)$,
\begin{align*}
\bar{S}_{v,r_N/2} 
&=  \sum_{u\in B(v,r)} e^{\beta (\pu - \phi_{N,v})} + \sum_{u\in B(v,r_N/2) \setminus B(v,r) } e^{\beta (\pu - \phi_{N,v})}  \\
&\leq (2r)^d + e^{\beta \lambda} \sum_{u\in B(v,r_N/2) \setminus B(v,r) } e^{\beta (\pu - m_N)}  \\
&\leq (2r)^d + e^{\beta \lambda} \sum_{\pu < m_N-\lambda^2 } e^{\beta (\pu - m_N)}  \\
&\overset{\eqref{event e1}}{\leq}   (2r)^d +  C_\beta  e^{\beta \lambda} (e^{-\tau_\beta  \lfloor  \lambda^2 \rfloor  } + N^{-\tau_\beta } )\leq  (2r)^d +C_\beta    e^{\beta \lambda -\tau_\beta   \lfloor  \lambda^2 \rfloor  } +1.
\end{align*}
 In addition,
\begin{align*}
\bar{S}_{v,r_N/2}  =  \sum_{u\in B(v,r_N/2)} e^{\beta (\pu - \phi_{N,v})}  \geq  e^{\beta (\phi_{N,v} - \phi_{N,v})} = 1.
\end{align*}

Note that  by  \eqref{meso},  {Corollary \ref{tight}} and  Lemma \ref{lemma 7.7}, for any $\e>0$,  one can take   large enough  constants $\lambda,r>1$ such that
{\begin{align} \label{480}
 \liminf_{N\ri}  \mathbb{P}( \{  |\max_{v\in V_N} \pv - m_N| \leq \lambda \} \cap \bar{E}_3 \cap E_1^{\lambda^2}  )    > 1-\e.
\end{align}}
Again to pass to $\eta,$ for such $\lambda,r>1$, set $M:=(2r)^d + C_\beta e^{\beta \lambda -\tau_\beta [\lambda^2]}+1 $, and let   $f_M : \R \times [0,\infty) \rightarrow [0,1]$ be a continuous function such that $f_M=1$ on   $[-\lambda,\lambda] \times [1,M]$ and $f_M=0$ outside  $ [-\lambda-1,\lambda+1] \times [1/2,M+1] $.  Then,  
{\begin{align*}
\mathbb{P}(\eta(  [-\lambda-1,\infty) \times  [1/2,M+1])  \geq 1 ) & \geq \mathbb{P}(  \langle \eta,f_M\rangle \geq 1) \\
& \geq \limsup_{ k\ri}  \mathbb{P}(  \langle   \eta_{n_k,r_{n_k}}  ,f_M\rangle  \geq 1) \\
&\geq   \limsup_{ k\ri}  \mathbb{P}(  \eta_{n_k,r_{n_k}}  (  [-\lambda, \lambda] \times  [1,M]) \geq 1)  \\
&\overset{\eqref{489}}{ \geq}   \limsup_{k\ri}  \mathbb{P}( \{ |\max_{v\in V_{n_k}} \pv - m_{n_k}| \leq \lambda \}\cap \bar{E}_3 \cap E_1^{\lambda^2}  ) \overset{\eqref{480}}{ \geq} 1-\e ,
\end{align*}}
{where  we used Portmanteau theorem again in the second inequality above and the event $\bar{E}_3 \cap E_1^{\lambda^2} $ is considered with $N = n_k$.}
Thus,
\begin{align*}
 \mathbb{P}(\eta(  \R \times  ( 0, \infty))  \geq 1 )  \geq 1-\e.
\end{align*}
Since $\e>0$ is arbitrary, this concludes the proof of the second statement of  \eqref{481}.\\

We will now use the invariance principle in Theorem \ref{theorem 7.1} to deduce that $\eta$ is a Cox process, {following the approach in \cite{bl1}}.\\

\textbf{Step 2. Characterization in terms of a Cox process.}
{Recall the Markov kernel $P_t$ on $\R \times [0,\infty)$  defined in \eqref{markov kernel}:} For $t>0$,
\begin{align} \label{kernel}
P_t((y,z), A) := \mathbb{P}((y+B_t - \sqrt{d/2}  t, z) \in A) ,\quad  A\subseteq \R\times [0,\infty) \  \text{Borel}.
\end{align}
For any non-negative measurable function $g$  and a (non-negative) measure $\mu$  on $\R \times [0,\infty)$, define the push-forwarded function $P_tg$ and measure $\mu P_t$ by 
\begin{align*} 
P_t g (y,z):= \mathbb{E} g(y+B_t - \sqrt{d/2}  \ t,z),\quad (\mu  P_t)(A):= \int P_t( (y,z),A) \mu( dy \  dz) .
\end{align*}
Then, it is straightforward to check that (by first verifying it for indicators, then for simple functions and appealing to Monotone convergence theorem)
\begin{align} \label{491}
\int (P_t g) (y,z)  \mu(dy \ dz)    = \int  g (y,z)  (\mu P_t)(dy \ dz).
\end{align} 
By the definition of $f_t$ in \eqref{700}, we have $1-e^{-f_t}= P_t(1-e^{-f})$. Thus, by the preceding display,
\begin{align} \label{487}
 \int (1-e^{-f_t(y,z)}) \mu(dy \ dz) =\int (1-e^{-f(y,z)})  (\mu P_t)(dy \ dz) .
\end{align}
In addition, due to the  negative linear drift $-\sqrt{d/2}t$ in \eqref{kernel} and the diffusivity of $B_t$, for any compact $C \subseteq \R\times [0,\infty)$,
\begin{align*}
\lim_{t\ri} \sup_{(y,z) \in \R \times [0,\infty)} P_t ((y,z),C)  = 0.
\end{align*}
This implies that for any $f: \R \times [0,\infty) \rightarrow [0,\infty)$ of compact support,
\begin{align*}
\lim_{t\ri} \sup_{(y,z) \in \R \times [0,\infty)} | 1  -  e^{-f_t(y,z)}|  = \lim_{t\ri} \sup_{(y,z) \in \R \times [0,\infty)} | 1  - \mathbb{E} e^{-f(y+B_t -  \sqrt{d/2} t ,z) }| = 0.
\end{align*}
{Using the fact that $\forall \e>0$, $\exists  \ \delta>0$ such that $1-e^{-x} \geq (1-\e)x$ for $x\in [0,\delta]$,
there exist constants $\kappa_t>0$ with $\lim_{t\ri} \kappa_t = 0$ such that for any $y\in \R$ and $z\in [0,\infty)$,
\begin{align*}
 1  -  e^{-f_t(y,z)} \leq   f_t(y,z) \leq (1+\kappa_t)( 1  -  e^{-f_t(y,z)}).
\end{align*}
Since $\eta \geq 0$, this implies the existence of (random)  $\tilde{\kappa}_t$ with $ 0 \leq \tilde{\kappa}_t \leq \kappa_t$ such that
\begin{align*}
\langle \eta, f_t \rangle   = (1+\tilde{\kappa}_t) \langle \eta, 1-e^{-f_t}  \rangle   .
\end{align*}}
Thus, for any  continuous $f: \R \times [0,\infty) \rightarrow [0,\infty)$ of compact support, 
{\begin{align} \label{484}
\mathbb{E} e^{-\langle \eta , f\rangle} \overset{\eqref{710}}{ =} \mathbb{E} e^{-\langle \eta , f_t\rangle}  &=  \mathbb{E}\exp \Big( - (1+\tilde{\kappa}_t ) \int  (  1-e^{-f_t (y,z)}  ) \eta (dy  \ dz) \Big) \nonumber \\
&\overset{\eqref{487}}{ =} \mathbb{E}\exp \Big( -(1+\tilde{\kappa}_t )\int (1-e^{-f(y,z)}) (\eta P_t) (dy  \ dz) \Big).
\end{align}}
We now show that $\{\eta P_t \}_{t\geq 0}$ is tight. 
We will apply a, by now, standard analysis argument. For any continuous $g:\R \times [0,\infty) \rightarrow [0,1)$ with a compact support, $f: = -  \log (1-g)  \geq 0$ is also continuous function of compact support. Taking $f = -  \log (1-g)$  in \eqref{484},
\begin{align*}
\mathbb{E} e^{ \langle \eta , \log (1-g) \rangle} = \mathbb{E} e^{- (1+\tilde{\kappa}_t)  \langle  \eta P_t, g \rangle}.
\end{align*}
{Replacing $g$ by $\chi g$ with $\chi  \in  [0,1]$,
\begin{align} \label{495}
\mathbb{E} e^{ \langle \eta , \log (1-\chi g) \rangle} = \mathbb{E} e^{- (1+\tilde{\kappa}_t) \chi  \langle  \eta P_t, g \rangle}.
\end{align}
Recalling $g\geq 0$, by monotone convergence theorem,
\begin{align} \label{493}
\lim_{\chi \downarrow 0} \mathbb{E} e^{ \langle \eta , \log (1-\chi g) \rangle}=1
\end{align}
For any $\e>0$, take a sufficiently small $\chi \in [0,1]$ such that  LHS of \eqref{495} is at least $1-\e$. For such $\chi$, for any $K>0$,
\begin{align} \label{494}
1-\e \leq \mathbb{E} e^{ \langle \eta , \log (1-\chi g) \rangle}  \overset{\eqref{495}}{=}  \mathbb{E} e^{- (1+\tilde{\kappa}_t) \chi  \langle  \eta P_t, g \rangle} \leq  \mathbb{P} ( \langle  \eta P_t, g \rangle < K  ) + e^{-\chi K}, 
\end{align}
where we used the fact $\tilde{\kappa}_t \geq 0$ in the last inequality.
Hence, there exists a large enough constant $K>0$ such that  $\mathbb{P} ( \langle  \eta P_t, g \rangle < K  ) \geq 1-2\e$ for  $t\geq 0$.
Since $\langle \eta P_t,g \rangle\geq 0$, this shows the tightness of $\{\langle \eta P_t,g \rangle \}_{t \geq 0}$ for  an  arbitrary continuous  and compactly supported function $g:\R \times [0,\infty) \rightarrow [0,1)$.
This ensures the existence of a subsequence $\{t_n\}_{n\in \mathbb{N}}$ in $\R$  with $\lim_{n\ri} t_n=\infty$ and a random Borel measure $M$ on $\R \times [0,\infty)$ such that $\lim_{n\ri} \eta P_{t_n}  = M$ in law  w.r.t. vague topology (see Lemma \ref{tight converge}  for a criteria for the relative compactness of random measures).}
{Let us take a limit $t\ri$ in \eqref{484} along the sequence  $\{t_n\}_{n\in \mathbb{N}}$.  Since $\lim_{n\ri} \eta P_{t_n}  = M$ w.r.t. vague topology and  noting that $1-e^{-f}$ is continuous and compactly supported, $\langle  \eta P_{t_n}, 1-e^{-f} \rangle  \rightarrow \langle  M, 1-e^{-f} \rangle$ in law as $n\ri$. Thus, using the fact that  $ 0 \leq \tilde{\kappa}_t \leq \kappa_t$ with  $\lim_{t\ri} \kappa_t = 0$,  
\begin{align*}
\mathbb{E} e^{-\langle \eta , f\rangle}  = \mathbb{E}\exp \Big( -\int (1-e^{-f(y,z)}) M (dy  \ dz) \Big).
\end{align*}}
{Note that RHS above is the Laplace functional of  $ \text{PPP}(M)$ (see  \eqref{902} in  Section \ref{appendix a} for a further elaboration). Since the Laplace functional tested against all  continuous  and compactly supported  functions uniquely determines the law of random measures (see \eqref{900} in  Section \ref{appendix a} for the precise statement), we deduce that $\eta \overset{\text{law}}{\sim}   \text{PPP}(M)$.}  {Note that  by \eqref{481}, for any $k>0$, almost surely,
\begin{align} \label{482}
M([-k,\infty) \times [0,\infty)) < \infty,\quad  M(\R \times (0,\infty)) > 0.
\end{align}
}

Now, we claim that for any $t>0$,
\begin{align} \label{485}
M P_t =  M \quad \text{in law}.
\end{align} 
{Since  $\eta \overset{\text{law}}{\sim}   \text{PPP}(M)$,  for any continuous and compactly supported function $f:\R \times  [0,\infty) \rightarrow [0,\infty)$,
\begin{align*}
\mathbb{E} e^{-\langle \eta , f\rangle} \overset{\eqref{710}}{ =}  \mathbb{E}e^{-\langle \eta,f_t \rangle}& =\mathbb{E} \exp\Big(-\int (1-e^{-f_t(y,z)}) M(dy \ dz)\Big) \\
& \overset{\eqref{487}}{ =} \mathbb{E} \exp\Big(-\int (1-e^{-f(y,z)})  (MP_t)(dy \ dz)\Big) . 
\end{align*}
Since the RHS above is the Laplace functional of $ \text{PPP}(MP_t)$, by the injectivity of the Laplace functional mentioned above, it follows that  $\eta \overset{\text{law}}{\sim}  \text{PPP}(MP_t)$. Recalling $\eta \overset{\text{law}}{\sim}  \text{PPP}(M)$ and the fact that the law of Poisson processes uniquely determines the law of intensity measures (see Lemma \ref{poisson unique} for the proof), we establish \eqref{485}.}

~

\textbf{Step 3. Characterization of the intensity measure.}
Finally, we characterize the law of $M$. Since the second coordinate in $\R \times [0,\infty)$ is not moved by  $P_t$ (see \eqref{kernel}), we consider the projected Markov kernel on the first coordinate as follows.
For any Borel set $A_2\subseteq [0,\infty)$, define a random measure $\bar{M}_{A_2}$  on $\R$ by
\begin{align*}
\bar{M}_{A_2}(A_1) := M(A_1\times A_2),\quad A_1\subseteq \R \  \text{Borel}.
\end{align*}
In addition, define the  Markov kernel $Q_t$ on $\R$ as
\begin{align} \label{q}
Q_t(y,A_1) = \mathbb{P}(y+ B_t -\sqrt{d/2}   t  \in A_1),\quad y\in \R, \  A_1\subseteq \R \  \text{Borel}.
\end{align}
Then, since $MP_t = M$ in law (see \eqref{485}) and $P_t$ does not change the second coordinate, for  any Borel set  $A_2\subseteq [0,\infty)$,
\begin{align} \label{470}
\bar{M}_{A_2} Q_t  = \bar{M}_{A_2}\quad \text{ in law}.
\end{align} In order to deduce that  $\bar{M}_{A_2} Q_t  = \bar{M}_{A_2}$ almost surely, we invoke \cite[Corollary 3.8]{liggett}: Let $\mu$ be a random measure and $P$ be a Markov kernel   on a non-compact Abelian group $S$ satisfying $P(x,E) = P(0,E-x)$ for any $x\in S$ {and measurable $E\subseteq S$} such that the kernel $P$ has no proper closed invariant subgroups. Then, 
$\mu P \overset{\text{law}}{\sim} \mu$ implies that $\mu P = \mu$ almost surely.

We apply this fact to the Markov kernel $Q_t$ (on the non-compact Abelian group $\mathbb{R}$) for each $t>0$. By \eqref{q}, $Q_t$ is translation invariant. {To see that $Q_t$ does not have proper closed invariant subgroups, we note that any subgroup of $\R$ is either dense in $\R$ or given by $a\mathbb{Z}$ for some $a\in \R$. The former case of subgroup is excluded since its closure is an entire set $\R$, and the latter case of subgroup cannot be  invariant under $Q_t$ {due to the absolute continuity of the law of $Q_t (y,\cdot)$}.} Therefore,    \cite[Corollary 3.8]{liggett} is applicable and by \eqref{470}, we deduce that 
\begin{align} \label{471}
 \bar{M}_{A_2} Q_t  = \bar{M}_{A_2}\quad \text{ almost surely.} 
\end{align}

Next, we characterize the distribution of $\bar{M}_{A_2} $ by applying \cite[Lemma 3.3]{bl1}: Let $\tilde{Q}_t$ be the Markov kernel defined similarly as in \eqref{q} (with the drift coefficient changed)
\begin{align*}
\tilde{Q}_t(y,A_1) = \mathbb{P}\Big(y+ B_t -  \frac{\alpha}{2} t  \in A_1\Big),\quad y\in \R, \  A_1\subseteq \R \  \text{Borel}.
\end{align*} 
Then, any Borel measure $\mu$ on $\R$ with $\mu([0,\infty))<\infty$  satisfying $\mu \tilde{Q}_t = \mu$  for some $t>0$  is given by $\mu(dy) = \alpha\mu ( [0,\infty))e^{- \alpha y}dy$. 
In our case, \cite[Lemma 3.3]{bl1} (with $\alpha = \sqrt{2d}$ and $\mu = \bar{M}_{A_2}$)  is applicable since $\bar{M}_{A_2}([0,\infty)) < \infty$ by  \eqref{482}.
Hence,  by \eqref{471}, we deduce that 
 almost surely, $$\bar{M}_{A_2} (dy) = \sqrt{2d} \bar{M}_{A_2}([0,\infty)) e^{-\sqrt{2d}y}dy.$$

 Now, define a (random) Borel measure $\nu$ on $[0,\infty)$ by
\begin{align*}
\nu(A_2) :=\sqrt{2d} \bar{M}_{A_2}( [0,\infty)) = \sqrt{2d} M( [0,\infty) \times A_2),\quad \forall A_2\subseteq [0,\infty) \  \text{Borel}.
\end{align*}
Note that $\nu ( [0,\infty) )<\infty$  by \eqref{482}.
Then, almost surely (whose exceptional zero-measure set may depend on $A_2$),
\begin{align} \label{483}
M(A_1 \times A_2) =\bar{M}_{A_2}(A_1)= \nu(A_2) \int_{A_1} e^{-\sqrt{2d}y}dy,\quad \forall A_1\subseteq \R \  \text{Borel}.
\end{align}
{Let $\mathcal{S} = \{ [0,q) : q\in  \mathbb{Q}, q>0\} $ be  the countable collection of intervals generating the Borel $\sigma$-algebra on $[0,\infty)$.} Then, by \eqref{483}, almost surely,  for any fixed bounded Borel $A_1$ in $\R$, the  two measures
\begin{align} \label{492}
A_2 \mapsto M(A_1 \times A_2)\ \text{and} \ A_2 \mapsto  \nu(A_2) \int_{A_1} e^{-\sqrt{2d}y}dy ,\quad  A_2\subseteq [0,\infty) \  \text{Borel}.
\end{align}
agree for $A_2 \in \mathcal{S}$.
These two measures are  finite (due to boundedness of $A_1$ and \eqref{482}) and  agree on $\mathcal{S}$ which is a $\pi$-system generating the Borel $\sigma$-algebra  in $[0,\infty)$. Hence, as an application of the $\pi-\lambda$ theorem {(see \cite[Theorem A.1.5]{durrett}, which says that if $\mu_1$ and  $\mu_2$ are measures that agree on the $\pi$-system $\mathcal{P}$ and there is a sequence $\{A_n\}_{n\in \mathbb{N}}$ such that  $A_n\in \mathcal{P}$ with $A_n \uparrow \Omega$ and $\mu_i (A_n)<\infty$ ($i=1,2$), then  $\mu_1$ and  $\mu_2$ agree on $\sigma(\mathcal{P})$)}, they agree on all Borel subsets of $[0,\infty)$.
Therefore, almost surely, \eqref{483} holds for all Borel $A_2 \subseteq [0,\infty)$ and  \emph{all} bounded Borel $A_1\subseteq \R$ ({since \eqref{483} holds for \emph{all} Borel sets $A_1$).}

Now, let us remove the bounded condition on $A_1$. Consider the collection of intervals $\mathcal{S}' = \{ (q_1,q_2) : -\infty < q_1<q_2<\infty, q_1,q_2\in  \R\} $ which is a $\pi$-system  and generates the Borel $\sigma$-algebra on $\R$.  It has already been shown that for each Borel $A_2\subseteq [0,\infty)$, the two    measures 
\begin{align*}
A_1 \mapsto M(A_1 \times A_2)\ \text{and} \ A_1 \mapsto  \nu(A_2) \int_{A_1} e^{-\sqrt{2d}y}dy,   \quad A_1\subseteq \R \  \text{Borel}.
\end{align*}
agree for $A_1\in \mathcal{S}'$. Hence, by the $\pi-\lambda$ theorem again, noting that these measures are finite on each $(-m,m)\in \mathcal{S}'$   for $m\in \mathbb{N}$, they agree on all Borel $A_1\subseteq \R$.

Therefore, we establish that  almost surely, for any Borel $A_1\subseteq \R$  and $A_2 \subseteq [0,\infty)$, the relation \eqref{483} holds. This implies that
\begin{align*}
M(dy  \ dz) = e^{-\sqrt{2d}y}dy \otimes \nu(dz).
\end{align*} 
Finally, by \eqref{482},  $\nu((0,\infty)) \in (0,\infty)$ almost surely.

\end{proof}

{
\begin{remark}\label{spatial}The reader might notice that a similar style of argument may be used to analyze the first and the second coordinates in the point process \eqref{08} and conclude that any subsequential limit (w.r.t. vague topology) of the point process
\begin{align*}
\sum_{v\in C_{N,r_N} }   \delta_{v/N}  (dx) \otimes  \delta_{\pv - m_N} (dy)
\end{align*}
has the law $\text{PPP}(\mathcal{Z}(dx) \otimes e^{-\sqrt{2d}y}dy )$ for some random Borel measure $\mathcal{Z}(dx)$ on $[0,1]^d$.
  Indeed, the argument in Section \ref{section 4} goes through to prove a corresponding invariance principle and further, the argument in this section can be used to deduce such characterization in terms of  the Cox process. Moreover, as in \cite{drz}, under additional conditions on the covariance structure (ensuring convergence of the covariance structure microscopically near the diagonal and macroscopically off-diagonally) which ensures a precise characterization of the limiting law of the centered maximum, one could prove that the random measure $\mathcal{Z}(dx)$ is uniquely characterized in law, which together with tightness implies the full convergence
    \begin{align} \label{spatial1}
\lim_{N\ri}  \sum_{v\in C_{N,r_N} }   \delta_{v/N}  (dx) \otimes  \delta_{\pv - m_N} (dy) = \text{PPP}(\mathcal{Z}(dx) \otimes e^{-\sqrt{2d}y}dy )\quad \text{in law}.
  \end{align} 
In addition, the random measure $\mathcal{Z}(dx)$ coincides with the derivative martingale, exactly as in the case of the maximum.
\end{remark}
}

\section{Proof of the main theorem} \label{section main theorem}
Having Theorem \ref{theorem 4.1} at hand, we first prove  the convergence result in Theorem \ref{theorem 1.1} for the sequence of cluster weights 
around $r_N$-local extrema points, i.e.   
\begin{align*}
S_{v,r_N/2}:= \sum_{u\in B(v,r_N/2)} e^{\beta(\pu-m_N)}
\end{align*} 
normalized by their sum, say $S.$ (Note the distinction between the above notation and $\bar{S}_{v,r}$ from \eqref{sbar}). The advantage of normalizing by $S$ is that it is a continuous function of the cluster weights and hence Theorem \ref{theorem 4.1} along with a continuous mapping theorem application yields the convergence. To upgrade to the full theorem, it then remains to show that $S$ is close to the full partition function $\sum_{u\in V_N } e^{\beta(\pu-m_N)}$ which is achieved by an application of  Lemma \ref{lemma 7.7}.

\begin{proof}[Proof of Theorem \ref{theorem 1.1}]
$\empty$\\
\textbf{Step 1. Normalized cluster weights converges to PD.} 
We need to set up some further notations. First, let $\{X_i\}_{i\in \mathbb{N}}$ and $\{Y_{N,v}\}_{v\in V_N}$ (for each $N$) denote i.i.d. $\text{Unif}[0,1]$ (uniform distribution on $[0,1]$) independent of everything else. Next, let $\{q_i\}_{i\in \mathbb{N}}$ denote, in decreasing order, the sample points of the Poisson point process on $[0,\infty)$ with intensity measure $x^{-1-\beta_c/\beta}dx$  independent of everything else (Observe that the intensity measure is integrable on $[\e,\infty)$ for any positive $\e$ and consequently, almost surely, the  point process has only finitely many points in $[\e, \infty)$).  We next proceed to showing that as $N\ri$, we have the following convergence of the normalized weights.
\begin{align} \label{101}
 \sum_{v\in C_{N,r_N}} \Big( \frac{S_{v,r_N/2}}{\sum_{u\in C_{N,r_N}} S_{u,r_N/2} } \Big) \delta_{Y_{N,v}} \rightarrow \sum_{i\in \mathbb{N}} \Big(\frac{q_i}{ \sum_{j\in \mathbb{N}} q_j}\Big) \delta_{X_i} \quad \text{in law}
\end{align}
(w.r.t. weak topology on probability measures on $[0,1]$).

We first claim that for any subsequence $\{a_k\}_{k\geq 1}$  in $\mathbb{N}$, there exists a further subsequence $\{b_k\}_{k\geq 1}$ along which \eqref{101} holds. To accomplish this, we analyze the un-normalized weights
\begin{align} \label{120}
 \sum_{v\in C_{N,r_N}}  S_{v,r_N/2}\delta_{Y_{N,v}} .
\end{align}
Recalling $\eta_{a_k,r_{a_k}}$ from
\eqref{point process}, by Theorem \ref{theorem 3.0}, for any continuous and compactly supported function $f: \R \times [0,\infty) \rightarrow  \R,$ the sequence  of random variables $\{ \langle \eta_{a_k,r_{a_k}},f\rangle\}_{k\geq 1}$ is tight. {To see this, assume that the first coordinate of $f$ is supported on $[-K,K]$.     For any $\delta>0$, take $t_0>K$ large enough so that $e^{- \sqrt{2d} t_0/8} < \delta$. Theorem \ref{theorem 3.0} says that  for  large enough $t\geq t_0$, for  sufficiently large $N$,
$
  \mathbb{P}(|\Gamma_N(t) |  \geq  e^{2\sqrt{2d} t}) \leq  e^{- \sqrt{2d} t/8}<\delta.
$ This implies that for any such $t$, $\mathbb{P}(|\langle \eta_{a_k,r_{a_k}},f\rangle |   \geq  \norm{f}_\infty\cdot   e^{2\sqrt{2d} t} )  < \delta$ for large  $k$, which verifies the tightness. }

{This implies the existence of  a subsequence $\{b_k\}_{k\geq 1}$  of  $\{a_k\}_{k\geq 1}$   and a random Borel measure $\eta$  on $\R \times [0,\infty)$ such that  
$\eta_{b_k,r_{b_k}} \rightarrow \eta$ in law as $k\ri$  w.r.t.  vague topology  (see Lemma \ref{tight converge}).}
By Theorem \ref{theorem 4.1}, $\eta =  \textup{PPP}( e^{-\sqrt{2d}y}dy \otimes \nu (dz))$ 
 for some random  Borel measure $\nu$ on $[0,\infty)$ such that $\nu((0,\infty)) \in (0,\infty)$ a.s (note that $\nu$ may depend on the subsequence $\{b_k\}_{k\geq 1}$).   Then, the independence of  $\{Y_{N,v}\}_{v\in V_N}$ and the field $\phi_N$ can be used to show that  
{ \begin{align} \label{121}
\lim_{k\ri}  \sum_{v\in C_{b_k,r_{b_k}}}  \delta_{Y_{N,v}} (dx) \otimes   \delta_{\pv - m_{b_k}} (dy) \otimes \delta_ {\bar{S}_{v,r_{b_k}/2} } (dz) =   \textup{PPP}( \text{Leb} (dx) \otimes e^{-\sqrt{2d}y}dy \otimes \nu (dz))
\end{align}
in law
(w.r.t vague topology on   measures on $[0,1] \times \R \times [0, \infty)$). This is a consequence of the fact that if $\{x_i \}_{i\in \mathbb{N}}$ form the sample points in $\text{PPP}(\mu_1 )$ and $\{y_i\}_{i\in \mathbb{N}}$ are i.i.d. distributed as a probability measure $\mu_2$ independent of $\{x_i \}_{i\in \mathbb{N}}$, then $\{(x_i,y_i)\}_{i\in \mathbb{N}}$ has the same law as the  enumeration of the sample points of $\text{PPP}(\mu_1 \otimes \mu_2)$ (see Lemma \ref{poisson joint}    for details)}.

{The convergence result \eqref{121} can be used to find the limit of \eqref{120}, by applying a suitable test function  which we describe now. 
For any $M \in \mathbb{N}$, let $g_M: \R \rightarrow [0,1] $ and $h_M:[0,\infty) \rightarrow [0,1]$ be continuous functions of compact support satisfying $g_M \uparrow 1$, $h_M \uparrow 1$ as $M\to \infty$ such that}
\begin{align}
g_M (y) = 1 \  \text{for}  \ y\in  [-M,M],& \quad g_M (y) = 0   \  \text{for}  \  y\notin [ -M-1,M+1], \label{g}  \\
h_M (z) = 1 \  \text{for}  \ z\in [0,M], &\quad h_M (z) = 0   \  \text{for}  \  z\notin  [0,M+1]. \label{h}
\end{align}
Let $ f:[0,1]\rightarrow [0,\infty)$ be  any continuous function.  We aim to consider the  Laplace functional of point processes in the  convergence \eqref{121} against a continuous and compactly supported function $F_M : [0,1] \times \R \times [0,\infty) \rightarrow [0,\infty)$ defined as
\begin{align*}
F_M(x,y,z) :=f(x) \cdot e^{\beta y} g_M(y) \cdot  z h_M(z).
\end{align*} 
 Since the convergence of random measures in law  (w.r.t. vague topology) is equivalent to the convergence of the corresponding Laplace functional tested against  continuous and compactly supported functions (a formal statement appears as \cite[Theorem 4.11]{random}  {which is recorded later in Section \ref{appendix a}}),  
{ \begin{align}  \label{104}
 \lim_{k\ri}& \mathbb{E}\exp\Big(- \sum_{v\in C_{b_k,r_{b_k}}} S_{v,r_{b_k}/2}  f(Y_{b_k,v}) g_M(\phi_{b_k,v}- m_{b_k})  h_M (\bar{S}_{v,r_{b_k}/2}) \Big) \nonumber\\
 &=  \mathbb{E} \exp \Big(-\int (1-e^{-f(x)\cdot e^{\beta y} g_M(y) \cdot  z h_M(z) } )dx \otimes e^{-\sqrt{2d}y}dy \otimes \nu (dz)\Big).
 \end{align}}
 In order to take $M\ri$ above and then interchange  two limits $k\ri$ and $M\ri$,  we verify that {for any $\e>0$, there exists $M_0$ and $N_0$ such that for any $N\geq N_0$ and $M\geq M_0$,
 \begin{align} \label{100}
\Big\vert  \mathbb{E}&\exp\Big(- \sum_{v\in C_{N,r_N}} S_{v,r_N/2} f(Y_{N,v})\Big)  \nonumber  \\
&-   \mathbb{E}\exp\Big(- \sum_{v\in C_{N,r_N}} S_{v,r_N/2}   f(Y_{N,v})  g_M(\pv- m_N)  h_M (\bar{S}_{v,r_N/2}) \Big)  \Big\vert <\e.
 \end{align}}
 This can be obtained by the tightness of the centered maximum (Corollary \ref{tight}), and a uniform bound on the quantity $\bar{S}_{v,r_N/2}$ deduced in   \eqref{489}, whose detailed proof is deferred to the end of proof of this theorem (see Step 3). Assuming \eqref{100},  taking $M\ri$ in  \eqref{104}  and then interchanging two limits $k\ri$ and $M\ri$,
\begin{align*}
 \lim_{k\ri} \lim_{M\ri} & \mathbb{E}\exp\Big(- \sum_{v\in C_{b_k,r_{b_k}}} S_{v,r_{b_k}/2}  f(Y_{b_k,v}) g_M(\phi_{b_k,v}- m_{b_k})  h_M (\bar{S}_{v,r_{b_k}/2})\Big) \nonumber\\
 &= \lim_{M\ri}  \mathbb{E} \exp \Big(-\int (1-e^{-f(x) \cdot e^{\beta y}  g_M(y) \cdot  z h_M(z) } )dx \otimes e^{-\sqrt{2d}y}dy \otimes \nu (dz)\Big).
\end{align*}
Applying the monotone convergence theorem (recall that $0\leq g_M \uparrow 1$ and $0\leq h_M \uparrow 1$ as $M\ri$),  
  \begin{align} \label{1020}
 \lim_{k\ri}  & \mathbb{E}\exp\Big(- \sum_{v\in C_{b_k,r_{b_k}}} S_{v,r_{b_k}/2} f(Y_{b_k,v})\Big) \nonumber\\
 &=  \mathbb{E} \exp \Big(-\int (1-e^{-f(x)e^{\beta y}  z} )dx \otimes e^{-\sqrt{2d}y}dy \otimes \nu (dz)\Big) \nonumber \\
  &=    \mathbb{E} \exp \Big(-\beta^{-1}\int (1-e^{-f(x)tz } )dx \otimes t^{-1-\beta_c/\beta} dt \otimes  \nu (dz) \Big)  \nonumber  \\ 
&=  \mathbb{E} \exp \Big(-\beta^{-1} \int z^{\beta_c/\beta} \nu (dz)  \cdot   \int (1-e^{-f(x)s}  ) dx\otimes  s^{-1-\beta_c/\beta} ds\Big).
 \end{align}

Set $Z:= \beta^{-1} \int  z ^{\beta_c/\beta} \nu (dz)$.  We verify that $Z\in (0,\infty)$ almost surely, which essentially follows from the uniform bound (from below and above) on the quantity ${S}_{v,r_{N}/2}$. First,  the strict positivity follows from the fact that  $\nu((0,\infty)) >0$. Next, we show the finiteness of $Z$ using a similar argument as in \eqref{495}-\eqref{494}. Taking $f \equiv a$ ($a>0$ is a constant) in \eqref{1020}, after a change of variable in the second integral in RHS of \eqref{1020},
\begin{align}  \label{126}
 \lim_{k\ri}  & \mathbb{E}\exp\Big(- a\sum_{v\in C_{b_k,r_{b_k}}} S_{v,r_{b_k}/2}   \Big) =  \mathbb{E} \exp(- a^{\beta_c/\beta}cZ),
\end{align}  
where $c :=    \int_0^\infty  (1-e^{-t}  )  t^{-1-\beta_c/\beta} dt    \in (0,\infty)$. Note that for any $a>0$,
\begin{align} \label{128}
\mathbb{E}   \exp(- a^{\beta_c/\beta}cZ) \leq   \mathbb{P}(Z<\infty).
\end{align} 
We aim to lower bound LHS of \eqref{126} by upper bounding the quantity   $\sum_{v\in C_{N,r_{N}}} S_{v,r_{N}/2} $, uniformly in $N$. Under the event  $$ \{\max_{v\in V_N} \pv \leq  m_N+\lambda\}\cap \{ |\Gamma_N(\lambda)| \leq e^{2\sqrt{2d}\lambda} \} \cap E_1^\lambda$$
(whose probability tends to 1 as $N\ri$ followed by  $\lambda\ri$, due to Corollary \ref{tight}, Theorem \ref{theorem 3.0}  and Lemma \ref{lemma 7.7}), 
\begin{align*}
&\sum_{v\in C_{N,r_{N}}} S_{v,r_{N}/2}  \leq \sum_{v\in V_N} e^{\beta(\pv - m_N)} \\
&= \sum_{v\in V_N: \pv < m_N-\lambda} e^{\beta(\pv - m_N)} +  \sum_{v\in V_N: |\pv-m_N| \leq \lambda } e^{\beta(\pv - m_N)}+ \sum_{v\in V_N: \pv > m_N + \lambda } e^{\beta(\pv - m_N)} \\
&\leq   C_\beta (e^{-\tau_\beta \lfloor \lambda \rfloor} + N^{-\tau_\beta }  ) + e^{\beta \lambda} \cdot e^{2\sqrt{2d}\lambda} .
\end{align*}
This implies that for any $\e>0$, there is $K>0$ such that for sufficiently large $N$,
\begin{align*}
\mathbb{P}\Big(\sum_{v\in C_{N,r_{N}}} S_{v,r_{N}/2} >K\Big)\leq \e,
\end{align*}
which implies
\begin{align}  \label{127}
 \mathbb{E}\exp\Big(- a\sum_{v\in C_{N,r_{N}}} S_{v,r_{N}/2}   \Big) \geq e^{-aK}(1-\e).
\end{align}
By taking a small enough $a>0$ in \eqref{126}, using \eqref{128} and \eqref{127} (applied to the sequence $\{b_k\}_{k\geq 1}$), we obtain   $\mathbb{P}( Z<\infty) \geq 1-2\e.$
Since $\e>0$ is arbitrary, we deduce that $\mathbb{P}( Z<\infty)=1.$

~

Now, from \eqref{1020}, one can compute the limit of \eqref{120}. The quantity \eqref{1020} is the Laplace functional of the (weighted) point process
\begin{align*}
   Z^{\beta/\beta_c} \sum_{i\in \mathbb{N}} q_i \delta_{X_i}  
\end{align*}
against a test function $f$ (this is recorded precisely as Lemma \ref{poisson}). 
Since the convergence of the Laplace functional against continuous and compactly supported functions  is equivalent to the  convergence of the corresponding random measures in law w.r.t. vague topology,
\begin{align} \label{123}
\lim_{k\ri} \sum_{v\in C_{b_k,r_{b_k}}} S_{v,r_{b_k}/2} \delta_{Y_{N,v}} =  Z^{\beta/\beta_c} \sum_{i\in \mathbb{N}} q_i \delta_{X_i}   \quad \text{in law}
\end{align}
{(w.r.t.  vague  topology on measures on $[0,1]$).
Note that by the definition of the vague topology on the  space of measures on $[0,1]$, the map $\mu \mapsto \int_{[0,1]} \1 d\mu =  \mu ([0,1])$ is continuous. Thus, applying the continuous mapping theorem to the continuous function $\mu \mapsto \mu / \mu([0,1])$ together with  \eqref{123} (note that  the denominator, the total mass of the measure, is non-zero since $S_{v,r} >0$ ($ \forall v,r$), $Z>0$ and $\sum_{i\in \mathbb{N}}q_i>0$  almost surely),} it follows that
\begin{align} \label{125}
\lim_{k\ri}  \sum_{v\in C_{b_k,r_{b_k}}} \Big( \frac{S_{v,r_{b_k}/2} }{\sum_{u\in  C_{b_k,r_{b_k}}} S_{u,r_{b_k}/2} } \Big) \delta_{Y_{N,v}}   = \sum_{i\in \mathbb{N}}  \Big(\frac{q_i}{ \sum_{j\in \mathbb{N}} q_j}\Big) \delta_{X_i} \quad \text{in law}.
\end{align}

Since every subsequence $\{a_k\}_{k\geq 1}$ in $\mathbb{N}$ is shown to have a further subsequence $\{b_k\}_{k\geq 1}$ along which \eqref{125} holds, we have the full convergence \eqref{101} (this convergence of random measures in law can be metrized by  L\'evy-Prokhorov  metric on probability measures on the set of measures equipped with vague topology which is a Polish space.)

Given the above, it remains to show that \eqref{101} can be translated to the convergence of the ordered sequence of weights:
\begin{align} \label{155}
\lim_{N\ri} \mathsf{Ord} \Big(\Big(   \frac{S_{v,r_N/2}}{\sum_{u\in C_{N,r_N}} S_{u,r_N/2} }  \Big)_{v\in C_{N,r_N}} \Big)  =  \mathsf{Ord}\Big(  \Big( \frac{q_i}{ \sum_{j\in \mathbb{N}} q_j} \Big)_{i\in \mathbb{N}} \Big) \quad \text{ in law}
\end{align}
w.r.t.   $\ell^1$-distance between sequences. This can be obtained by choosing suitable functions tested against the point measures, whose detailed proof is deferred to the end of proof (see Step 4).\\

~
The next step seeks to replace the denominator on the LHS above by the full partition function.\\

\textbf{Step 2. The normalized and Gibbs weights are close.}  
Let $S$ be the whole sum of the exponential weights
\begin{align*}
S:= \sum_{v\in V_N}e^{\beta(\pv-m_N)}.
\end{align*}
We show that for any $\iota>0$,
\begin{align} \label{103}
 \lim_{N\ri} \mathbb{P}\Big(  \ell^1 \Big( \Big(\frac{S_{v,r_N/2}}{\sum_{u\in C_{N,r_N}} S_{u,r_N/2} } \Big)_{v\in C_{N,r_N}}, \Big(\frac{S_{v,r_N/2}}{S} \Big)_{v\in C_{N,r_N}} \Big) > \iota \Big)  = 0.
\end{align}
Note that the  $\ell^1$-distance between the above two sequences is given by 
\begin{align} \label{108}
\Big (\frac{1}{\sum_{u\in C_{N,r_N}} S_{u,r_N/2}}- \frac{1}{S}\Big) \sum_{v\in C_{N,r_N}} S_{v,r_N/2} =  \frac{1}{S}\sum_{u\in B(C_{N,r_N},r_N/2)^c}e^{\beta(\pu-m_N)} ,
\end{align}
where  
\begin{align*}
B(C_{N,r_N},r_N/2) := \{u\in V_N: u\in B(v,r_N/2) \ \text{for some} \ v\in C_{N,r_N}\}.
\end{align*}
We first claim that   for any $\iota'>0$,
\begin{align} \label{105}
 \lim_{N\ri} \P\Big( \sum_{u\in B(C_{N,r_N},r_N/2)^c}e^{\beta(\pu-m_N)} > \iota'\Big) = 0.
\end{align}
To see this, note that   for any $\lambda$,
\begin{align}  \label{107}
\sum_{u\in B(C_{N,r_N},r_N/2)^c}e^{\beta(\pu-m_N)} &\leq   \sum_{u\in B(C_{N,r_N},r_N/2)^c\cap \Gamma_N(\lambda)}e^{\beta(\pu-m_N)}+\sum_{u \notin \Gamma_N(\lambda) }e^{\beta(\pu-m_N)}.
\end{align}
In order to control the first term above, similarly as in  \eqref{435}, define the event
\begin{align} \label{e2} 
 \tilde{E}_3 &:= \{|u-v| \notin [r_N/2, 4r_N],  \  \forall u,v\in \Gamma_N(\lambda)  \} .  
\end{align}
Then,  under the event $ \tilde{E}_3 $, 
\begin{align}
 \Gamma_N(\lambda)  &\subseteq B(C_{N,r_N} ,r_N/2). \label{122}
\end{align}
In fact, for any $v\in \Gamma_N(\lambda)$,  $w : = \argmax_{u\in B(v,2r_N)} \pu $ satisfies $\phi_{N,w} \geq m_N-\lambda$. Since $|w-v| \leq 2r_N$ and $v,w\in \Gamma_N(\lambda)$, under the event $ \tilde{E}_3 $,  $|w-v| \leq r_N/2$. This shows that $w$ is $r_N$-local extrema and $v\in B(w,r_N/2)$, which verifies \eqref{122}.

{Therefore, no terms are included in the first summation in \eqref{107} under the event $ \tilde{E}_3 $.  Also,  under the event $E_1^\lambda$  (defined in \eqref{event e1}), the second summation in \eqref{107} converges to 0 as $N,\lambda\ri $. Since $$  \lim_{\lambda\ri} \liminf_{N\ri} \mathbb{P}( \tilde{E}_3 \cap  E_1^\lambda ) = 1$$ 
by \eqref{meso} and Lemma  \ref{lemma 7.7}, we obtain \eqref{105}.}

Next, note that since $S \geq \sum_{u\in C_{N,r_N}} S_{u,r_N/2}  \geq e^{\beta (\max_{v\in V_N} \pv - m_N)} $, by Corollary \ref{tight} it follows that for any $\e>0$, there exists $\kappa>0$ such that
\begin{align}   \label{1050}
 \limsup_{N\ri} \P(S< \kappa )\leq  \limsup_{N\ri} \P\Big( \sum_{u\in C_{N,r_N}} S_{u,r_N/2} < \kappa \Big)< \e.
\end{align}Since
\begin{align*}
\Big\{ \frac{1}{S}\sum_{u\in B(C_{N,r_N},r_N/2)^c}e^{\beta(\pu-m_N)} >\iota \Big  \} \subseteq   \Big\{\sum_{u\in B(C_{N,r_N},r_N/2)^c}e^{\beta(\pu-m_N)} >\iota
 \kappa \Big\} \cup \{S< \kappa\},
\end{align*}
by \eqref{108}, \eqref{105} and \eqref{1050}, we obtain \eqref{103}.

Since the two sequences  $(\frac{S_{v,r_N/2}}{\sum_{u\in C_{N,r_N}} S_{u,r_N/2} } )_{v\in C_{N,r_N}}$ and $( \frac{S_{v,r_N/2}}{S} )_{v\in C_{N,r_N}} $  have the same  non-increasing ordering indexed by $v\in C_{N,r_N}$  \eqref{103} implies that
\begin{align} \label{111}
  \ell^1 \Big( \mathsf{Ord} \Big(\frac{S_{v,r_N/2}}{\sum_{u\in C_{N,r_N}} S_{u,r_N/2} } \Big)_{v\in C_{N,r_N}}, \mathsf{Ord} \Big(\frac{S_{v,r_N/2}}{S} \Big)_{v\in C_{N,r_N}} \Big)  \rightarrow 0  \quad \text{ in probability}.
\end{align}
Therefore by \eqref{155},
{\begin{align*}
\mathsf{Ord} \Big( \Big(   \frac{S_{v,r_N/2}}{S}  \Big)_{v\in C_{N,r_N}} \Big) \rightarrow  \mathsf{Ord}\Big( \Big( \frac{q_i}{ \sum_{j\in \mathbb{N}} q_j} \Big)_{i\in \mathbb{N}} \Big) \quad \text{ in law}.
\end{align*}}
Noting that
$
 \mu_N^\beta( B(v,r_N/2)) =  \frac{S_{v,r_N/2}}{S}
$
and 
$
\mathsf{Ord}\Big( \Big( \frac{q_i}{ \sum_{j\in \mathbb{N}} q_j} \Big)_{i\in \mathbb{N}} \Big)  \overset{\text{law}}{\sim}     \textup{PD}(\beta_c/\beta),
$
we establish Theorem \ref{theorem 1.1} by setting the collection of $B_i$s to be $\{ B(v,r_N/2) \}_{v\in C_{N,r_N}} $.\\

\noindent
\textbf{Step 3. Proof of \eqref{100}.} 
By the inequality \eqref{key}, it suffices to verify that 
 for any $\e>0$, there exists $N_0$ and $M_0$ such that for any $N\geq N_0$ and $M\geq M_0$,
\begin{align}
  \mathbb{P} \Big(   \sum_{v\in C_{N,r_N}} S_{v,r_N/2} f(Y_{N,v}) (1- g_M(\pv - m_N) h_M(\bar{S}_{v,r_N/2} ))  > \e \Big)   <\e.
\end{align}
By the boundedness of $f$ and  the conditions \eqref{g} and \eqref{h}, it suffices to show that for any $\e>0$, for large enough $N$ and $M$,
\begin{align} \label{119}
  \mathbb{P} \Big( \sum_{v\in C_{N,r_N}}   S_{v,r_N/2} (1- \1_{|\pv - m_N| \leq M} \1_{\bar{S}_{v,r_N/2} \leq  M}  )> \e \Big)  \leq \e.
\end{align}

First, note that for $\lambda>0$, under the event $E_1^\lambda$,
\begin{align} \label{113}
 \sum_{v\in C_{N,r_N}}   S_{v,r_N/2} \1_{v\notin \Gamma_N(\lambda)} \leq  \sum_{u\notin \Gamma_N(\lambda)} e^{\beta(\pu-m_N)} \overset{\eqref{event e1}}{\leq}  C_\beta (e^{-\tau_\beta \lfloor \lambda \rfloor} + N^{-\tau_\beta}),
\end{align}
where   in the first inequality, we used the fact that any $u\in B(v,r_N/2)$ with $ v\in C_{N,r_N} \cap  \Gamma_N(\lambda)^c$ satisfies $u\notin \Gamma_N(\lambda)$.

{Next, in order to uniformly bound  $\bar{S}_{v,r_N/2}$ from above, we recall the event $\bar{E}_3$ defined in \eqref{499}: 
\begin{align*}
\bar{E}_3:= \{  |u-v| \notin [r/2,2r_N], \  \forall u,v\in \Gamma_N(\lambda^2)\}.
\end{align*} }
 By  \eqref{480}, for any $\e>0$, there exist   $\lambda,r,N_0>0$  satisfying
 \begin{align} \label{112}
 C_\beta e^{-\tau_\beta \lfloor \lambda \rfloor}  < \e/2
\end{align}  such that   for any $N\geq N_0$, 
\begin{align}   \label{116} 
 \mathbb{P}( \{\max_{v\in V_N} \pv < m_N+\lambda\} \cap \bar{E}_3 \cap E_1^{\lambda^2} ) &\geq  1-\e .
\end{align} 
 Also,
under the event  $\bar{E}_3 \cap E_1^{\lambda^2}$,
 for any $v\in \Gamma_N(\lambda)$,
 \begin{align} \label{114}
 \bar{S}_{v,r_N/2} \overset{\eqref{489}}{ \leq}  (2r)^d + C_\beta  e^{\beta \lambda -\tau_\beta \lfloor  \lambda^2 \rfloor } +1 \leq (2r)^d + C_0
 \end{align}
for some constant $C_0>0$ depending only on $\beta$ (since the quantity $\beta \lambda -\tau_\beta \lfloor  \lambda^2 \rfloor$ is bounded from above, uniformly in $\lambda$).

Now, note that for any $M\geq \max\{\lambda, (2r)^d + C_0\}$,
 \begin{align*}
1- \1_{|\pv - m_N| \leq M} \1_{\bar{S}_{v,r_N/2} \leq  M}   \leq  \1_{\pv > m_N +\lambda}  + \1_{\pv < m_N -\lambda}  + \1_{|\pv - m_N| \leq \lambda} \1_{\bar{S}_{v,r_N/2}> (2r)^d + C_0}  .
\end{align*}
Hence,   by  \eqref{113} and \eqref{114},  {under the event  $ E_1^{\lambda}\cap  \big( \{\max_{v\in V_N} \pv < m_N+\lambda\} \cap \bar{E}_3 \cap E_1^{\lambda^2} \big)  $,} for any   $M\geq \max\{\lambda, (2r)^d + C_0\}$ and sufficiently large $N$,
\begin{align} \label{117}
 \sum_{v\in C_{N,r_N}}   S_{v,r_N/2} (1- \1_{|\pv - m_N| \leq  M} \1_{\bar{S}_{v,r_N/2} \leq  M}  ) \leq   C_\beta (e^{-\tau_\beta \lfloor \lambda \rfloor} + N^{-\tau_\beta})  \overset{\eqref{112}}{\leq} \e.
\end{align}
Combining this with Lemma \ref{lemma 7.7} and  \eqref{116}, we deduce \eqref{119}.\\

\noindent
\textbf{Step 4. Proof of \eqref{155}.} 
For $v\in C_{N,r_N}$, let 
\begin{align*}
\nu_N(v):= \frac{S_{v,r_N/2}}{\sum_{u\in C_{N,r_N}} S_{u,r_N/2} } .
\end{align*}
{Since every subspace of a separable metric space is separable, support of  the limiting point measure in \eqref{101}, a subspace of the space of probability measures on $[0,1]$ equipped with weak topology,  is also separable.}
Thus, by Skorokhod's representation theorem (see Lemma \ref{represent} for a precise statement), there is a coupling on a common probability space $(\Omega, \mathcal{F},\mathbb{P}) $ under which  the convergence \eqref{101}  holds $\mathbb{P}$-almost surely. By ordering the weights in a non-increasing order, we assume that almost surely,
\begin{align} \label{130}
\sum_{i\in \mathbb{N}} A_{N,i} \delta_{Y_{N,i}} \rightarrow \sum_{i\in \mathbb{N}} B_{i} \delta_{X_{i}}  
\end{align}
(w.r.t.   weak topology on probability measures on $[0,1]$),
where  $$(A_{N,i})_{i\in \mathbb{N}} \overset{\text{law}}{\sim}    \mathsf{Ord}((  \nu_N(v) )_{v\in C_{N,r_N}}),\quad (B_{i})_{i\in \mathbb{N}} \overset{\text{law}}{\sim}  \mathsf{Ord}\Big( \Big(\frac{q_i}{ \sum_{j\in \mathbb{N}} q_j}\Big)_{i\in \mathbb{N}}\Big) ,$$ and $\{X_i\}_{i\in \mathbb{N}}$ and $ \{Y_{N,i}\}_{i\in \mathbb{N}}$ (for each $N$) denote i.i.d.  $\text{Unif}[0,1]$.

We will then be done if we can show that $(A_{N,i})_{i\in \mathbb{N}}$ converges to $(B_{i})_{i\in \mathbb{N}}$ in $\ell_1.$

~

Towards this, we first verify that the quantity $\nu_N(v)$ for $v\in C_{N,r_N}\cap \Gamma_N(\lambda)^c$ is negligible.
Under the event $E_1^\lambda$ whose probability tends to 1 as $N\ri$ followed by $\lambda\ri$ (Lemma \ref{lemma 7.7}),
\begin{align*}
\sum _{v\in C_{N,r_N}\cap \Gamma_N(\lambda)^c}  S_{v,r_N/2}\leq \sum_{u\in \Gamma_N(\lambda)^c} e^{\beta(\pu- m_N)} \overset{\eqref{event e1}}{ \leq}  C_\beta (e^{-\tau_\beta \lfloor \lambda \rfloor} + N^{-\tau_\beta }  )  . 
\end{align*}
Thus,  combining with  \eqref{1050}, for any $\iota,\e>0$,  for large enough $\lambda$ and $N$,  
\begin{align} \label{131}
\mathbb{P}\Big( \sum _{v\in C_{N,r_N}\cap \Gamma_N(\lambda)^c}  \nu_N(v)  < \iota\Big) \geq 1-\e.
\end{align}
{Note that for any positive integer $k$, the event
 \begin{align*}
  \{ | C_{N,r_N} \cap \Gamma_N(\lambda) | \leq k\}\cap  \Big \{\sum _{v\in C_{N,r_N} \cap \Gamma_N(\lambda)^c}   \nu_N(v) < \iota\Big\} 
\end{align*}
implies that the $\ell^1$-norm of the non-increasingly ordered sequence $ \mathsf{Ord}((  \nu_N(v) )_{v\in C_{N,r_N}})$, except the first $k$ coordinates, is less than $\iota$.}
Thus, by Theorem \ref{theorem 3.0} and \eqref{131},   there exists a function $L:(0,\infty)^2 \rightarrow \mathbb{N}$ such that for sufficiently large $N$,
\begin{align} \label{137}
\mathbb{P}\Big(\sum_{i=L(\iota,\e)+1}^\infty A_{N,i} < \iota\Big)  \geq  1-2\e.
\end{align}
Let $ K(\iota,\e)\in \mathbb{N}$ be a positive integer satisfying  
\begin{align} \label{144}
\mathbb{P}\Big(\sum_{i =K(\iota,\e)+1}^\infty B_i < \iota\Big) \geq 1-\e
\end{align}
(the existence  of such $K(\iota,\e)$  is proved in Lemma \ref{ppp} later).
From now on,  we suppress the notations $\iota$ and $\e$ in the expression $L(\iota,\e)$ and  $K(\iota,\e)$.

 Recalling that $ \{Y_{N,i}\}_{i\in \mathbb{N}}$ are i.i.d. distributed as Unif[0,1],  by a union bound,
for any  $k\in \mathbb{N}$,  
\begin{align*} 
\mathbb{P}  (  | Y_{N,i}  - Y_{N,j}|\leq  \e k^{-2} , \  \exists i\neq j\in \{1,\cdots,k\} )  \leq {k \choose 2} \cdot 2\e k^{-2} \leq \e.
\end{align*}
Thus,
\begin{align} \label{133}
\mathbb{P}(  | Y_{N,i}  - Y_{N,j}| \geq \e   \max\{ L , K \}^{-2} , \  \forall i\neq j  \in \{1,\cdots,  \max\{L  , K\} \} ) \geq 1-\e,
\end{align} 
and by the same reason
\begin{align}\label{134}
\mathbb{P}(  | X_{i}  - X_{j}| \geq \e  \max\{ L , K \}^{-2} , \  \forall i\neq j  \in \{1,\cdots,    \max\{ L , K \} \} ) \geq 1-\e.
\end{align}
Therefore, {by the above discussions}, setting the event
\begin{align} \label{147}
G_{N,\iota,\e}:&= \Big\{\sum_{i=L+1}^\infty A_{N,i} < \iota\Big\}  \cap \Big\{\sum_{i =K+1}^\infty B_i < \iota\Big\}  \nonumber \\
&\cap  \{ | X_{i}  - X_{j}| \geq \e \max\{ L , K \}^{-2} , \  \forall i\neq j  \in \{1,\cdots,   \max\{ L , K \} \} \} \nonumber\\ 
&\cap   \{ | Y_{N,i}  - Y_{N,j}| \geq \e   \max\{ L , K \}^{-2} , \  \forall i\neq j  \in \{1,\cdots,  \max\{L  , K\} \} \} ,
\end{align}
for sufficiently large $N$,
\begin{align} \label{150}
\mathbb{P}(G_{N,\iota,\e}) \geq 1-5\e.
\end{align}
In addition, by Egorov's theorem, there exist an event $ \Omega_\e \subseteq \Omega$  and  $N(\iota,\e)\in \mathbb{N}$   with
\begin{align} \label{143}
\mathbb{P}(\Omega_\e) \geq  1-\e
\end{align} 
 such that
under the event $\Omega_\e$, for $N\geq N(\iota,\e)$,
\begin{align} \label{146}
d\Big(  \sum_{i\in \mathbb{N}} A_{N,i} \delta_{Y_{N,i}}    ,  \sum_{i\in \mathbb{N}} B_{i} \delta_{X_{i}} \Big) <  \min \Big\{\frac{1}{10} \e \max \{L,K\}^{-2} , \iota \Big\}=: \kappa,
\end{align}
where   $d$ denotes the L\'evy–Prokhorov metric which induces the weak topology on the space of probability measures on $[0,1]$.

Next, we show that under the event $G_{N,\iota,\e} \cap \Omega_\e $,  for  sufficiently large $N$ (depending on $\iota$ and $\e$) and  $k=1,\cdots, L$, 
\begin{align}  \label{135}
\Big\vert  \sum_{i=1}^k  A_{N, i} - \sum_{i=1}^k   B_i  \Big\vert  \leq 2\iota.
\end{align} 
Let  $\xi \in (0, \e \max\{ L , K \}^{-2}/10) $ be a constant. Such $\xi$ is chosen so that
\begin{align}\label{xi}
2(\xi+\kappa) < \e \max\{ L , K \}^{-2}.
\end{align} By \eqref{146} and the definition of   L\'evy–Prokhorov metric, under the event $G_{N,\iota,\e} \cap \Omega_\e $, for  $N\geq N(\iota,\e) $ and  $k=1,\cdots,L$, 
\begin{align*}
 \sum_{i=1}^\infty   A_{N,i} \delta_{Y_{N,i}}  (\cup_{j=1}^k (X_j - \xi-\kappa, X_j + \xi+\kappa)) \geq  \sum_{i=1}^k   B_i -  \kappa.
\end{align*}
By \eqref{xi},   intervals $(X_j - \xi-\kappa, X_j+ \xi+\kappa)$ for $j=1,\cdots,L$ are  disjoint under the event  $G_{N,\iota,\e}$ ($X_j$s are well-separated by  the definition of  $G_{N,\iota,\e}$ in \eqref{147}). Since $\sum_{i=L + 1}^\infty A_{N,i}  < \iota$, {this event along with $\Omega_\e $} implies
 \begin{align} \label{139}
 \sum_{i=1}^{ L} A_{N,i} \delta_{Y_{N,i}}  (\cup_{j=1}^k (X_j - \xi-\kappa, X_j + \xi+\kappa)) \geq  \sum_{i=1}^k   B_i -  \iota - \kappa.
\end{align}
Under the event $G_{N,\iota,\e}  $, for each $j=1,\cdots,L$, the interval $(X_j - \xi-\kappa, X_j + \xi+\kappa)$ contains at most one point  among $Y_{N,1},\cdots,Y_{N,L}$ {(this is due to the $ \e \max\{ L , K \}^{-2}$-separation property of $Y_{N,j}$s in \eqref{147}, and the condition \eqref{xi}).}
Thus, the LHS of \eqref{139} is at most the sum of $k$ (distinct) elements in $\{A_{N, i}\}_{i\in \mathbb{N}}$.  Recalling that the sequence $(A_{N, i})_{i\in \mathbb{N}}$ is non-increasing, for each $k=1,\cdots,L$, 
{\begin{align} \label{140}
 \sum_{i=1}^k  A_{N, i}  \geq  \sum_{i=1}^k   B_i - \iota-\kappa.
\end{align}}
We now derive the opposite inequality using a similar argument as above. Under the event  $G_{N,\iota,\e}$,
 intervals $(Y_{N,j} -  \xi-\kappa, Y_{N,j} + \xi+\kappa)$ for $j=1,\cdots,L$ are  mutually disjoint. Thus, by \eqref{146} again, {this event together with $\Omega_\e $ imply that}, for  $N\geq N(\iota,\e) $ and $k=1,\cdots,L$, 
\begin{align} \label{141}  
 \sum_{i=1}^k  A_{N, i} & \leq   \sum_{i=1}^\infty  B_{i} \delta_{X_{i}}   ( \cup_{j=1}^k  (Y_{N,j} - \xi-\kappa, Y_{N,j}+ \xi+\kappa))  +\kappa \nonumber \\
&\overset{ \eqref{147}}{\leq } \sum_{i=1}^{ K}  B_{i} \delta_{X_{i}}   ( \cup_{j=1}^k  (Y_{N,j} - \xi-\kappa, Y_{N,j}+ \xi+\kappa))  +\kappa+\iota.
\end{align}
Under the event $G_{N,\iota,\e}  $, for each $j=1,\cdots,L$, the interval $(Y_{N,j} - \xi-\kappa, Y_{N,j} + \xi+\kappa)$ contains at most one point  among $X_1,\cdots,X_{K}$. 
Thus, recalling that the sequence $(B_{i})_{i\in \mathbb{N}}$ is non-increasing,  for each $k=1,\cdots,L$, 
{\begin{align} \label{142}
 \sum_{i=1}^k  A_{N, i}  \leq  \sum_{i=1}^k   B_i + \kappa+ \iota.
\end{align}}
Thus,  recalling $\kappa \leq \iota$ (see \eqref{146}), \eqref{140} and \eqref{142} conclude the proof of  \eqref{135}. In particular,  by \eqref{135} with $k=L$, under the event $G_{N,\iota,\e} \cap \Omega_\e $,    recalling $\sum_{i=L + 1}^\infty A_{N,i}  < \iota$,
\begin{align} \label{149}
 \sum_{i=L+1}^\infty   B_i  \leq 3\iota.
\end{align}

 Now, we are ready to verify that for any $\delta>0$, 
\begin{align} \label{148}
\lim_{N\ri} \mathbb{P}\Big(\sum_{i=1}^\infty  |A_{N, i}  - B_i|>\delta\Big) = 0.
\end{align}
For any  $\e>0$, take  any $\iota_1 \in (0,\frac{\delta}{10})$ and then consider the corresponding $L (\iota_1,\e)$ (see its definition in \eqref{137}). Next, take any $\iota_2 \in (0,  \min\{ \iota_1, \frac{\delta}{10 L (\iota_1,\e)}\})$. Since $\iota_2 \leq \iota_1$, we can assume that  $L (\iota_1,\e) \leq L (\iota_2,\e)$.
By \eqref{135},  for sufficiently large $N$, under the event $G_{N,\iota_2,\e} \cap \Omega_\e $,  for  $k=1,\cdots, L(\iota_2,\e)$, 
 \begin{align*}
 \Big\vert  \sum_{i=1}^k  A_{N, i} - \sum_{i=1}^k   B_i  \Big\vert  \leq 2\iota_2.
 \end{align*}
 This in particular implies that  for  $k=1,\cdots, L(\iota_2,\e)$,  
  \begin{align*}
  | A_{N, k} - B_k| \leq  \Big\vert  \sum_{i=1}^k  A_{N, i} - \sum_{i=1}^k   B_i  \Big\vert  + \Big\vert  \sum_{i=1}^{k-1}  A_{N, i} - \sum_{i=1}^{k-1}   B_i  \Big\vert  \leq 4\iota_2.
 \end{align*}
Recalling $L (\iota_1,\e) \leq L (\iota_2,\e)$,  under the event $G_{N,\iota_2,\e} \cap \Omega_\e $,  
{ \begin{align} \label{151}
  \sum_{i=1}^{ L (\iota_1,\e) }  |A_{N, i}  - B_i|  \leq  4 L (\iota_1,\e) \iota  _2 .
 \end{align}}
 In addition, under the event  $G_{N,\iota_1,\e} \cap \Omega_\e$,   
 \begin{align}  \label{152}
  \sum_{i=L (\iota_1,\e)+1}^\infty  A_{N, i} \leq \iota_1,
 \end{align}
 and by \eqref{149},
 \begin{align} \label{153}
  \sum_{i=L (\iota_1,\e)+1}^\infty  B_{i} \leq 3\iota_1.
 \end{align}
Hence,  by \eqref{151}-\eqref{153}, under the event $ G_{N,\iota_1,\e} \cap G_{N,\iota_2,\e}\cap \Omega_\e$,  for sufficiently large $N$,
\begin{align*}
\sum_{i=1}^\infty  |A_{N, i}  - B_i| \leq  \sum_{i=1}^{ L (\iota_1,\e) }  |A_{N, i}  - B_i|   +    \sum_{i=L (\iota_1,\e)+1}^\infty  A_{N, i} +   \sum_{i=L (\iota_1,\e)+1}^\infty  B_{i} \leq   4 L (\iota_1,\e) \iota  _2 + 4\iota_1.
\end{align*}
By the condition on $\iota_1$ and $\iota_2$, the above quantity is at most $\delta$.  Also, by \eqref{150} and \eqref{143}, for sufficiently large $N$,
\begin{align*}
 \mathbb{P}( G_{N,\iota_1,\e} \cap G_{N,\iota_2,\e}\cap \Omega_\e) \geq 1-11\e.
\end{align*}
Hence, for sufficiently large $N$,
\begin{align*}
  \mathbb{P}\Big(\sum_{i=1}^\infty  |A_{N, i}  - B_i|>\delta\Big)  \leq 11 \e.
\end{align*}
This concludes the proof of \eqref{148}, which implies \eqref{155}.

\end{proof}

\section{Random walk estimates} \label{section 6}
In this section, we include some estimates about  the random walk bridge, branching random walk and modified branching random walk, which have featured in many of our arguments. While the section is somewhat long, all of the results are quoted either from the literature or proven by simple adaptations of known arguments.

\subsection{Random walk bridge}
Let $X_1,\cdots,X_n$ be i.i.d. standard Gaussians, and define the random walk  $S_0:=0$ and  $S_k:=X_1+\cdots+X_k$ for $k=1,\cdots,n$.  The random walk bridge is defined to be the random walk $\{S_k\}_{k=0,1,\cdots,n}$, conditioned on $S_n=0$. Note that the  law of Gaussian random walk bridge is same as that of Brownian bridge at  integer times.
For $n\in \mathbb{N}$, let $W^n=(W^n_t)_{0\leq t\leq n}$ be a   Brownian bridge, a standard Brownian motion $(B_t)_{t\geq 0}$ conditioned on $B_n=0$. 
  The following lemma, obtained in \cite[Proposition 2]{b}, lower bounds the probability that the Brownian bridge stays below a curve which is asymptotically a logarithmic function.
  
\begin{lemma}[\cite{b}, Proposition 2] \label{lemma 6.5}
For a constant $\gamma>0$, define $L_n : \{0,1,\cdots,n\} \rightarrow [0,\infty)$ to be $L_n(j) = \gamma \log ( (j \wedge (n-j)) \vee 1)$, i.e., $L_n(0)= L_n(n)=0$ and 
 \begin{align}\label{function l}
 L_n(j) = \begin{cases}
 \gamma\log j, &j=1,\cdots,\lfloor n/2 \rfloor, \\
 \gamma\log  (n-j), &j= \lfloor n/2 \rfloor+1,\cdots,n-1.
 \end{cases}
 \end{align}
Then,   there exists  a constant $c_{6}>0$  (depending on $\gamma$) such that for any $n\in \mathbb{N}$,
\begin{align*}
  \mathbb{P}(W^n_j \leq 1- L_n(j) \ \textup{for} \   j=0,1,\cdots,n  ) \geq   \mathbb{P}(W^n_t \leq 1- L_n(t) \ \textup{for} \  t\in [0,n])  \geq  \frac{c_{6}}{n}.
\end{align*}
\end{lemma} 

Next, we state a lemma which upper bounds the probability that the Brownian bridge stays below a straight line at integer times.
 \begin{lemma} \label{lemma 6.3}
For any constant $\kappa>0$, there exists $C = C(\kappa)>0$ such that the following holds.
For any $a,b,y > 0$  and $n \in \mathbb{N}$ such that $a+bn > y > \kappa n$,
 \begin{align*}
 \mathbb{P}&(B_{n} \in [y,y+1], B_j \leq a+bj \ \textup{for} \   j =0, 1,\cdots,n)  \leq C   ( \max \{a, a+bn-y \} +1)^2 n^{-3/2} e^{-y^2 / 2n}.
 \end{align*}
 \end{lemma}
In \cite[Lemma 11]{b}, the upper bound for the probability of a continuous version of the event in  Lemma \ref{lemma 6.3} (i.e.  Brownian motion stays below a straight line for all real $t\in  [t_0,t_1]$) is obtained. Though not difficult to obtain, it is also not  trivial to extend this upper  bound effectively to the discrete-time version event considered above.  We need to resort to the following related estimate which upper bounds the probability that the random walk bridge is below a straight line.
 
 \begin{lemma} \label{lemma 6.4}
There exists $C>0$ such that for any $n\in \mathbb{N}$ and $x_1,x_2>0$,
\begin{align*}
\mathbb{P}\Big(W^n_j \leq \frac{j}{n}x_1+  \frac{n-j}{n}x_2  \ \textup{for} \    j=0,1,\cdots,n\Big) \leq \frac{C (\max\{x_1,x_2\}+1)^2}{n}.
\end{align*}
 
 \end{lemma}
 
This then allows one to prove Lemma \ref{lemma 6.3} by affine transforming $B_n$ to get $W^n.$

 \begin{proof}[Proof of  Lemma \ref{lemma 6.3}]
The probability we aim to bound is at most
 \begin{align} \label{801}
  \mathbb{P}\Big(&B_{n} \in [y,y+1], B_j   - \frac{j}{n}B_n \leq a+bj- \frac{j}{n}y  \ \textup{for} \   j=0, 1,\cdots,n \Big) .
 \end{align}
 Note that $(B_t  - \frac{t}{n}B_n)_{t\in [0,n]}$ has the same law as the  Brownian bridge $(W_t^{n})_{t\in [0,n]}$ independent of $B_n$. Thus, the above probability can be written as
 \begin{align*}
 \mathbb{P}&(B_{n} \in [y,y+1]) \mathbb{P} \Big (W^{n}_j  \leq  a+bj- \frac{j}{n}y  \ \textup{for} \  j=0,1,\cdots,n \Big).
 \end{align*}
By the Gaussian tail estimate $\mathbb{P}(X>x)\leq \frac{C}{x}e^{-x^2/2}$ for $x>0$ ($X$ is a standard normal), recalling  the condition $y>\kappa n$, the first term above is bounded by $Cn^{-1/2} e^{-y^2/2n}$ for some $C = C(\kappa)>0$. 
Together with Lemma \ref{lemma 6.4} (with $x_1 = a+bn-y >0$ and $x_2=a>0$),  we conclude the proof.
 
 \end{proof}

We next prove Lemma \ref{lemma 6.4} by bounding the relevant probability in terms of its continuous version (i.e. Brownian bridge stays below a curve for all real $t\in [0,n]$), up to some multiplicative factor. Such an estimate was previously obtained in \cite[Lemma 4.15]{bl2}, at the price of pushing the curve that the Brownian bridge is required to stay below, upwards by a factor of two.

\begin{lemma}[\cite{bl2}, Lemma 4.15] \label{lemma 6.7}
For any non-decreasing concave function  $\zeta:[0,\infty)\rightarrow [0,\infty)$,  
\begin{align} \label{807}
\mathbb{P}&(W^n_j \leq \zeta(j \wedge (n-j)) \ \textup{for} \   j=0,1,\cdots,n) \nonumber \\
&\leq \mathbb{P}(W^n_t \leq 2\zeta(t \wedge (n-t)) \ \textup{for} \  t\in [0,n] ) \prod_{j=0}^{n}(1-e^{-2\zeta(j)^2})^{-2}.
\end{align}
 
\end{lemma}
The idea is simple. A Brownian bridge is made of a random walk bridge and unit order Brownian bridges between the integer points. Thus if  the random walk bridge stays below a curve and the unit order Brownian bridges do not oscillate too much, then the Brownian bridge stays below twice the curve.
Lemma \ref{lemma 6.7} was obtained by conditioning on  the Brownian bridge at integer times and  then analyzing  independent Brownian bridges on the intervals between two consecutive integers.

By Lemma \ref{lemma 6.7}, in order to deduce  Lemma  \ref{lemma 6.4}, it suffices to upper bound  the probability that the Brownian bridge stays below a barrier  for all real $t\in [0,n]$.  In the case when the barrier is asymptotically logarithmic, the upper bound is  obtained in the classical work of Bramson \cite[Proposition 1']{b}.

\begin{lemma}[\cite{b}, Proposition 1'] \label{lemma 6.8}
There exists    $C>0$ such that for any $x>0$ and $ n\in \mathbb{N}$,
\begin{align*}
\mathbb{P}(W^n_t <  3\cdot 2^{-3/2}  \log ( (t\wedge (n-t)   ) \vee 1) + x \ \textup{for} \   t\in [0,n] )\leq C  \frac{x^2}{n}.
\end{align*}

\end{lemma}
We finish the proof of Lemma \ref{lemma 6.4} with the aid of Lemmas \ref{lemma 6.7} and \ref{lemma 6.8}.
 \begin{proof}[Proof of  Lemma \ref{lemma 6.4}]
Setting $x:= \max\{x_1,x_2\}>0$,
  \begin{align} \label{501}
 \mathbb{P}\Big(W^n_j \leq \frac{j}{n}x_1+  \frac{n-j}{n}x_2  \ \textup{for} \    j=0,1,\cdots,n\Big) & \leq   \mathbb{P} (W^n_j \leq x\ \textup{for} \    j=0,1,\cdots,n  )  \nonumber \\
 &\leq   \mathbb{P} (W^n_j \leq   L_n(j) + x +1 \  \textup{for} \    j=0,1,\cdots,n  ),
 \end{align}
 where  $L_n$ is a function  defined in \eqref{function l} with $\gamma  =   3\cdot 2^{-5/2}$. The reason of adding an additional curve $L_n$ for a barrier of the Brownian bridge is to effectively control the multiplicative factor in \eqref{807}.
By Lemma \ref{lemma 6.7} with $$\zeta(t) :=  3\cdot 2^{-5/2} \log (  t \vee 1)+x+1,$$  the probability \eqref{501} is bounded by 
\begin{align} \label{502}
  \mathbb{P}& (W^n_t \leq  3\cdot 2^{-3/2} \log ( ( t \wedge (n-t)) \vee 1)+2x+2  \  \textup{for} \   t\in  [0,n])  \prod_{j=0}^{n}(1-e^{-2 (3\cdot 2^{-5/2} \log ( (j \wedge (n-j))\vee 1 )+x+1)^2})^{-2}  \nonumber \\
  &\leq C  \mathbb{P} (W^n_t \leq  3\cdot 2^{-3/2} \log ( ( t \wedge (n-t)) \vee 1)+2x+2  \  \textup{for} \   t\in  [0,n]) .
\end{align} 
Here, the inequality follows from
\begin{align*}
 \prod_{j=0}^{n}(1-e^{-2 (3\cdot 2^{-5/2} \log ( (j \wedge (n-j))\vee 1 )+x+1)^2})^{-2} &\leq C \prod_{j=2}^{ \lfloor n/2 \rfloor }(1-e^{-2 (3\cdot 2^{-5/2} \log  j )^2})^{-2}  \\
 &\leq  C \exp\Big(c'\sum_{j=2}^{ \lfloor n/2 \rfloor }    e^{-2 (3\cdot 2^{-5/2} \log  j )^2} \Big) \leq C,
\end{align*}  
where we used the fact that there is a constant $c'>0$ such that  $(1-r)^{-2} \leq 1+c'r \leq e^{c'r}$ for any $r \in (0,0.99)$.
 By Lemma \ref{lemma 6.8}, the probability in \eqref{502} is bounded by 
$C\frac{(x+1)^2}{n}$.
 This concludes the proof.

 \end{proof}

We conclude this subsection by stating, for the record, the invariance of Brownian bridge  under a  linear shift.
Denote by
 $p(t_1,x_1 ; \cdots ; t_k,x_k )$ the joint density  of the Brownian motion at points $x_1,\cdots,x_k$ and times $t_1,\cdots,t_k$, and denote by $p(t_1,x_1 ; \cdots ; t_k,x_k \mid s,y)$ the conditional joint density, conditioned on $B_s=y$. Then, 
\begin{align}\label{bridge}
p(t_1,x_1 ; \cdots ; t_k,x_k  \mid n,0) = p\Big(t_1,x_1+\frac{t_1}{n}y ; \cdots ; t_k,x_k+\frac{t_k}{n}y   \mid  n,y\Big). 
\end{align}

\subsection{Branching random walk}

Recall that $ \varphi_N = (\varphi_{N,v})_{v\in V_N}$ denotes the  $d$-dimensional BRW after $n$ generations, where $N=2^n$.
In this subsection, we  establish a sharp upper bound on the expectation of the sum of  $\ell$ largest values of the BRW, {which is crucially used in the proof of  Proposition \ref{prop 3.6}}.
{Similarly as in \eqref{sum},}
for $N,\ell \in \mathbb{N}$,
 define  $  S^{\text{BRW}}_{N,\ell} $ by
 \begin{align*} 
 S^{\text{BRW}}_{N,\ell}: = \max\Big\{\sum_{v\in I}\varphi_{N,v} : I \subseteq V_N, |I|=\ell\Big\}
 \end{align*}  
 (we set $ S^{\text{BRW}}_{N,\ell}:  = -\infty$ when $\ell>N$).

\begin{lemma} \label{lemma 6.2}
For any constant $\lambda < \frac{1}{\sqrt{2d}}$, for $N\geq 1$ and sufficiently large  $\ell$,
\begin{align} \label{611}
\mathbb{E}  S^{\textup{BRW}}_{N,\ell} \leq \ell(m_N-\lambda \log \ell).
\end{align}
\end{lemma}
{To deduce this lemma, we make use of a tail estimate for the maximum and the cardinality bound for level sets, recorded in the following lemmas.  We first state an almost matching lower and upper bound on the right tail of the maximum of BRW.}

\begin{lemma}[Lemma 3.2 in \cite{dz}] \label{lemma 6.0}
 There exist  constants $C_1,C_2>0$ such that for   any $t\geq 0$,
\begin{align} \label{600}
\mathbb{P}(\max_{v\in V_N}   \varphi_{N,v} \geq m_N + t)\leq C_1(1+t)e^{-\sqrt{2d}t},
\end{align}
and for any $0\leq t\leq \sqrt{n} $,
\begin{align} \label{601}
\mathbb{P}(\max_{v\in V_N}   \varphi_{N,v} \geq m_N + t)\geq  C_2 e^{-\sqrt{2d}t}.
\end{align}
\end{lemma}
 (We  do not have a term $\log 2$ in the exponent in  \eqref{600} and \eqref{601} as opposed to  \cite[Lemma 3.2]{dz}, since the variance of increments of BRW is $\log 2$ in our case.)
Further, in  \cite[Lemma 3.2]{dz}, the argument is only provided   for the two-dimensional BRW with standard normal increments, {but the argument extends to any dimensions.}

Next,  we state a lemma on the cardinality of  level sets of BRW.  {Analogous to the previous notation  \eqref{level set notation},
for $z\in \R$, define  
\begin{align*}
\Gamma_N^{\text{BRW}}(z):= \{v\in V_N: \varphi_{N,v}  \geq m_N-z \}.
\end{align*}}
The two dimensional version appears as   \cite[Proposition 3.3]{dz}, and the argument readily extends to $d$-dimensions. We briefly sketch it  for completeness.

\begin{lemma} \label{lemma 6.1}

There exist  $C,c>0$ such that for any $N\geq 1$, $z\geq 0$ and $u\in \R$ such that  $z+u>c$ and $u> \frac{3}{\sqrt{2d}}\log (z+u)+2>0$,
\begin{align} \label{624}
\mathbb{P}(|\Gamma_N^{\textup{BRW}} (z)| \geq e^{\sqrt{2d}(z+u)}) \leq Ce^{-\sqrt{2d} u+C\log (z+u)}.
\end{align}
\end{lemma}

\begin{proof}
Define $r:=   \left \lfloor{ 2(z+u)^2 }\right \rfloor  
$.
 Then, by \eqref{462} which implies $m_{2^rN} \leq m_N + \sqrt{2d}\log 2^r$,
\begin{align*}
m_{2^rN} + u - \frac{3}{2\sqrt{2d}} \log r & \leq (m_{N} - z)  +  \sqrt{2d}\log 2^r    - \frac{3}{2\sqrt{2d}} \log r + z +u \\
&=(m_{N} - z)  + ( m_{2^r} + z+u + a),
\end{align*}
{where $a$ is an explicit constant given by $a:=  - \frac{3}{2\sqrt{2d}} \log \log 2$.}
Thus, by the branching structure of BRW, for any positive integer  $L$   which will be chosen later, 
\begin{align} \label{621}
\mathbb{P}&\Big(\max_{v\in V_{2^ rN}} \varphi_{2^r N,v} \geq  m_{2^r N}+u-\frac{3}{2\sqrt{2d}} \log r\Big) \nonumber \\
&\geq \mathbb{P}(|\Gamma_N^{\text{BRW}} (z)|\geq L)(1-\mathbb{P}(\max_{v\in V_{2^r}} \varphi_{2^r,v} \leq m_{2^r}  +  z+u+a)^L) .
\end{align}
Since $ 0<  z+u+a < \sqrt{r}$  for sufficiently large $z+u$,
by {\eqref{601}  in Lemma \ref{lemma 6.0}},  
\begin{align}  \label{651}
\mathbb{P}(\max_{v\in V_{2^r}} \varphi_{2^r,v} \leq m_{2^r}  +  z+ u +a) \leq 1-C_2 e^{-\sqrt{2d}(   z+u+a)}.
\end{align}
Also, due to the condition   $u> \frac{3}{\sqrt{2d}}\log (z+u)+2\geq \frac{3}{2\sqrt{2d}} \log r$, by \eqref{600} in  Lemma \ref{lemma 6.0},
\begin{align}   \label{652}
\mathbb{P}\Big(\max_{v\in V_{2^r V_N}} \varphi_{2^r N,v} \geq  m_{2^r N}+u-\frac{3}{2\sqrt{2d}} \log r\Big) \leq  C_1 \Big(1+u -  \frac{3}{2\sqrt{2d}} \log r \Big)e^{-\sqrt{2d}(u-\frac{3}{2\sqrt{2d}} \log r)}.
\end{align}
Let us take $L=  \lceil{  e^{\sqrt{2d}(z+u)}     }  \rceil
 $ in \eqref{621}. Using the fact $(1-\frac{1}{x})^x \leq e^{-1}$ for $x>0$, \eqref{651} implies
\begin{align}  
1-\mathbb{P}(\max_{v\in V_{2^r}} \varphi_{2^r,v} \leq m_{2^r}  +  z+u+a)^L \geq 1-e^{-C_2e^{-\sqrt{2d}a} } =:a_1>0.
\end{align}
Applying this together with \eqref{652} to \eqref{621},
 \begin{align*}
 \mathbb{P}(|\Gamma_N^{\textup{BRW}} (z)| \geq e^{\sqrt{2d}(z+u)}) &\leq a_1^{-1} \cdot C_1 \Big(1+u -  \frac{3}{2\sqrt{2d}} \log r \Big) e^{-\sqrt{2d}(u-\frac{3}{2\sqrt{2d}} \log r)} \leq  C e^{-\sqrt{2d} u+C\log (z+u)},
 \end{align*}
 where the last inequality is obtained using the fact $1+u  - \frac{3}{2\sqrt{2d}} \log r \leq   r^{1/2} = e^{(\log r)/2}$ (recall  $r=   \left \lfloor{ 2(z+u)^2 }\right \rfloor  
$ and $u,z\geq 0$).
\end{proof}

Now, we prove Lemma \ref{lemma 6.2} with the aid of Lemmas \ref{lemma 6.0} and \ref{lemma 6.1}. A similar approach was indeed used  in \cite[Proof of Theorem 1.2]{dz} to deduce that for the two-dimensional BRW, there exists a constant $c>0$ such that for sufficiently large $\ell$,
\begin{align} \label{brw}
\mathbb{E}  S^{\text{BRW}}_{N,\ell} \leq \ell(m_N-c\log \ell).
\end{align}
For our purposes, we need the sharp bound in \eqref{611} which we provide the details for next.

\begin{proof}[Proof of Lemma \ref{lemma 6.2}]
Throughout the proof, we suppress the notation $\text{BRW}$ in  $ S^{\text{BRW}}_{N,\ell}$.
We claim that  there exists a constant $C_0>0$ such that the following holds: For any $\e_0>0$, for large enough $\ell$ (depending on $\e_0$), {$N\geq 1$} and  any $\e \in [\e_0,1]$,  
\begin{align} \label{626}
\mathbb{P}\Big(S_{N,\ell} \geq  \ell\Big(m_N - \frac{1-\e}{ (1+\e) \sqrt{2d}}\log \ell\Big)\Big)\leq  C_0 \frac{1}{\e \ell^{{\e/4}}}.  
\end{align}
Let $v_1,\cdots,v_\ell\in V_N$ be (random) points such that  $S_{N,\ell} = \varphi_{N,v_1}+\cdots+\varphi_{N,v_\ell}$.
Define 
the events  $A$ and $B$ by
\begin{align*}
A:= \Big\{\Big\vert \Big\{i=1,\cdots,\ell: \varphi_{N,v_i} \geq m_N - \frac{1}{(1+\e)\sqrt{2d}}\log \ell\Big\}\Big\vert < \frac{\e}{1+\e}\ell\Big\}
\end{align*}
and 
\begin{align*}
B:= \Big\{\max_{v\in V_N} \varphi_{N,v}  < m_N + \frac{\e}{(1+\e) \sqrt{2d}}\log \ell \Big\}.
\end{align*}
We claim that
\begin{align} \label{625}
A\cap B \Rightarrow S_{N,\ell} < \ell\Big(m_N - \frac{1-\e}{(1+\e) \sqrt{2d}}\log \ell\Big).
\end{align}
In fact, under the event $A$, among $v_1,\cdots,v_\ell$, there are at least $\frac{1}{1+\e}\ell$ many $v_i$s  ($i=1,\cdots,\ell$) whose values are at most $m_N - \frac{1}{(1+\e)\sqrt{2d}}\log \ell$. Thus, under the event  $A\cap B$, 
\begin{align*}
\ell \Big(m_N + \frac{\e}{(1+\e) \sqrt{2d}}\log \ell\Big) - S_{N,\ell} &>  \frac{1}{1+\e}\ell\cdot \Big (\frac{1}{(1+\e) \sqrt{2d}}\log \ell + \frac{\e}{(1+\e)\sqrt{2d}}\log \ell\Big) \\
&= \frac{1}{(1+\e) \sqrt{2d}}\ell \log \ell,
\end{align*}
which implies \eqref{625}.

By   Lemmas  \ref{lemma 6.0} and \ref{lemma 6.1} (with $z= \frac{1}{(1+\e) \sqrt{2d}}\log \ell$ and $z+u = \frac{1}{\sqrt{2d}}\log  (\frac{\e}{1+\e} \ell)$, such $z$ and $u$ satisfy the conditions in Lemma \ref{lemma 6.1}),  for sufficiently large $\ell$ (depending only on $\e_0$) and any $\e \in [\e_0,1]$,
\begin{align} \label{622}
\mathbb{P}(A^c) \leq C (\log \ell)^{C} \cdot  \frac{e^{\sqrt{2d} \frac{1}{(1+\e)\sqrt{2d}}\log \ell}}{\frac{\e}{1+\e}\ell} \leq  C \frac{1}{\e}\ell^{-{\e/4}}
\end{align}
and
\begin{align} \label{623}
\mathbb{P}(B^c) \leq  C \log \ell \cdot e^{-\sqrt{2d}\frac{\e}{(1+\e)\sqrt{2d}}\log \ell} \leq C \ell^{-{\e/4}}.
\end{align}
Thus, by \eqref{625}, \eqref{622}, and \eqref{623}, the statement \eqref{626} is proved.

Now, we conclude the proof using \eqref{626}.
Let us choose any $\lambda'\in (\lambda,\frac{1}{\sqrt{2d}})$. Then,  using \eqref{useful},
\begin{align} \label{627}
\mathbb{E}S_{N,\ell} \leq \ell(m_N - \lambda' \log \ell) + \int_{\ell(m_N - \lambda' \log \ell)}^{\ell m_N} \mathbb{P}(S_{N,\ell}>t)dt +  \int_{\ell m_N}^{\infty} \mathbb{P}(S_{N,\ell}>t)dt. 
\end{align}
Let us estimate above two integrals separately.
Take $\e_0>0$ such that $\lambda'  = \frac{1-\e_0}{\sqrt{2d}(1+\e_0)} $. Then, by \eqref{626},   for large enough $\ell$,
\begin{align} \label{628}
\int_{\ell(m_N - \lambda' \log \ell)}^{lm_N} \mathbb{P}(S_{N,\ell}>t)dt  &=  \int_{\e_0}^{1} \mathbb{P}\Big(S_{N,\ell}> \ell\Big(m_N- \frac{1-\e}{ (1+\e) \sqrt{2d}}\log \ell\Big)\Big) \frac{2}{(1+\e)^2} \frac{1}{\sqrt{2d}}\ell \log \ell d\e  \nonumber\\
&\leq  C_0 \ell\log \ell \int_{\e_0}^1  \frac{1}{\e (1+\e)^2}\ell^{-\e/4}  d\e .
\end{align}
{This implies that for any constant  $\eta >0$, for sufficiently large  $\ell$,}
  \begin{align} \label{631}
  \int_{\ell(m_N - \lambda' \log \ell)}^{\ell  m_N} \mathbb{P}(S_{N,\ell}>t)dt  <  \eta \ell \log \ell.
  \end{align}
{We assume that $\eta$ is small enough such that $\eta \in (0,\lambda'-\lambda)$.} We now estimate the second integral in \eqref{627}.
Since  $S_{N,\ell}>\ell(m_N+s)$ implies  $\max_{v\in V_N} \varphi_{N,v} \geq m_N+s$,  by Lemma \ref{lemma 6.0},
 \begin{align} \label{629}
  \int_{\ell m_N}^{\infty} \mathbb{P}(S_{N,\ell}>t)dt   \leq  \ell \int_0^\infty \mathbb{P}(\max_{v\in V_N} \varphi_{N,v} \geq m_N+s) ds <  C\ell.
 \end{align}
Thus, applying \eqref{631} and \eqref{629} to \eqref{627}, using the fact $\eta<\lambda'-\lambda$, for sufficiently large $\ell$,
\begin{align*}
\mathbb{E}S_{N,\ell}\leq   \ell(m_N - \lambda' \log \ell) + \eta \ell \log \ell +  C \ell <  \ell(m_N - \lambda \log \ell),
\end{align*}
 which concludes the proof.

\end{proof}

\subsection{Modified branching random walk}

In this subsection, we prove
Lemma \ref{lemma 3.7}, which provides a uniform lower bound on the right tail  of  $S_\ell  (\theta_N) - \ell(m_N-\frac{1}{\sqrt{2d}}\log \ell)$ for MBRW $\theta_N$.  
Proof   is based on the, by now common in this area, modified second moment argument together with random bridge estimates. The argument bears resemblance to the proof of  \cite[Proposition 5.3]{bz}, where the corresponding statement about  the centered maximum  of MBRW  was obtained.
 
 \begin{proof} [Proof of Lemma \ref{lemma 3.7}]
Recall $N=2^n$, and  throughout the proof, we use the function $L_n$ defined in \eqref{function l} (with large enough constant $\gamma>0$ chosen later). To prove the lemma, we will obtain a uniform lower bound on the cardinality of level sets of MBRW by applying the second moment method after suitable truncation.
For $x>0$ and  $t\in \R$, define the interval $I_x(t)$ by
\begin{align*}
I_x (t):=  \Big[m_N-\frac{1}{\sqrt{2d}}\log x-t,m_N-\frac{1}{\sqrt{2d}}\log x-t+1\Big].
\end{align*} 
Next, we introduce  MBRW truncated at each level: for $0\leq i\leq  j\leq n-1$,
\begin{align*}
  \theta_{N,v}(i,j) := \sum_{k=i}^j \sum_{B\in \mathcal{B}_k(v)}  g_{k,B}^N \overset{\text{law}}{\sim} \mathcal{N}(0, j-i+1).
\end{align*}
Note that $\theta_{N,v}(0,n-1)  = \theta_{N,v}$.
 For $v\in V_N$, define the event
 \begin{align*}
E_{v,x} := \cap_{j=1}^n \Big\{{\theta_{N,v}(n-j,n-1) \leq \frac{j}{n}\Big(m_N-\frac{1}{\sqrt{2d}}\log x+1\Big) -  L_n(j)+1}\Big \} \cap \{ \theta_{N,v} \in I_x(0) \},
 \end{align*}
 and  then define
 \begin{align*}
 T_{N,x}:= \sum_{v\in V_N} \1_{E_{v,x}}.
 \end{align*}
 
 The expert reader will immediately recognize that $E_{v,x}$ is the well known barrier event going back to the seminal work of Bramson \cite{b}, used to control the second moment which otherwise is known to blow up. 

 We claim that there exist constants $C_1,C_2>0$ such that for any $x>0$, for sufficiently large $N$,
 \begin{align}\label{371}
 \mathbb{E}T_{N,x} \geq  C_1x
 \end{align}
 and
 \begin{align} \label{372}
 \mathbb{E}T_{N,x}^2 \leq C_2x^2+C_2x,
 \end{align}
 which particularly implies that there exists $C_3>0$ such that  for any $x\geq 2/C_1$,
 \begin{align} \label{375}
  \mathbb{E}T_{N,x}^2 \leq C_3x^2.
\end{align}  
Indeed, if \eqref{371} and \eqref{375} hold, then by the Paley-Zygmund inequality, there exists a constant $c'>0$ such that  for any $x\geq 2/C_1$,
\begin{align} \label{377} 
\mathbb{P}\Big(T_{N,x} \geq  \frac{1}{2} C_1 x\Big) \geq c'.
\end{align} 
 For  $x\geq 2/C_1$ with $\frac{1}{2}C_1x \in \mathbb{N}$, the event $\{T_{N,x} \geq  \frac{1}{2} C_1 x\}$ implies that there exist at least $\frac{1}{2}C_1 x$  many points $v$ such that  $\theta_{N,v} \geq  m_N -\frac{1}{\sqrt{2d}}\log x$. This  particularly  implies that 
\begin{align} \label{370}
 \mathbb{P}\Big( S_{\frac{1}{2} C_1x }(\theta_N) \geq  \frac{1}{2} C_1 x \Big(m_N-\frac{1}{\sqrt{2d}}\log x \Big)\Big) \geq c'.
\end{align}
For any positive integer $\ell\geq 1$, by setting $x = \frac{2\ell}{C_1} \geq \frac{2}{C_1}$, we establish   Lemma \ref{lemma 3.7}.

From now on, we suppress the notation $x$ in $E_{v,x}$ and $T_{N,x}$. 
We first prove the first order estimate \eqref{371}. Since  $\theta_{N,v}  $ is distribued as $ \mathcal{N}(0,\log N )$,
by a standard Gaussian estimate,  
\begin{align} \label{373}
  \frac{N^d}{n} \mathbb{P}(\theta_{N,v} \in I_x(0))  \asymp x.
\end{align}
In addition, we have
\begin{align} \label{374}
\mathbb{P}(B_j \leq 1-  L_n(j), j=1,\cdots,n \mid B_n=0) \leq \frac{\mathbb{P}(E_{v})}{\mathbb{P}(\theta_{N,v} \in I_x(0))} \leq  \mathbb{P}(B_j\leq 2, j=1,\cdots,n\mid B_n=0).
\end{align}
In fact,  by  monotonicity of the  conditional probability in the conditioned value of $B_n$ and noting that $(\theta_{N,v}(n-j,n-1))_{j=1,\cdots,n} $ is distributed as a Brownian motion at integer times between 1 and $n$,
\begin{align*}
&\frac{\mathbb{P}(E_{v})}{\mathbb{P}(\theta_{N,v} \in I_x(0)) } \\
& \geq \mathbb{P}\Big(B_j \leq  \frac{j}{n}\Big(m_N-\frac{1}{\sqrt{2d}}\log x+1\Big) -  L_n(j)+1, j=1,\cdots,n \mid B_n= m_N - \frac{1}{\sqrt{2d}}\log x + 1\Big) \\
&\overset{\eqref{bridge}}{ =}\mathbb{P}(B_j\leq -  L_n(j)+1, j=1,\cdots,n \mid B_n=0) .
\end{align*} 
The upper bound of  \eqref{374} can be obtained similarly.
 By Lemma \ref{lemma 6.5} and \ref{lemma 6.4}, both LHS and RHS  in \eqref{374} are of order $\frac{1}{n}$. Thus, by \eqref{373} and \eqref{374},  
 \begin{align} \label{381}
  \mathbb{P}(E_{v}) \asymp \frac{x}{N^d}.
\end{align}  This provides the first moment estimate   $\mathbb{E}T_{N} \asymp x$ which  particularly implies \eqref{371}.

 Now, let us prove the second moment estimate \eqref{372}.
 For $u,v\in V_N$,  define
 \begin{align} \label{998}
 r(u,v): = n - \lceil  \log_2(|u-v|^{(N)}+1) \rceil 
 \end{align}
 (see Lemma \ref{lemma 2.3} for the definition of $|u-v|^{(N)}$). A crucial property of such quantity is that
 \begin{align} \label{indep}
 \{\theta_{N,v}(j,n-r(u,v)-1)\}_{0\leq j \leq n-r(u,v)-1}  \ \text{and} \  \{\theta_{N,u}(n-j,n-1)\}_{1\leq j\leq n} \  \text{are} \  \text{independent}.
\end{align}  This is because for $0\leq j\leq n-r(u,v)-1$, the set of boxes of size $2^j$ containing $u$ and $v$ (denoted by $\mathcal{B}_j(u)$ and   $\mathcal{B}_j(v)$ respectively)  are disjoint, which follows from the fact that $2^j\leq 2^{n-r(u,v) - 1} \leq  |u-v|^{(N)}$. Note that $0\leq r(u,v) \leq n$, and $r(u,v) = n$ if and only if $u= v$.

  We write
 \begin{align}\label{382}
 \mathbb{E}T_{N}^2 = 	\sum_{r(u,v)\geq n/2} \mathbb{P}(E_u\cap E_v) + 	\sum_{r(u,v)< n/2} \mathbb{P}(E_u\cap E_v).
 \end{align}
We compute the above two terms separately. Let us first analyze the first term, by decoupling two events $E_u$ and $E_v$ in an appropriate way.  
Let $ F_{u,v}$ be  the event
 \begin{align*}
 F_{u,v} :=\Big\{\theta_{N,v}(n-r(u,v),n-1)\leq \frac{r(u,v)}{n}\Big(m_N-\frac{1}{\sqrt{2d}}\log x+1\Big) -L_n(r(u,v))+1 \Big\} \cap \{ \theta_{N,v}\in I_x(0)  \}.
 \end{align*}
 The motivation behind the event $F_{u,v}$ (which holds under the event $E_v$) is that it implies   the event
 \begin{align*}
\Big\{ \theta_{N,v}(0,n-r(u,v)-1) \geq \Big(1-\frac{r(u,v)}{n}\Big)\Big(m_N-\frac{1}{\sqrt{2d}}\log x+1\Big) + L_n(r(u,v)) - 1 \Big\}
\end{align*}  
which is independent of $E_u$ (see \eqref{indep}). 
Abbreviating $r(u,v)$ to $r$, for sufficiently large $n$,
 \begin{align} \label{384}
 \mathbb{P}(E_u\cap E_v) & \leq \mathbb{P}(E_u \cap F_{u,v}) \nonumber \\
 &\leq  \mathbb{P}(E_u)\mathbb{P}\Big(\theta_{N,v}(0,n-r-1) \geq \Big(1-\frac{r}{n}\Big)\Big(m_N-\frac{1}{\sqrt{2d}}\log x+1\Big) + L_n(r) - 1\Big) \nonumber \\
 &\overset{\eqref{381}}{\leq} CxN^{-d}  \cdot  \exp\Big(-\frac{  ((1-\frac{r}{n})(m_N-\frac{1}{\sqrt{2d}}\log x+1) + \gamma\log (n-r) -1 )^2}{2 (\log 2)(n-r)}\Big) \nonumber \\
 &\leq Cx N^{-d}    \cdot 2^{-d\log_2 |u-v|^{(N)}}   e^{\frac{3}{2} \frac{n-r}{n}  \log n + \frac{n-r}{n}  \log x}  e^{-\gamma \log (n-r)\cdot  ( \sqrt{2d} -  \frac{3}{2(\log 2) \sqrt{2d}} \frac{\log n}{n} - \frac{1}{(\log 2) \sqrt{2d}} \frac{\log x}{n} )} \nonumber \\
 &\leq  Cx N^{-d}  \cdot   2^{-d\log_2 |u-v|^{(N)}}    e^{\frac{3}{2} \frac{\log n}{n}     (n-r) - \gamma \frac{\sqrt{2d}}{2}\log (n-r) }  x^{  \frac{n-r}{n} +  \frac{\gamma}{(\log 2) \sqrt{2d}} \frac{\log (n-r)}{n}}   .
 \end{align}
The derivation of the fourth line from the third line requires some computation and will be presented at the end of proof. Deducing the last line from the fourth line is straightforward, and we analyze the exponents of $e$ and $x$  in the last line.
 Since $x\mapsto \frac{\log x}{x}$ is decreasing  on $(1,\infty)$, there exists a  large enough constant $\gamma>0$ such that  for any $ 0 \leq r\leq n-1$,
 \begin{align} \label{387}
 \frac{3}{2} \frac{\log n}{n}     (n-r) - \gamma \frac{\sqrt{2d}}{2}\log  (n-r) < - 2 \log  (n-r).
\end{align} 
For such $\gamma$,  there exists a constant $r_0>0$ (depending on $\gamma$) such that for any $\frac{n}{2} \leq r\leq n-r_0$, 
\begin{align} \label{388}
 \frac{ n-r}{n} +  \frac{\gamma}{(\log 2) \sqrt{2d}} \frac{\log  (n-r)}{n} < \frac{2 (n-r)}{n} \leq 1.
\end{align} 
Applying \eqref{387} and  \eqref{388} to \eqref{384}, for any $\frac{n}{2} \leq r\leq n-r_0$, 
\begin{align} \label{389}
 \mathbb{P}(E_u\cap E_v) \leq    Cx^2 N^{-d} \cdot    2^{-d\log_2 |u-v|^{(N)}}   (n-r)^{-2} .
\end{align}
In addition, for $n-r_0 < r\leq n$,  
\begin{align} \label{380}
 \mathbb{P}(E_u\cap E_v) \leq   \mathbb{P}(E_u)  &\overset{\eqref{381}}{\leq}  C xN^{-d}.
\end{align}
Thus,  using the fact that for any $u\in V_N$, the number of points $v\in V_N$ with $|u-v|^{(N)}\in [2^k,2^{k+1}]$ is of order $2^{dk}$,  by  \eqref{389} and \eqref{380},
 \begin{align}  \label{383}
 \sum_{r(u,v)\geq  n/2}  \mathbb{P}(E_u\cap E_v) &\leq   \sum_{r=\lceil  n/2 \rceil  }^{n-r_0} Cx^2 (n-r)^{-2}  +  \sum_{r=n-r_0+1}^n   C  xN^{-d}\cdot N^d  \leq Cx^2+Cx.
 \end{align}

 Now, let us bound the second term in \eqref{382}, i.e.  $r:= r(u,v) < n/2$. We consider two cases $1\leq r<n/2$ and $r=0$. Let us first consider the case $1\leq r<n/2$. Setting 
 \begin{align} \label{390}
 \Lambda (j) :=  \frac{j}{n}\Big(m_N-\frac{1}{\sqrt{2d}}\log x+1\Big) - L_n(j)+1 , \quad j=1,\cdots,n,
\end{align}  define the event 
 \begin{align*}
  F^1_{u,v} :=\{\theta_{N,v}(n-r,n-1)\leq   \Lambda (r) \}.
 \end{align*}
{In addition, for $t\in \R$, define the event}
 \begin{align*}
  F^2_{u,v}(t) :=&  \cap_{j=r+1}^n \Big \{\theta_{N,v}(n-j,n-r-1) \leq \frac{j}{n}\Big(m_N-\frac{1}{\sqrt{2d}}\log x+1\Big) +1 -t \Big\} \\
  &\cap \{\theta_{N,v}(0,n-r-1)   \in I_x  (t) \}.
 \end{align*}
The motivation behind these events is that, as already noted in \eqref{indep},  since $\{ \theta_{N,v}(n-j,n-r-1)\}_{ r+1\leq j\leq n}$   is   independent of  $E_u$,  one can decouple the event
 $F^2_{u,v}(\theta_{N,v}( n- r,n-1) )$ 	  (which is implied by the event $E_v$) from the event $E_u$, conditionally on $\theta_{N,v}( n- r,n-1)$. 
Since  $\{ \theta_{N,v}(n-j,n-r-1)\}_{ r+1\leq j\leq n}$ is independent of  the joint law of $\theta_{N,v}( n- r,n-1) $ and   $\{ \theta_{N,u}(j,n-1)\}_{0\leq  j\leq n-1}$, we deduce that  $\{ \theta_{N,v}(n-j,n-r-1)\}_{ r+1\leq j\leq n}$  (hence the event $  F^2_{u,v}(t)$)  is conditionally independent of the event $E_u$, given $\theta_{N,v}( n- r,n-1) $. Hence, noting that  $E_v \subseteq F^1_{u,v}\cap   F^2_{u,v}( \theta_{N,v}( n- r,n-1) ) $, 
\begin{align} \label{379}
  \mathbb{P}(E_u\cap E_v) & \leq \mathbb{P}(E_u \cap F^1_{u,v}\cap   F^2_{u,v}( \theta_{N,v}( n- r,n-1) )) \nonumber \\
 &= \mathbb{E}[\mathbb{P}(E_u \cap   F^2_{u,v}( \theta_{N,v}( n- r,n-1) ) \mid \theta_{N,v}( n- r,n-1) ) \1_{ F^1_{u,v}}] \nonumber \\
  &= \mathbb{E}[\mathbb{P}(E_u  \mid \theta_{N,v}( n- r,n-1) ) \cdot  \mathbb{P}( F^2_{u,v}( \theta_{N,v}( n- r,n-1) )  \mid \theta_{N,v}( n- r,n-1) ) \1_{ F^1_{u,v}}]  \nonumber \\
  &\leq  \mathbb{E}[\mathbb{P}(E_u   \mid  \theta_{N,v}( n- r,n-1) ) \cdot   \max_{t\leq  \Lambda (r)} \mathbb{P}(F^2_{u,v}(t) ) ] \nonumber \\
  &  =   \mathbb{P}(E_u) \max_{t\leq  \Lambda (r)} \mathbb{P}(F^2_{u,v}(t) ).
 \end{align}
 We use Lemma \ref{lemma 6.3} to bound  the probability $\mathbb{P}(F^2_{u,v}( t) )$.
Note that there exists a constant $\kappa>0$ such that $I_x(t) > \kappa (n-r)$ for any $t\leq \Lambda (r)$ and large enough $n$. Thus,  noting that $(\theta_{N,v}(n-j,n-r-1))_{j=r+1,\cdots,n} $ is distributed as a Brownian motion at integer times from 1 to $n-r$,  by  Lemma \ref{lemma 6.3} with 
\begin{align*}
a:=  \frac{r}{n}\Big(m_N-\frac{1}{\sqrt{2d}}\log x + 1\Big ) +1-t, \quad b:= \frac{1}{n} \Big(m_N-\frac{1}{\sqrt{2d}}\log x + 1\Big ),\quad  y:=  m_N-\frac{1}{\sqrt{2d}}\log x-t,
\end{align*} 
we obtain that for $t \leq \Lambda (r)$,
\begin{align} \label{378}
\mathbb{P}(F^2_{u,v}( t) )&\leq C(n-r)^{-3/2}    ( \max \{a, a+b(n-r)-y \} +1)^2  \exp\Big( -\frac{y^2 }{2(\log 2)(n-r)} \Big) \nonumber \\
&\leq Cn^{-3/2} \Big(\frac{r}{n}\Big(m_N-\frac{1}{\sqrt{2d}}\log x + 1\Big ) +4-t\Big)^2 \exp\Big( -\frac{(m_N-\frac{1}{\sqrt{2d}}\log x - t)^2 }{2(\log 2)(n-r)} \Big),
\end{align}
where  in the last inequality, we used the fact $r<n/2$ and  
\begin{align*}
b(n-r)-y &\leq     \frac{n-r}{n} \Big(m_N-\frac{1}{\sqrt{2d}}\log x + 1\Big ) - \Big( m_N-\frac{1}{\sqrt{2d}}\log x \Big)+\Lambda (r)  \overset{\eqref{390}}{=} -L_n(r)+2 \leq 2.
\end{align*}
 Note that we have additional factor $\log 2$ in the exponent in \eqref{378}, since the increment of MBRW has a variance $\log 2$.
For $t\leq \Lambda(r)$, RHS above is maximized at  $t=\Lambda(r)$ (it can be deduced by taking the derivative of RHS in $t$ and details will be deferred to the end of proof).
Hence,
\begin{align} \label{355}
 \max_{t\leq  \Lambda (r)} \mathbb{P}(F^2_{u,v}(t) )& \leq  C  n^{-3/2}  ( L_n(r) + 3)^2  \exp\Big(-\frac{  ((1-\frac{r}{n})(m_N-\frac{1}{\sqrt{2d}}\log x+1) + \gamma\log r-2 )^2}{2 (\log 2)(n-r)}\Big)\nonumber \\
  &\leq C   n^{-3/2}   (\log r+1)^2  \cdot 2^{-d\log_2 |u-v|^{(N)}}    e^{\frac{3}{2} \frac{\log n}{n}     (n-r) - \gamma \frac{\sqrt{2d} }{2}\log r}  x^{  \frac{n-r}{n} +  \frac{\gamma}{(\log 2) \sqrt{2d}} \frac{\log r}{n}}  .
\end{align}
Derivation of the second inequality follows similarly as in \eqref{384} and will be presented at the end of proof.
We analyze the exponents of $e$ and $x$ respectively.    The exponent of $e$ is bounded by
\begin{align*}
\frac{3}{2} \frac{\log n}{n}     (n-r) - \gamma \frac{\sqrt{2d} }{2}\log r\leq  \frac{3}{2}\log n - 0.5 \gamma \log r,
\end{align*}
and the exponent of $x$ is bounded by
\begin{align*} 
 \frac{ n-r}{n} +  \frac{\gamma}{(\log 2) \sqrt{2d}} \frac{\log r}{n} \leq 1  + \frac{C}{n}.
\end{align*} 
Hence, by  \eqref{379} and \eqref{355},  
\begin{align} \label{357}
\mathbb{P}(E_u\cap E_v) & \overset{\eqref{381}}{\leq}   C xN^{-d}\cdot  n^{-3/2}   (\log r+1)^2  \cdot 2^{-d\log_2 |u-v|^{(N)}}   n^{3/2} r^{ - 0.5 \gamma }  x^{1 + \frac{C}{n}} \nonumber \\
&\leq  C  x^2N^{-d}\cdot   r^{ - 0.5 \gamma }   (\log r +1)^2  \cdot 2^{-d\log_2 |u-v|^{(N)}},
\end{align}
where the last inequality holds for large enough $n$, for any fixed $x>0$.

In the case when $r(u,v)=0$, using the fact that the events $E_u$ and $E_v$ are independent (since  $\mathcal{B}_j(u)$ and $\mathcal{B}_j(v)$ are disjoint for all $0\leq j\leq n-1$ if $r(u,v)=0$),
\begin{align} \label{356}
\mathbb{P}(E_u\cap E_v) = \mathbb{P}(E_u)\mathbb{P}(E_v)\overset{\eqref{381}}{ \leq}C  x^2 N^{-{2d}}.
\end{align}
Hence, applying \eqref{357} and \eqref{356} to  \eqref{379},  
 \begin{align} \label{3790}
  \sum_{r=0}^{ \lfloor  n/2 \rfloor  } \mathbb{P}(E_u\cap E_v)\leq    \sum_{r=1}^{ \lfloor  n/2 \rfloor  } C  x^2  \cdot  r^{ - 0.5 \gamma } (\log r+1)^2  + Cx^2 \leq Cx^{2 }.
 \end{align}
Therefore,  the second moment bound \eqref{372} follows from \eqref{383} and  \eqref{3790}, which together with the first moment bound concludes the proof by the Paley-Zygmund inequality.

~
 
Finally, we verify that for $t\leq \Lambda(r)$,  RHS of \eqref{378} is maximized at  $t= \Lambda(r)$.  We do this by proving that its derivative in $t$ is positive for $t\leq \Lambda (r)$. Taking  derivative in $t$, it suffices to prove the positivity of
\begin{align*}
&  \frac{2( m_N-\frac{1}{\sqrt{2d}}\log x - t)}{2(\log 2 )(n-r)}\Big(\frac{r}{n}\Big(m_N-\frac{1}{\sqrt{2d}}\log x + 1\Big ) +4-t\Big)^2-2\Big(\frac{r}{n}\Big(m_N-\frac{1}{\sqrt{2d}}\log x + 1\Big ) +4-t\Big) .
\end{align*}
Since
\begin{align*}
 \frac{r}{n}\Big(m_N-\frac{1}{\sqrt{2d}}\log x + 1\Big ) +4-t \geq L_n(r)+3 \geq 3
\end{align*}
for $t\leq \Lambda(r)$, it reduces to prove the positivity of
\begin{align*}
  &  \frac{ \frac{n-r}{n} (m_N-\frac{1}{\sqrt{2d}}\log x+1 ) + L_n(r)-2 }{(\log 2 )(n-r)} \cdot 3-2 \geq 3   \Big( \sqrt{2d}-  \frac{c}{n} (\log n + \log x) -\frac{2}{(\log 2)(n-r)} \Big) -2
\end{align*}
($c>0$ is a constant). This quantity is positive for large enough $N$, since $\sqrt{2d}-  \frac{c}{n} (\log n + \log x) -\frac{2}{(\log 2)(n-r)} >1$ for any  $r<n/2$ and sufficiently large $n$.

~

We end the proof by verifying the fourth inequality  in \eqref{384}. Note that
\begin{align} \label{999}
m_N =   \sqrt{2d}\log N -\frac{3}{2\sqrt{2d}}\log \log N =  \sqrt{2d} ( \log 2 )n  -\frac{3}{2\sqrt{2d}} \log  n -a
\end{align}
for some constant $a>0$. Using this, we lower bound 
\begin{align*}
\frac{  1}{2 (\log 2)(n-r)} \Big (\Big(1-\frac{r}{n}\Big)\Big(m_N-\frac{1}{\sqrt{2d}}\log x+1 \Big) + \gamma\log (n-r) -1 \Big)^2 = \frac{  1}{2 (\log 2)(n-r)}  (A+B-1)^2,
\end{align*} 
where $A:=(1-\frac{r}{n})(m_N-\frac{1}{\sqrt{2d}}\log x+1)$ and $B:=  \gamma\log (n-r)$, by analyzing the diagonal and crossing term. We bound the diagonal term
\begin{align*}
 \frac{A^2}{2(\log 2)(n-r)}  & \overset{\eqref{999}}{=}   \frac{1}{2(\log 2)} \frac{n-r}{n^2} \Big(    \sqrt{2d} ( \log 2 )n  -\frac{3}{2\sqrt{2d}} \log  n -\frac{1}{\sqrt{2d}}\log x-a+1 \Big)^2 \\
 &\geq   \frac{1}{2(\log 2)} \frac{n-r}{n^2} (2d (\log 2)^2 n^2 - 3 (\log 2) n \log n - 2(\log 2) (\log x) n - 4\sqrt{2d} ( \log 2 ) an   ) \\
 & \geq   d (\log 2) ( n-r)  -\frac{3}{2} \frac{n-r}{n}\log n - \frac{n-r}{n}\log x - 2\sqrt{2d} a 
\end{align*}
and the cross term 
\begin{align*}
\frac{AB}{(\log 2)(n-r)}  \geq  \gamma \log (n-r)\cdot \Big  ( \sqrt{2d} -  \frac{3}{2(\log 2) \sqrt{2d}} \frac{\log n}{n} - \frac{1}{(\log 2) \sqrt{2d}} \frac{\log x}{n}  \Big)-1.
\end{align*}
In addition,
\begin{align*}
\frac{A+B}{(\log 2)(n-r)}  \overset{\eqref{999}}{\leq}   \frac{1}{(\log 2)(n-r)}  \Big( \frac{n-r}{n} (    \sqrt{2d} ( \log 2 )n+1 ) + \gamma \log (n-r)\Big) < 2 \sqrt{2d}.
\end{align*}
Combining these estimates together, using the fact
\begin{align*}
e^{ -d (\log 2) ( n-r) } \overset{\eqref{998}}{ = }2^{ -d \lceil  \log_2(|u-v|^{(N)}+1) \rceil   }  \leq   2^{ -d   \log_2  |u-v|^{(N)}    }  ,
\end{align*} we deduce \eqref{384}. The proof of \eqref{355} follows similarly as above, by noticing that the only difference is a quantity $B$ which is replaced with $\gamma \log r$.
\end{proof}

\section{Basics of random measures and their convergence} \label{appendix a}

In this section we record all the  crucial definitions and properties from the theory of   random measures and Poisson point processes that  have featured throughout the article.

\subsection{Random measures}
We begin by setting up the concept of random measures, specifying the underlying topology and state a crucial criteria for the relative compactness of a collection of  random measures. 
We first recall the notion of convergence of measures (and probability measures) defined on a topological space   $S$.
A sequence of measures   $(\mu_n)_{n\in \mathbb{N}}$ is said to converge to the measure  $\mu$ in \emph{vague} topology if
\begin{align*}
\lim_{n\ri} \langle \mu_n, f\rangle  =  \langle \mu, f\rangle   \quad  \forall \text{ continuous and compactly supported functions}   \ f:S\rightarrow  [0,\infty),
\end{align*} 
where we use the standard notation $\langle \mu, f\rangle =\int f d\mu.$
 In contrast, recall that a sequence of \emph{probability}  measures   $(\mu_n)_{n\in \mathbb{N}}$ weakly converges to the probability measure  $\mu$ if
\begin{align*}
\lim_{n\ri} \langle \mu_n, f\rangle  =  \langle \mu, f\rangle   \quad  \forall \text{ bounded and continuous  functions}   \ f:S\rightarrow  [0,\infty).
\end{align*} 
In particular, when the underlying space is compact, vague convergence of a sequence of probability measures $\mu_n$ is equivalent to their  weak convergence.

~

When the sequence of probability measures  $(\mu_n)_{n\in \mathbb{N}}$ weakly converges to $\mu$, one can upgrade this to almost sure convergence  on a common probability space. 
\begin{lemma}[Skorokhod's representation theorem]\label{represent}
Suppose that the  sequence of probability measures $(\mu_n)_{n\in \mathbb{N}}$  on a metric space $S$  weakly converges to some probability measure $\mu$ as $n\ri$. Suppose also that the support of  $\mu$ is separable. Then, there exist $S$-valued random variables  $X_n$ and $X$ defined on a common probability space  $(\Omega,\mathcal{F},\mathbb{P})$ such that the law of $X_n$ (resp. $X$) is $\mu_n$ (resp. $\mu$)  and  $X_n\rightarrow X$, $\mathbb{P}$-almost surely as $n\ri$.
\end{lemma}

Now, we introduce random measures. A \emph{random measure}  $\xi$ is  a (almost surely) locally finite transition kernel from the probability space $(\Omega, \mathcal{F},\mathbb{P})$ to the measurable space $(E,\mathcal{E})$.  Being a transition kernel means that the mapping $w \mapsto \xi (w,B)$ is measurable  for any $B\in \mathcal{E}$, and the mapping $B\mapsto \xi(w,B)$ is  a measure on $(E,\mathcal{E})$ for every $w\in \Omega$. 
Equivalently, a random measure is a random element from the probability space $(\Omega, \mathcal{F},\mathbb{P})$ to the collection of locally finite measures (on a measurable space $(E,\mathcal{E})$), equipped with the canonical $\sigma$-algebra $\mathbb{M}$   induced by the mappings $\mu \mapsto \mu(B)$ for all bounded and measurable sets $B\in \mathcal{E}$.  We say  that $\xi$ is a random Borel measure if $E$ is a topological space and $\mathcal{E}$ is the $\sigma$-algebra of Borel sets.

   Now, we arrive at the vague convergence of random measures, when the underlying topological space $S$ is a Polish space.  Let $\mathcal{M}_S$ be the collection of measures on $S$, equipped with the vague topology. A particularly convenient property is that  for a Polish space $S$, the space  $\mathcal{M}_S$ can be metrized to also form a Polish space (see \cite[Theorem 4.2]{random}). 
 This  enables us to apply the standard theory of weak convergence. 
   We say the sequence of  random measures   $(\mu_n)_{n\in \mathbb{N}}$ weakly converges to the random measure  $\mu$ in \emph{vague} topology if
   \begin{align*}
\lim_{n\ri} \mathbb{E} g(\mu_n)  =  \mathbb{E}g(\mu) \quad  \forall \text{ bounded and continuous functions}   \ g:\mathcal{M}_S \rightarrow  [0,\infty)
\end{align*}
(the continuity of $g$ is w.r.t. vague topology on $\mathcal{M}_S $). Once $S$ is a locally compact Polish space,
this turns out to be equivalent to
\begin{align*}
\lim_{n\ri} \langle \mu_n, f\rangle  =  \langle \mu, f\rangle   \quad \text{in law}, \quad   \forall \text{ continuous and compactly supported functions}   \ f:S\rightarrow  [0,\infty)
\end{align*}
(see \cite[Theorem 14.16]{modern} for details).

The following lemma relates  the tightness and  relative compactness of random Borel measures.

\begin{lemma} [Lemma 14.15 in \cite{modern}] \label{tight converge}
Let  $(\eta_n)_{n\in \mathbb{N}}$ be a sequence of random  Borel measures on a locally compact  Polish space $S$. Assume  that for any compact set $K$ in $S$, the collection of random variables  $\{\eta_n(K)\}_{n\in \mathbb{N}}$ is tight.  Then, there
is a subsequence  $(n_k)_{k\in \mathbb{N}}$  in $\mathbb{N}$ and a   random Borel measure $\eta$ on $S$  such that  $\eta_{n_k}\rightarrow \eta$ in law as $k\ri$ w.r.t. vague topology.
\end{lemma}

Equivalently,  we deduce the same conclusion if the collection of random variables  $\{\langle \eta_n,f \rangle\}_{n\in \mathbb{N}}$ is tight for any continuous and compactly supported function $f:S\rightarrow [0,\infty)$.
 
 ~
  
\subsection{Laplace functional}
In this subsection, we discuss some basic properties of the Laplace functional of random measures, a key tool to study their convergence. From now on, we assume that the topological space $S$ is a Polish space, equipped with the Borel $\sigma$-algebra. Let $\mu$ be a random measure on $S$. Then, the
 \emph{Laplace functional} of $\mu$   is defined as
\begin{align*}
L_\mu(f): =\mathbb{E} e^{-\langle \mu, f \rangle},
\end{align*}
where $f: S\rightarrow [0,\infty) $ is any measurable function.

 Laplace functional tested against continuous  and compactly supported  functions {$f:S\rightarrow [0,\infty)$} uniquely determines the law of random measures (see \cite[Corollary 2.3]{random}), i.e.
\begin{align} \label{900}
L_\mu(f) = L_\nu(f),\quad \forall \text{ continuous and compactly supported functions}  \ f:S\rightarrow [0,\infty)
\end{align}
implies that {$\mu \overset{\text{law}}{\sim} \nu$.}

~

In addition, the convergence of random measures  is equivalent to the convergence of the corresponding Laplace functionals \cite[Theorem 4.11]{random}, i.e. the followings are equivalent:
\begin{enumerate}
\item  $\mu_n\rightarrow \mu$ in law  (w.r.t. vague topology).
\item  $L_{\mu_n} (f) \rightarrow L_\mu (f)$ for any continuous and compactly supported function $f:S \rightarrow [0,\infty)$.
\end{enumerate}

\subsection{Poisson point process} 
In this subsection, we record some properties of Poisson point processes.  For a random measure $\mu$  defined on a Polish space $S$, let  $\textup{PPP}(\mu)$ be the  Poisson point process with  intensity measure $\mu$ (As already alluded to, such process are called Cox process).  Then, by \cite[Lemma 3.1]{random}, for any (deterministic)  measure $\xi$ on $S$,
\begin{align*}
\mathbb{E}(e^{-\langle \mu,f \rangle} \mid  \mu =\xi) = \exp\Big(-\int (1-e^{-f}) \xi(dx)\Big) ,
\end{align*}
where $f:S\rightarrow [0,\infty)$ is a measurable function.
Taking the expectation with respect to $\mu$, we deduce that the 
Laplace functional of  $\textup{PPP}(\mu)$ is 
 \begin{align} \label{902}
 \mathbb{E}\exp\Big(-\int (1-e^{-f}) \mu(dx)\Big) .
 \end{align}
 
 The following lemma says that the law of Poisson point processes uniquely determines the law of intensity measures.
 \begin{lemma} \label{poisson unique}
 Let $\mu$ and $\nu$ be random measures  on the Polish space $S$. If  $\textup{PPP}(\mu) \overset{\text{law}}{\sim} \textup{PPP}(\nu) $, then $\mu \overset{\text{law}}{\sim}  \nu$.
 \end{lemma}
\begin{proof}
For any measurable function $f:S \rightarrow [0,\infty)$,
\begin{align*}
\mathbb{E}\exp\Big(-\int (1-e^{-f}) \mu(dx)\Big) = \mathbb{E}\exp\Big(-\int (1-e^{-f}) \nu(dx)\Big).
\end{align*}
For any  continuous and compactly supported function $g:S \rightarrow [0,1)$, there exists a measurable function $f: S \rightarrow [0,\infty)$ such that $g= 1-e^{-f}$. Thus, we deduce that for any such  $g:S \rightarrow [0,1)$,
\begin{align} \label{901}
\mathbb{E} e^{-\langle \mu, g \rangle} = \mathbb{E} e^{-\langle \nu, g \rangle}.
\end{align}
Let $h:S \rightarrow [0,\infty)$ be any   continuous and compactly supported function. Since $h$ is bounded, by the above display, there exists $\e>0$ such that for any $t\in [0,\e]$,
\begin{align*}
\mathbb{E} e^{-\langle \mu, th \rangle} = \mathbb{E} e^{-\langle \nu, th \rangle}.
\end{align*} 
Since both sides are analytic in $t\in (0,\infty)$,  the  above relation
extends to all $t\geq 0$, and in particular for $t=1$. This implies that $L_\mu(h) = L_\nu(h)$ for any    continuous and compactly supported function $h$. By  \eqref{900}, we conclude that $\mu \overset{\text{law}}{\sim}  \nu$.
\end{proof}

{Next,  we state a lemma regarding the law of Poisson point processes decorated with independent i.i.d. sample  points.}

\begin{lemma} \label{poisson joint}
Let $S_1$ and $S_2$ be topological spaces.
Let $\mu_1$   be a Borel measure on $S_1$ and $\mu_2$ be a probability measure on $S_2$.
Suppose that $\{X_{n,i}\}_{n,i\in \mathbb{N}}$ is a collection of (random) points in $S_1$ such that 
\begin{align*}
 \lim_{n\ri}  \sum_{i=1}^\infty \delta_{X_{n,i}} = \textup{PPP} (\mu_1) \quad \textup{in law}
\end{align*}
with respect to the vague topology on the set of  measures on $S_1$. Assume also that   $\{Y_i\}_{i\in \mathbb{N}}$ is a collection of (random) points in $S_2$ i.i.d. distributed as $\mu_2$, independent of $\{X_{n,i}\}_{n,i\in \mathbb{N}}$. Then,
\begin{align*}
 \lim_{n\ri}  \sum_{i=1}^\infty \delta_{X_{n,i}} \otimes  \delta_{Y_{i}}  = \textup{PPP} (\mu_1 \otimes \mu_2)  \quad \textup{in law}
\end{align*}
with respect to the vague topology on the set of   measures on $S_1\times S_2$.
\end{lemma}

\begin{proof}
It suffices to show that for any continuous and compactly supported function $f:S_1\times S_2 \rightarrow [0, \infty)$,
\begin{align} \label{810}
 \lim_{n\ri}   \mathbb{E} e^{-   \sum_{i=1}^\infty f(X_{n,i}, Y_i)} = \exp \Big(-\int ( 1-e^{-f(x,y)}) (\mu_1 \otimes \mu_2)(dx \  dy)  \Big).
\end{align}
Define a function $\tilde{f}:S_1\rightarrow [0,\infty)$ by 
\begin{align} \label{811}
e^{-\tilde{f}(x)} := \mathbb{E}e^{-f(x,Y_i)} = \int e^{-f(x,y)} \mu_2 (dy).
\end{align}
Since $f$ is compactly supported and $\mu_2$ is a probability measure, $\tilde{f}$ is also compactly supported. {Also, $\tilde{f}$ is continuous since $f$ is a uniformly continuous function.}
Therefore, 
\begin{align*}
 \lim_{n\ri} \mathbb{E}  [ \mathbb{E} (  e^{-   \sum_{i=1}^\infty f(X_{n,i}, Y_i)} | \{X_{n,i} \}_{n,i\in \mathbb{N}}) ] &=   \lim_{n\ri} \mathbb{E} \Big[ \prod_{i=1}^\infty \mathbb{E}(e^{- f(X_{n,i}, Y_i)} | \{X_{n,i} \}_{n,i\in \mathbb{N}})  \Big] \\
&=  \lim_{n\ri}  \mathbb{E}    e^{-   \sum_{i=1}^\infty  \tilde{f}(X_{n,i})} \\
&=  \exp \Big(-\int ( 1-e^{- \tilde{f}(x) }) \mu_1  (dx) \Big) \\
& \overset{\eqref{811}}{ = }\exp \Big(-\int ( 1-e^{-f(x,y)}) (\mu_1 \otimes \mu_2)(dx \  dy)  \Big),
\end{align*}
where in the last equality, we used Fubini's theorem and the fact that $\mu_2$ is a probability measure. This concludes the proof of \eqref{810}.
\end{proof}

By considering the special case when $\{X_{n,i}\}_{n,i\in \mathbb{N}}$ enumerates the  sample points in $\textup{PPP} (\mu_1)$ for each $n$, we obtain the following corollary.
\begin{corollary}  \label{904}
Let $S_1$ and $S_2$ be topological spaces.
Let $\mu_1$   be a Borel measure on $S_1$ and $\mu_2$ be a probability measure on $S_2$. Suppose that  $\{X_i\}_{i\in \mathbb{N}}$ is the enumeration of sample points in $\textup{PPP} (\mu_1)$ and $\{Y_i\}_{i\in \mathbb{N}}$ is i.i.d. distributed as $\mu_2$, independent of $\{X_{i}\}_{i\in \mathbb{N}}$. Then,  $\{(X_i,Y_i)\}_{i\in \mathbb{N}}$ has the same law as the   enumeration of sample points in $\textup{PPP} (\mu_1 \otimes \mu_2)$.
\end{corollary} 
Since  $\{Y_i\}_{i\in \mathbb{N}}$ is i.i.d, no matter how we enumerate the sample points in $\textup{PPP} (\mu_1)$, we obtain the identical law for the collection of points $\{(X_i,Y_i)\}_{i\in \mathbb{N}}$.

~

{The final two statements concern the point process featuring in the proof of Theorem \ref{theorem 1.1}.}
For a constant $\gamma \in (0,1)$, let $\{q_i\}_{i\in \mathbb{N}}$ be the enumeration of the sample points of the Poisson point process on $[0,\infty)$ with the intensity measure $x^{-1-\gamma}dx$, in a decreasing order, i.e. $q_1>q_2>\cdots$ (Observe that the intensity measure is integrable on $[\e,\infty)$ for any positive $\e$ and consequently, almost surely, the  point process has only finitely many points in $[\e, \infty)$)).
 We show that with high probability, the sum of the largest $q_i$s  takes a majority of the total sum  of $q_i$s.
\begin{lemma} \label{ppp}
For any $\iota,\e>0$, there exists $K = K(\iota,\e)\in \mathbb{N}$ such that
\begin{align*}
\mathbb{P}\Big(\sum_{i=1}^K q_i  > (1-\iota) \sum_{i\in \mathbb{N}} q_i\Big)  \geq 1-\e.
\end{align*}
\end{lemma}
\begin{proof}
Note that for $a>0$,  the probability  of the existence of a sample point in   $[a,\infty)$ is $1-\exp(-\int_a^\infty x^{-1-\gamma}dx)$. Since  $\int_a^\infty x^{-1-\gamma}dx =1/\gamma a^\gamma \rightarrow \infty$ as $a\downarrow 0$, there exists $\delta = \delta(\e)>0$ such that the  probability  of the existence of a sample point in $[\delta,\infty)$ is at least $1-\e$. Setting $S:= \sum_{i\in \mathbb{N}} q_i$, this implies that 
\begin{align} \label{801}
\mathbb{P}(S < \delta)  \leq \e.
\end{align}
Take   a small enough  $\kappa = \kappa(\iota,\e)>0$ such that
\begin{align} \label{804}
\frac{2}{\e}\frac{\kappa^\gamma}{1-\gamma} < \iota \delta.
\end{align}
 Note that
\begin{align} \label{806}
\mathbb{E}\sum_{q_i < \kappa } q_i  = \int_0^\kappa x \cdot x^{-1-\gamma} dx = \frac{\kappa^\gamma}{1-\gamma}.
\end{align}
{To see this, first partition $[0,\kappa]$ into infinitely many intervals whose lengths go to zero. Then, the sum of $q_i$s in each interval is bounded from below and above by a Poisson random variable, multiplied by  the left and right boundary value of the interval.  Taking expectations provides a Riemann sum type approximation to the integral and then sending the size of the intervals to zero, we obtain \eqref{806}.}
In particular, this   implies
 \begin{align} \label{802}
 \mathbb{P}\Big(\sum_{q_i < \kappa } q_i> \frac{1}{\e}\frac{\kappa^\gamma}{1-\gamma} \Big)  \leq \e.
 \end{align}
  Therefore, by \eqref{801}, \eqref{804} and \eqref{802},
 \begin{align} \label{805}
 \mathbb{P} \Big( \iota S < 2  \sum_{q_i < \kappa } q_i   \Big) \leq 2\e. 
 \end{align}
 Next, since the number of sample points in $[\kappa,\infty)$, denoted by $N(\kappa)$, is the Poisson random variable with intensity $\int_\kappa^\infty x^{-1-\gamma}dx = 1/\gamma \kappa^\gamma $, there exists $K = K(\iota,\e)>0$ such that
 \begin{align} \label{803}
 \mathbb{P}(N(\kappa) \geq  K ) \leq \e.
 \end{align} 
 By   \eqref{805} and \eqref{803},
 \begin{align*}
 \mathbb{P}\Big( \iota S < 2\Big(S - \sum_{i=1}^K q_i \Big) \Big) \leq 3\e.
 \end{align*}
 Arranging this, 
 \begin{align*}
\mathbb{P}\Big(\sum_{i=1}^K q_i  > (1-\iota/2) \sum_{i\in \mathbb{N}} q_i\Big)  \geq 1-3\e.
\end{align*}
 Since $\iota,\e>0$ are arbitrary, we conclude the proof.
 
\end{proof}

Finally, with $q_i$ as before, and a probability measure $\mu$ on some topological space $S$, let $\{X_i\}_{i\in \mathbb{N}}$ be i.i.d. distributed as $\mu$,  independent of $\{q_i\}_{i\in \mathbb{N}}$. Then, for $c>0$, define the random measure
\begin{align*}
\xi:=c \sum_{i\in \mathbb{N}} q_i \delta_{X_i}.
\end{align*}

\begin{lemma} \label{poisson}
The Laplace functional of $\xi$ is given by
\begin{align*}
L_\xi(f) = \mathbb{E} \exp \Big(    -c^{\gamma } \int (1-e^{-f(x)s}  ) \mu( dx) \otimes  s^{-1-\gamma} ds\Big)
\end{align*}
for any  measurable $f:S\rightarrow [0,\infty)$.
\end{lemma}
\begin{proof}
By Corollary \ref{904},  $\{(X_i,q_i)\}_{i\in \mathbb{N}}$ has the same law as the enumeration of sample points in  $\text{PPP} ( \mu(dx) \otimes t^{-1-\gamma}dt) $. Thus, for any measurable function $f:S\rightarrow [0,\infty)$,  using a test function $F:S\times [0,\infty) \rightarrow [0,\infty)$ defined as $F(x,t):=c f(x)t$,
\begin{align*}
L_\xi(f) =\mathbb{E}e^{-c \sum_{i\in \mathbb{N}}  q_if(X_i)} = \mathbb{E}e^{-  \sum_{i\in \mathbb{N}}   F(X_i,q_i)} & 
=\mathbb{E}\exp\Big(- \int (1-e^{-c f(x)t} )   \mu (dx) \otimes t^{-1-\gamma}dt \Big) \\
&=\mathbb{E}\exp\Big(-c^{\gamma }\int (1-e^{- f(x)s} )   \mu (dx) \otimes s^{-1-\gamma}ds\Big ).
\end{align*}
\end{proof}

\section{Appendix} \label{appendix b}
In the  Appendix,
we provide the outstanding proofs of  Lemmas \ref{lemma 3.11} and \ref{tensor}.
\begin{proof}[Proof of Lemma \ref{lemma 3.11}]$\empty$\\
\textbf{Case 1. $u\neq v$ belong to the same $D_i$.}  Using the fact  $|u-v|^{(N)} = |u-v|$ (since $|u-v|_\infty  \leq KN' < N/2$), by Lemma \ref{lemma 2.3},
\begin{align*}
\cov   (\theta_{N,u}, \theta_{N,v}) \leq \log \frac{N}{|u-v|}+c_1.
\end{align*}
Also, by the same argument as in \eqref{351}, 
\begin{align*}
\cov  (  \tilde{\theta}_{N,u}^{N',K},  \tilde{\theta}_{N,v}^{N',K})  \geq  \log \frac{N}{|u-v|} +\log K -c_1 .
\end{align*} Thus, we obtain  \eqref{285} for   large constant $K$. 

\textbf{Case 2. $u\in D_i$ and $v\in D_j$ with $i\neq j$.}  Since $|u-v|^{(N)} \geq \min \{  KN', \frac{2}{10}N  \} = KN'$ (recall that $N\geq 6KN'$ and $u,v\in V_N^{1/10}$), by Lemma \ref{lemma 2.3},
\begin{align*}
 \cov   (\theta_{N,u}, \theta_{N,v}) \leq \log \frac{N}{KN'} + c_1 = \log N - \log N' - \log K  + c_1  . 
\end{align*} 
Also, by the same argument as in \eqref{352},  using \eqref{353},
\begin{align*}
\cov  (  \tilde{\theta}_{N,u}^{N',K},  \tilde{\theta}_{N,v}^{N',K}) = \log N- \log N' .
\end{align*}
Thus, \eqref{285} holds for any large constant $K$. 
\end{proof}

 \begin{proof}[Proof of Lemma \ref{tensor}]
Let $\e \in (0,1)$ and  $A = (a_{ij})_{1\leq i,j\leq n}$ be a  covariance matrix of $X$, divided by $\sigma^2$. Then, there is a coupling on $(X,Y)$ such that $X=AY$. Since $|a_{ii}-1|<\delta $ and $|a_{ij}| < \delta$ for $i\neq j$ under the condition \eqref{gaussian condition}, 
\begin{align} \label{980}
|X_i -Y_i| = \Big\vert \sum_{j=1}^n a_{ij}  Y_j - Y_i \Big\vert  \leq |a_{ii} -1| |Y_i| + \sum_{j\neq i} |a_{ij} | |Y_j| \leq \delta \norm{Y}_1.
\end{align}
 Due to the uniform continuity of $g$, there exists  $ \kappa>0$ such that
\begin{align}\label{981}
|z-z'| + |w-w'| \leq \kappa \Rightarrow |g(z,w) -g(z',w')| \leq \e/2m.
\end{align}
Let $K>0$ be a large number which will be chosen later. If $\delta \leq  \kappa / K$, then for  $\norm{Y}_1 \leq K $, we have  $|X_i -Y_i| \overset{\eqref{980}}{  \leq}  \delta K \leq  \kappa  $ for $1\leq i\leq n$, which   implies
\begin{align*}
\Big\vert \sum_{i=1}^n g(z_i+X_i,w_i) -\sum_{i=1}^n g(z_i+Y_i,w_i) \Big\vert \leq  \sum_{i=1}^n | g(z_i+X_i,w_i) -g(z_i+Y_i,w_i)|  \overset{\eqref{981}}{  \leq} n\cdot \e/2m \leq \e/2.
\end{align*}
In addition, $\mathbb{P}(\norm{Y}_1 > K)  \leq   n \mathbb{P}(|Y_1| >K/n) \leq n\cdot Cm/K \cdot e^{-K^2/2m^2} \leq    C e^{-K^2/2m^2} /K  $.  Hence,  using the fact that  $n\leq m$ and $g$ is bounded from below and above,
\begin{align*}
\mathbb{E }\exp \Big(  \sum_{i=1}^n g(z_i+X_i,w_i)  \Big) &=  \mathbb{E} \Big[ \exp \Big(  \sum_{i=1}^n g(z_i+X_i,w_i)  \Big)  \1_{\norm{Y}_1\leq K} \Big] + \iota_1 \\
&=  \mathbb{E} \Big[  e^{\zeta}\exp \Big(  \sum_{i=1}^n g(z_i+Y_i,w_i)  \Big)  \1_{\norm{Y}_1\leq K} \Big]  +\iota_1 \\
&=   \mathbb{E} \Big[  e^{\zeta}\exp \Big(  \sum_{i=1}^n g(z_i+Y_i,w_i)  \Big) \Big] + \iota_2 .
\end{align*}
where $\zeta$ is a random quantity with $|\zeta|\leq \e/2 < 1/2$ and $\iota_i$ ($i=1,2$) satisfies $|\iota_i| \leq  C e^{-K^2/2m^2} /K$.  Thus, using the boundedness of $g$ again, the above quantity can be written as $e^\gamma  \cdot \mathbb{E} \exp (  \sum_{i=1}^n g(z_i+Y_i,w_i)  )$ with  $|\gamma|<\e$ for    sufficiently large  $K$.  This concludes the proof.
\end{proof}

 \bibliography{bib2}
\bibliographystyle{plain}

\end{document}